\documentclass{article}

\usepackage{amsmath}
\usepackage{amsthm}
\usepackage{amsfonts}
\usepackage{mathrsfs}
\usepackage{amssymb}
\usepackage{bbm}
\usepackage{tikz-cd}
\usepackage{tikz}
\usetikzlibrary{cd,matrix,arrows,decorations.pathmorphing}

\usepackage{enumerate}

\usepackage{hyperref}

\theoremstyle{definition}
\newtheorem{theorem}{Theorem}[section]
\newtheorem{definition}[theorem]{Definition}
\newtheorem{lemma}[theorem]{Lemma}
\newtheorem{proposition}[theorem]{Proposition}
\newtheorem{corollary}[theorem]{Corollary}
\newtheorem{remark}[theorem]{Remark}

\def\calA{\mathcal{A}}

\def\calC{\mathcal{C}}
\def\calD{\mathcal{D}}

\def\calF{\mathcal{F}}
\def\calG{\mathcal{G}}
\def\calH{\mathcal{H}}

\def\calM{\mathcal{M}}

\def\calO{\mathcal{O}}

\def\calS{\mathcal{S}}

\def\calZ{\mathcal{Z}}
\def\cG{\mathcal{G}}

\def\frb{\mathfrak{b}}

\def\frg{\mathfrak{g}}
\def\frh{\mathfrak{h}}

\def\frm{\mathfrak{m}}
\def\frn{\mathfrak{n}}

\def\frz{\mathfrak{z}}

\def\ffrb{\mathfrak{B}}

\def\sC{{\mathscr{C}}}

\def\sF{{\mathscr{F}}}
\def\sG{{\mathscr{G}}}
\def\sH{{\mathscr{H}}}
\def\sI{{\mathscr{I}}}

\def\bbA{\mathbb{A}}
\def\bbB{\mathbb{B}}
\def\bbC{\mathbb{C}}
\def\bbD{\mathbb{D}}

\def\bbF{\mathbb{F}}
\def\bbG{\mathbb{G}}
\def\bbH{\mathbb{H}}

\def\bbN{\mathbb{N}}

\def\bbP{\mathbb{P}}
\def\bbQ{\mathbb{Q}}
\def\bbR{\mathbb{R}}
\def\bbS{\mathbb{S}}
\def\bbT{\mathbb{T}}

\def\bbZ{\mathbb{Z}}
\def\bZ{{\mathbb Z}}

\DeclareFontFamily{U}{wncy}{}
\DeclareFontShape{U}{wncy}{m}{n}{<->wncyr10}{}
\DeclareSymbolFont{mcy}{U}{wncy}{m}{n}
\DeclareMathSymbol{\Sha}{\mathord}{mcy}{"58} 

\def\bs{\backslash}

\DeclareMathOperator{\ox}{\otimes}

\def\comp{\circ}
\def\inj{\hookrightarrow}
\def\surj{\twoheadrightarrow}
\def\isom{\cong}
\def\aisom{\xrightarrow{\sim}}

\DeclareMathOperator{\coker}{coker}
\DeclareMathOperator{\colim}{colim}

\DeclareMathOperator{\op}{op}
\DeclareMathOperator{\pr}{pr}

\newcommand{\dlim}{\varinjlim}
\newcommand{\ilim}{\varprojlim}

\DeclareMathOperator{\Fil}{Fil}

\DeclareMathOperator{\Hom}{Hom}

\DeclareMathOperator{\Ext}{Ext}

\DeclareMathOperator{\Tor}{Tor}

\DeclareMathOperator{\Ind}{Ind}

\DeclareMathOperator{\Gal}{Gal}

\DeclareMathOperator{\GL}{GL}

\DeclareMathOperator{\End}{End}

\DeclareMathOperator{\Mat}{Mat}

\DeclareMathOperator{\Lie}{Lie}

\DeclareMathOperator{\ad}{ad}

\DeclareMathOperator{\gl}{\mathfrak{gl}}

\DeclareMathOperator{\cInd}{\text{c-}\Ind}

\def\bb1{\mathbbm{1}}

\DeclareMathOperator{\DR}{DR}

\DeclareMathOperator{\gr}{gr}
\DeclareMathOperator{\ur}{ur}

\DeclareMathOperator{\cont}{cont}

\DeclareMathOperator{\HT}{HT}
\DeclareMathOperator{\dR}{dR}

\DeclareMathOperator{\Cont}{Cont}

\DeclareMathOperator{\Spf}{Spf}

\def\LT{\mathrm{LT}}

\def\et{\text{\'et}}

\def\Cont{\mathrm{Cont}}

\def\an{\mathrm{an}}
\def\lan{\mathrm{la}}
\def\sm{\mathrm{sm}}

\def\lalg{\mathrm{lalg}}

\def\hat{\widehat}
\def\bar{\overline}
\def\tilde{\widetilde}

\def\St{\mathrm{St}}

\def\-{\text{-}}

\def\pro{\mathrm{pro}}

\def\-{\text{-}}

\def\fl{\mathscr{F}\ell}

\def\Sen{\mathrm{Sen}}

\def\Ab{\mathrm{Ab}}

\def\GM{\mathrm{GM}}

\def\Dr{\mathrm{Dr}}

\DeclareMathOperator{\proet}{\text{pro\'et}}

\DeclareMathOperator{\solid}{Solid}
\DeclareMathOperator{\cond}{Cond}

\DeclareMathOperator{\supp}{supp}
\DeclareMathOperator{\open}{Open}
\DeclareMathOperator{\Func}{Func}
\DeclareMathOperator{\ExtDisc}{ExtDisc}

\def\ov{\stackrel}

\usepackage{geometry}
\usepackage{indentfirst}
\usepackage{changes}
\usepackage{ulem}
\title{Locally analytic vectors in the completed cohomology of quaternionic Shimura curves}
\author{Zhenghui Li\footnote{IMJ-PRG, Sorbonne Universit\'e, 4 place Jussieu, 75005 Paris, France}\\
\and 
Benchao Su\footnote{School of Mathematical Sciences, Peking University, No.5 Yiheyuan Road, Haidian District,Beijing 100871, P.R.China}\\
\and
Zhixiang Wu\footnote{School of Mathematical Sciences, University of Science and Technology of China, 96 Jinzhai Road, 230026 Hefei, China}}
\date{}
\begin{document}

\maketitle

\begin{abstract}
We use the methods introduced by Lue Pan to study the locally analytic vectors of the completed cohomology of Shimura curves associated to an indefinite quaternion algebra $D$ which is ramified at a prime number $p$. Let $D_p^{\times}$ be the group of units of $D$ at $p$. Using $p$-adic uniformization of the quaternionic Shimura curves, we compute the Hecke eigenspace of the completed cohomology with the Hecke eigenvalues associated to a classical automorphic form on another quaternion algebra $\bar D$ (switching invariants of $D$ at $p,\infty$). We present this locally analytic $D_p^\times$-representation using the de Rham complex of the Lubin--Tate tower of dimension $1$. This is analogous to the Breuil--Strauch conjecture for the group $\GL_2(\bbQ_p)$. We show that the locally analytic $D_p^{\times}$-representation does not detect the Hodge filtration of the local de Rham Galois representation at $p$ in the crystalline case, and also give applications for the locally analytic Jacquet--Langlands correspondence for $\GL_2(\bbQ_p)$ and $D_p^\times$.
\end{abstract}
\setcounter{tocdepth}{2}

\tableofcontents

\section{Introduction}
Let $p$ be a prime number. Let $E$ be a sufficiently large finite extension of $\bbQ_p$. The $p$-adic local Langlands correspondence for $\GL_2(\bbQ_p)$ constructs a correspondence between certain $2$-dimensional $p$-adic representations of $\Gal(\overline{\bbQ_p}/\bbQ_p)$ and certain $p$-adic representations of $\GL_2(\bbQ_p)$ over $E$ \cite{breuil2010emerging,colmez2010representations,MR3272011}. There have been recent developments of the interactions between $p$-adic or mod-$p$ representations of $\GL_2(\bbQ_p)$ and $D_p^\times$, with $D_p$ a non-split quaternion algebra over $\bbQ_p$, known as the $p$-adic Jacquet--Langlands correspondence. For a non-exhaustive list of relevant works, see, e.g., \cite{Scholze2018LubinTate, ludwig2017quotient, pavskunas2022some,hansen2022padicsheavesclassifyingstacks,HuWang,hu2024some,dospinescu2024jacquetlanglandsfunctorpadiclocally}. While the $\GL_2(\bbQ_p)$ side is more or less well understood, the representations of $D_p^\times$ arising from the $p$-adic correspondences remain largely mysterious.

This paper gives a first complete description of the internal structure of the locally analytic representations of $D_p^{\times}$ associated to de Rham representations of $\Gal(\overline{\bbQ_p}/\bbQ_p)$ with global origins, in terms of the (compactly supported) coherent cohomology groups of the $1$-dimensional Lubin--Tate tower. This establishes the quaternionic analogue of the Breuil--Strauch conjecture firstly proved by Dospinescu--Le Bras \cite{DLB17}, which describes  $\GL_2(\bbQ_p)$-representations using the de Rham complexes of the $1$-dimensional Drinfeld tower. Our description for the $D_p^{\times}$-representations exhibits several features of the $p$-adic Jacquet--Langlands correspondance.

Our work is based on the study of the locally analytic vectors in the completed cohomology of quaternionic Shimura curves using the methods introduced by Lue Pan \cite{Pan22,PanII}. Let $D$ be an indefinite quaternion algebra over $\bbQ$ which is ramified at $p$, and let $G$ be the group of units of $D$, considered as an algebraic group over $\bbQ$. For $K\subset G(\bbA_f)$ a neat open compact subgroup, the compact Riemann surface $S_{K}(\bbC):=G(\bbQ)\bs ((\bbC\bs\bbR)\times G(\bbA_f)/K)$
admits a canonical model $S_K$, which is a proper smooth algebraic curve over $\bbQ$. Fix a neat open compact subgroup $K^p\subset G(\bbA_f^p)$, where $\bbA_f^p$ denotes the ring of finite adeles away from $p$. As $K_p$ varies in open compact subgroups of $D_p^\times=G(\bbQ_p)$, we consider Emerton's completed cohomology group \cite{Emerton2006interpolation}
\begin{align*}
    \tilde{H}^1(K^p,E):=E\ox_{\bbZ_p}(\ilim_n\dlim_{K_p} H^1(S_{K^pK_p}(\bbC),\bbZ/p^n\bbZ)),
\end{align*}
which is a Banach space over $E$, carrying a continuous $E$-linear action of $\Gal(\bar\bbQ/\bbQ)\times D_p^\times$. It is expected, as the local-global compatibility of the $p$-adic Langlands correspondence in the $\GL_2(\bbQ_p)$-case \cite{emerton2011local,PanII}, that if $\rho:\Gal(\bar\bbQ/\bbQ)\rightarrow \GL_2(E)$ is an absolutely irreducible Galois representation such that the $D_p^{\times}$-representation 
\begin{align*}
    \check{\Pi}(\rho):=\Hom_{\Gal(\overline{\bbQ}/\bbQ)}(\rho,\tilde{H}^1(K^p,E))
\end{align*}
is non-zero, then its subspace of locally analytic vectors $\check{\Pi}(\rho)^{\lan}$ should be a finite copy of the conjectural locally analytic $D_p^{\times}$-representation $\tau(\rho_p)$ attached to $\rho_p=\rho|_{\Gal(\overline{\bbQ_p}/\bbQ_p)}$ in the $p$-adic Langlands program.

On the other hand, we can give a direct construction of a locally analytic $D_p^{\times}$-representation $\tau(\rho_p)$ for a $2$-dimensional de Rham representation $\rho_p$ of $\Gal(\overline{\bbQ_p}/\bbQ_p)$ using the de Rham complexes of Lubin--Tate spaces. Let $\{\calM_{\LT,n}\}_{n\ge 0}$ be the Lubin--Tate tower arising from a $1$-dimensional formal $p$-divisible group over $\bar\bbF_p$ of height $2$. They are $1$-dimensional rigid analytic Stein spaces over $\bbC_p:=\hat{\overline{\bbQ_p}}$, and the whole tower carries an action of $\GL_2(\bbQ_p)\times D_p^\times$ \cite{GH94,RZ96}. For each fixed level $n\ge 0$, the compactly supported cohomology groups $H^1_c(\calM_{\LT,n},\calO_{\calM_{\LT,n}})$ and $H^1_c(\calM_{\LT,n},\Omega^1_{\calM_{\LT,n}})$ (with $\Omega^1_{\calM_{\LT,n}}$ the sheaf of differentials) carry locally analytic $D_p^\times$-actions in the sense of Schneider--Teitelbaum \cite{MR1887640} (see also the work \cite{sheth2020locally}). The colimits $\dlim_n H^1_c(\calM_{\LT,n},\calO_{\calM_{\LT,n}})$ and $\dlim_n H^1_c(\calM_{\LT,n},\Omega_{\calM_{\LT,n}}^1)$ are also smooth representations of $\GL_2(\bbQ_p)$, thus are naturally modules over the smooth Hecke algebra $\calH(G)$ (the convolution algebra of compactly supported smooth functions on $G$ after fixing a Haar measure on $G$) with $G=\GL_2(\bbQ_p)$. 

We fix $\rho_p:\Gal(\overline{\bbQ_p}/\bbQ_p)\rightarrow\GL_2(E)$ to be de Rham with Hodge--Tate weight $0,1$. By Fontaine's theory, we can associate to $\rho_p$ a $2$-dimensional filtered $E$-vector space $D_{\dR}(\rho_p)$, with the filtration given by the Hodge filtration, and a Weil--Deligne representation $r_p$ over $E$. We can then associate to $r_p$ a smooth representation $\pi_p$ of $\GL_2(\bbQ_p)$ and a smooth representation $\tau_p$ of $D_p^\times$ (with coefficients in $\bbC_p$) by the local Langlands correspondence\footnote{We use the version in \cite[\S 4.2]{emerton2011local} and we only consider $\rho_p$ such that $\pi_p$ is irreducible and infinite-dimensional.} and the Jacquet--Langlands correspondence. Here, we set $\tau_p=0$ if $\pi_p$ is an irreducible principal series representation. From $\pi_p$, we can define the following locally analytic representations of $D_p^{\times}$ using the Lubin--Tate tower:
\begin{align*}
    \widetilde{\tau}:=(\dlim_n H^1_c(\calM_{\LT,n},\calO_{\calM_{\LT,n}}))\ox_{\calH(G)}\pi_p,\\
    \tau_c:=(\dlim_n H^1_c(\calM_{\LT,n},\Omega_{\calM_{\LT,n}}^1))\ox_{\calH(G)}\pi_p.
\end{align*}
We will show that the de Rham complexes of the Lubin--Tate spaces and the classical Jacquet--Langlands correspondence induce a short exact sequence of locally analytic $D_p^\times$-representations (see \S\ref{subsec:repofDptimes}):
\begin{align*}\label{equationintroductionexactseuqnece}
    0\to\tau_p^{\oplus 2} \rightarrow \widetilde{\tau}\to \tau_c\to 0
\end{align*}
which depends only on $\pi_p$ and equivalently $r_p$.

We define the locally analytic $D_p^{\times}$-representation $\tau(\rho_p)$ associated to $\rho_p$ as the quotient representation
\[\tau(\rho_p):=\widetilde{\tau}/i_{\rho_p}(\tau_p),\]
where the map $i_{\rho_p}:\tau_p\hookrightarrow \tau_p^{\oplus 2}\subset \widetilde{\tau}$ is induced by a canonical isomorphism $\tau_p^{\oplus 2}\simeq \tau_p\otimes D_{\dR}(\rho_p)$ and the $1$-dimensional subspace of $D_{\dR}(\rho_p)$ given by the Hodge filtration. Then $\tau(\rho_p)$ sits in an exact sequence
\begin{align}\label{equationintroextension}
    0\rightarrow \tau_p\rightarrow \tau(\rho_p)\rightarrow \tau_c\rightarrow 0.
\end{align}
The following main theorem establishes the local-global compatibility for the $D_p^{\times}$-representation $\tau(\rho_p)$.
\begin{theorem}[Theorem \ref{thm:classicality}, Corollary \ref{corollarydescription}]\label{thm:mainintro}
Let $\rho_p$ be a $2$-dimensional de Rham representation of $\Gal(\overline{\bbQ_p}/\bbQ_p)$ over $E$ with Hodge--Tate weight $0,1$. Suppose that there exists an absolutely irreducible Galois representation $\rho$ of $\Gal(\bar\bbQ/\bbQ)$ over $E$ such that $\rho|_{\Gal(\overline{\bbQ_p}/\bbQ_p)}\simeq \rho_p$. 
\begin{enumerate}[(1)]
    \item Assume that $\rho$ appears in $\widetilde{H}^1(K^p,E)$, i.e. $\check{\Pi}(\rho)=\Hom_{\Gal(\overline{\bbQ}/\bbQ)}(\rho,\tilde{H}^1(K^p,E))\neq 0$. Then there exists some $m\ge 1$, such that there is a $D_p^\times$-equivariant and topological isomorphism\footnote{Here we fix an embedding $E\to \bbC_p$.}
    \begin{align*}
        \check{\Pi}(\rho)^{\lan}\widehat{\ox}_E\bbC_p\isom \tau(\rho_p)^{\oplus m}.
    \end{align*}
    Moreover, the representation $\tau(\rho_p)$ and the above isomorphism can be defined over $E$. And with the $E$-structure $\tau(\rho_p)$ is an infinite-dimensional admissible locally analytic representation of $D_p^{\times}$ in the sense of Schneider--Teitelbaum \cite{ST03}.
    \item The representation $\rho$ appears in $\widetilde{H}^1(K^p,E)$ if and only if $\rho$ is the Galois representation associated to some classical automorphic form of the tame level $K^p$ on the group $\overline{G}$, the algebraic group of the units of the definite quaternion algebra $\overline{D}$ over $\bbQ$ obtained from $D$ by switching the Hasse invariants at $p,\infty$. 
\end{enumerate}
\end{theorem}
\begin{remark}
\begin{enumerate}[(i)]
    \item Theorem \ref{thm:mainintro} also holds when $\rho_p$ is de Rham of arbitrary regular Hodge--Tate weight. Moreover, following the methods of \cite{QS24}, similar results for locally $L$-analytic representations of $D_L^\times$ also hold, with $L$ a finite extension of $\bbQ_p$, $D_L$ a non-split quaternion algebra over $L$, by using some Shimura curves with $D_L^\times$ being a local factor of the group at $p$.
    \item The second part of Theorem \ref{thm:mainintro} proves the classicality for the de Rham Galois representations that appear in $\widetilde{H}^1(K^p,E)$. This is essentially known by the proof of the Fontaine-Mazur conjecture for $\GL_{2/\bbQ}$ \cite{emerton2011local,PanII} and the global Jacquet--Langlands correspondence. The more interesting part is the non-vanishing of $\check{\Pi}(\rho)$ when $\tau_p=0$. In this case there is no classical automorphic form on $G$ associated to $\rho$ while $\tau(\rho_p)\neq 0$. Such non-vanishing result in the $p$-adic Jacquet--Langlands correspondence was previously known by Pa{\v{s}}k{\=u}nas \cite[Theorem 1.4]{pavskunas2022some} (see also \cite[\S 1.3]{Howe2022SpectralPadicJacquetLanglands}) using very different methods.
\end{enumerate}
\end{remark}
The structure of $\tau(\rho_p)$ illustrates similarities and differences between the groups $D_p^{\times}$ and $\GL_2(\bbQ_p)$ under the $p$-adic local Langlands correspondences:
\begin{itemize}
    \item When the associated Weil--Deligne representation $r_p$ is absolutely irreducible, the presentation of $\tau(\rho_p)$ is similar to that of the locally analytic representation $\pi(\rho_p)$ of $\GL_2(\bbQ_p)$ attached to $\rho_p$ via the $p$-adic local Langlands correspondence. In that case, by the Breuil--Strauch conjecture proved by Dospinescu-Le Bras \cite{DLB17} and reproved by Lue Pan \cite[Theorem 7.3.2]{PanII}, $\pi(\rho_p)$ admits a presentation
    \begin{align}\label{equationintrorepdrinfeld}
        \pi(\rho_p)\isom (\dlim_nH^1_c(\calM_{\Dr,n},\calO_{\calM_{\Dr,n}})\ox_{\calH(\check G)}\tau_p)/i_{\rho_p}(\pi_p)
    \end{align}
    where $\{\calM_{\Dr,n}\}_{n\ge 0}$ is the Drinfeld tower \cite{Dr76}, with $\check{G}=D_p^\times$, and the map $i_{\rho_p}$ is determined by the Hodge filtration of $D_{\dR}(\rho_p)$. However, if the Weil--Deligne representation $r_p$ is reducible, for example if $\rho_p$ is crystalline, then the $\GL_2(\bbQ_p)$-representation $\pi(\rho_p)$ can not be described using the Drinfeld tower in a similar way. 
    \item In contrast, the $D_p^\times$-representation $\tau(\rho_p)$ is always (only) related to the Lubin--Tate tower, regardless of whether the Weil--Deligne representation $r_p$ is irreducible or not. Moreover, unlike the classical non-abelian Lubin-Tate theory and the Jacquet--Langlands correspondence, e.g. \cite{Carayol1990_nonabelianLT,dat2007_LT_nonabelian}, we will show directly in Theorem \ref{thmintornonvanishing} that $(\dlim_n H^1_c(\calM_{\LT,n},\Omega_{\calM_{\LT,n}}^1))\ox_{\calH(G)}\pi_p\neq 0$ and consequently $\tau(\rho_p)$ is never zero even if the Jacquet--Langlands transfer $\tau_p$ of $\pi_p$ is $0$. This will enable us to prove the appearance of $\rho$ in $\widetilde{H}^1(K^p,E)$ when $\rho_p$ is crystalline in (2) of Theorem \ref{thm:mainintro}. 
    \item Suppose that $\rho_p$ is crystalline where $\tau_p=0$. Then $\tau(\rho_p)$, as well as $\check{\Pi}(\rho)^{\lan}$ up to a multiplicity, is independent of the Hodge filtration of $D_{\dR}(\rho_p)$ (thus depends only on the Weil--Deligne representation $r_p$ attached to $\rho_p$). Hence the representation $\tau(\rho_p)$ can lose some information about $\rho_p$ (see Remark \ref{remarkordinary} for examples). This is a new phenomenon in the $p$-adic Langlands program for the group $D_p^{\times}$, which is different from the $\GL_2(\bbQ_p)$-case where $\pi(\rho_p)$ always determines $\rho_p$. Note that in the mod $p$ setting, a similar result was obtained by Hu--Wang \cite[Theorem 1.3(ii)(a)]{HuWang}.
\end{itemize}

Our proof of Theorem \ref{thm:mainintro} is based on the machinery developed by Lue Pan \cite{Pan22,PanII} on the description of the locally analytic vectors in the completed cohomology. Let $\calS_{K}$ be the adic space associated to $S_{K}\times_{\bbQ}\bbC_p$. Let $\calS_{K^p}\sim\ilim_{K_p}\calS_{K^pK_p}$ be the perfectoid quaternionic Shimura curve, together with the Hodge--Tate period map $\pi_{\HT}:\calS_{K^p}\to \check{\fl}$ (cf. \cite{Sch15}). Here, $\check{\fl}$ is (the base change to $\bbC_p$ of) the Brauer--Severi variety associated to the quaternion algebra $D_p$. Recall that $\tilde{H}^1(K^p,E)$ is the $E$-valued completed cohomology group of the quaternionic Shimura curve. Let $\tilde{H}^1(K^p,E)^{\lan}$ be the subspace of locally analytic vectors in $\tilde{H}^1(K^p,E)$ for the $D_p^{\times}$-actions. By the primitive comparison theorem \cite{scholze_p-adic_2013} and the geometric Sen theory \cite{Pan22,camargo2022geometric}, we have the following natural isomorphism  
\begin{align*}
    \tilde{H}^1(K^p,E)^{\lan}\widehat{\ox}_{\bbQ_p}\bbC_p\isom H^1(\check{\fl},\calO_{K^p}^{\lan})\ox_{\bbQ_p}E,
\end{align*}
where $\calO_{K^p}^{\lan}:=(\pi_{\HT,*}\calO_{\calS_{K^p}})^{\lan}$ and $\calO_{\calS_{K^p}}$ is the completed structure sheaf of $\calS_{K^p}$. Then we can study the completed cohomology group by studying the equivariant sheaf $\calO_{K^p}^{\lan}$ on $\check{\fl}$. As in \cite{PanII}, the ``de Rham part'' of the completed cohomology is killed by the Fontaine operator, which can be representation-theoretically expressed as some differential operators along the tower of Shimura curves and the flag variety. For example, there is a differential operator $d_{K^p}$ on the weight $(0,0)$-part (for the horizontal action of the Cartan subalgebra) of $\calO_{K^p}^{\lan}$ which induces a ``de Rham'' complex (see \S\ref{sec:passageShimura} for more details)
\begin{align*}
   d_{K^p}:\calO_{K^p}^{\lan,(0,0)}\stackrel{}{\rightarrow} \calO_{K^p}^{\lan,(0,0)}\otimes_{\calO_{K^p}^{\sm}}\Omega^{1,\sm}_{K^p}.
\end{align*}    
The above complex extends the classical de Rham complexes of Shimura curves of finite levels and its cohomology contributes to the ``de Rham part'' of the completed cohomology. 

The cohomology of these differential operators, such as $d_{K^p}$ above, can be analysed using the $p$-adic uniformization of the quaternionic Shimura curves, which expresses the tower of quaternionic Shimura curve as a finite disjoint union of quotients of the Drinfeld tower by some discrete cocompact subgroups of $\GL_2(\bbQ_p)$ (see \S\ref{sec:uniformization}). Let $\bar D$ be the quaternion algebra over $\bbQ$ obtained from $D$ by switching invariants at $p$ and $\infty$, and let $\bar G$ be the associated algebraic group over $\bbQ$ as in Theorem \ref{thm:mainintro}. We view $K^p$ as an open compact subgroup of $\bar G(\bbA_f^p)$ under the isomorphism $G(\bbA_f^p)\isom \bar G(\bbA_f^p)$. Let $\calA^{K^p}_{\bar G,0}$ be the space of smooth functions on the Shimura set $\bar G(\bbQ)\bs\bar G(\bbA_f)/K^p$ (of infinite level at $p$). Then $\calA^{K^p}_{\bar G,0}$ can be viewed as the space of classical algebraic automorphic forms on $\bar G$ with the trivial weight and the tame level $K^p$, and is naturally a smooth representation for the action of $\bar G(\bbQ_p)\simeq \GL_2(\bbQ_p)$. Both $\tilde{H}^1(K^p,E)$ and $\calA_{\bar G,0}^{K^p}$ carry an action of the spherical Hecke algebra $\bbT^S$ away from a finite set $S$ of prime numbers including $p$. We have the following typical product formula obtained from the $p$-adic uniformization. 
\begin{proposition}[Proposition \ref{prop:productformula}]\label{introproductformula}
    There exists a $\bbT^S$-equivariant isomorphism of $D_p^{\times}$-representations
    \begin{align*}
        R\Gamma(\check{\fl},\ker d_{K^p})[1]&\isom (\dlim_nH^1_c(\calM_{\LT,n},\calO_{\calM_{\LT,n}}))\ox_{\calH(G)}^L\calA^{K^p}_{\bar G,0}.
    \end{align*}
\end{proposition}
To obtain such a product formula, we first do the computation on the local Shimura variety (as in \cite[\S 5.2, 5.3]{PanII}) and then descend the results along the quotients by discrete cocompact subgroups in $\GL_2(\bbQ_p)$. Here, the key point is that we need do descent for some huge sheaves (e.g. sheaves related to $\ker d_{K^p}$) along some infinite-sheet coverings. For this reason, in Appendix \ref{section:properpushforward}, we developed a general theory of the compactly supported cohomology group with coefficients in some (possibly solid) abelian sheaves (like $\calO_{\calS_{K^p}}^{\lan}$) on partially proper adic spaces, and we also developed some relevant descent theory. While the theory is not general enough (we do not establish a full $6$-functor formalism), it suffices for our purpose.
\begin{remark}
    Using the $p$-adic uniformization, we can also establish similar product formulae relating the coherent cohomology of Shimura curves with that of the Drinfeld tower of dimension $1$ and the completed (co)homology of Shimura sets, see Remark \ref{remarkproductformula}. Such uniformization results for coherent cohomology of (infinite or finite level) Shimura varieties seem to be new and very useful. For example, formula (\ref{equationproductformulafinitelevel}) will allow us to calculate some higher extension groups for certain locally analytic representations of $\GL_2(\bbQ_p)$ arising from the coherent cohomology of Drinfeld spaces in \cite{DLB17} (e.g. (\ref{equationintrorepdrinfeld})), which will be treated in future work.
\end{remark}

Another key ingredient towards Theorem \ref{thm:mainintro} is the following non-vanishing result of the $D_p^{\times}$-representation $\tau_c,\widetilde{\tau}$ defined before Theorem \ref{thm:mainintro}. 
\begin{theorem}[Theorem \ref{thm:nonvanishingoftildetau}]\label{thmintornonvanishing}
    The $D_p^{\times}$-representation 
    \[\tau_c=(\dlim_n H^1_c(\calM_{\LT,n},\Omega_{\calM_{\LT,n}}^1))\ox_{\calH(G)}\pi_p\] 
    is infinite-dimensional. In particular, it is non-zero.
\end{theorem}
The proof for the theorem is mostly local (with some global input to understand the de Rham cohomology of the Lubin--Tate tower). We use some explicit formula for the action of the Lie algebra of $D_p^{\times}$ on equivariant sheaves on the Lubin--Tate spaces (deduced from the Lie algebra action formula on $\check{\fl}$ via the Gross--Hopkins period map) in \cite[\S 25]{GH94}. We refer the reader to \S\ref{subsectionnonvanishing} for more details. 
\begin{remark}
    \begin{enumerate}[(1)]
        \item If $\pi_p$ is a local component of a regular algebraic cuspidal automorphic representation of $\bar G$, in Proposition \ref{prop:Torvanishing}, we also prove that for all $i>0$, 
        \[\Tor_i^{\calH(G)}(\dlim_nH^1_c(\calM_{\LT,n},\calO_{\calM_{\LT,n}}),\pi_p)=0.\]
        \item We can also establish some representation-theoretic properties of $\tau_c$ when $\pi_p$ has global origins as in (1). The representation $\tau_c$ admits a natural $E$-structure denoted by $\tau_{c,E}$. It is an admissible locally analytic $D_p^\times$-representation over $E$ (by Theorem \ref{thm:mainintro}). Then we show that $\tau_{c,E}$ has Gelfand--Kirillov dimension $1$ (see Theorem \ref{thm:GKD}, which is based on the work \cite{dospinescu2023gelfand}), and it is pure (see Lemma \ref{lem:no-smooth-quotient}).
    \end{enumerate}  
\end{remark}
\begin{remark}
    We expect that when the Weil--Deligne representation $r_p$ is absolutely irreducible, the representation $\tau(\rho_p)$ we constructed should still determine $\rho_p$ as in the $\GL_2(\bbQ_p)$ case, see for example \cite{Din22}. Using the methods in \cite{QS24,su2025locallyanalytictextext1conjecturetextgl2l} and our geometric description of $\tau_{c}$ using the Lubin--Tate spaces, we can show that $\dim \mathrm{Ext}^1_{D_p^{\times}}(\tau_c,\tau_p)=2$. Hence the extension class (\ref{equationintroextension}) given by $\tau(\rho_p)$ should determine the Hodge filtration in the $2$-dimensional space $D_{\dR}(\rho_p)$. However, to show that $\tau(\rho_p)$ is non-isomorphic for distinct $\rho_p$ associated with the same $r_p$, further input is required (e.g. $\dim \Hom_{D_p^{\times}}(\tau_c,\tau_c)=1$). This will be left to future work.
\end{remark}
\subsection*{Overview} We briefly describe the structure of this article. 

In \S\ref{sec:quaternionicshimuracurve}, we introduce the perfectoid quaternionic Shimura curve and its $p$-adic uniformization. In \S\ref{subsec:locallyanalyticvectors}, we recall and adapt various constructions and results of Lue Pan to our quaternion setting. In \S\ref{sec:productformula}, we apply the $p$-adic uniformization and the descent results developed in Appendix \ref{section:properpushforward} to establish the product formula for compactly supported cohomology (Proposition \ref{introproductformula}). In \S\ref{sec:cohomologyI} and \S\ref{sec:cohomologyII}, we prove the main results on the  ``de Rham'' part of the locally analytic vectors of the completed cohomology of the Shimura curves (e.g. the spectral decomposition in Theorem \ref{thm:kerI1classical}).

In \S\ref{sec:padicLLD}, we study the locally analytic representations of $D_p^{\times}$ that appear in the completed cohomology attached to de Rham Galois representations. In \S\ref{subsec:repofDptimes} and \S\ref{sec:representationproperties}, we define these representations in terms of the Lubin--Tate tower (e.g. $\tau_c$ and $\widetilde{\tau}$) and study various properties of them, including the non-vanishing result (Theorem \ref{thmintornonvanishing}), the admissibility (in Theorem \ref{thm:mainintro}) and the Gelfand-Kirillov dimension (Theorem \ref{thm:GKD}). Then we apply the results to the completed cohomology and the $p$-adic Langlands correspondence for $D_p^{\times}$ in \S\ref{sec:completedcohoclassicality}. Finally, in \S\ref{sec:application}, we discuss the $p$-adic Jacquet--Langlands correspondence for $\GL_2(\bbQ_p)$ (i.e. Scholze's functor).

\subsection*{Notation and conventions}
In the whole paper, the field $E$ will be a finite extension of $\bbQ_p$. Let $\bbQ_p^{\ur}$ be the maximal unramified extension of $\bbQ_p$ in an algebraic closure $\overline{\bbQ_p}$ and let $\breve\bbQ_p$ be its $p$-adic completion with the ring of integers $\breve\bbZ_p=W(\overline{\bbF_p})$. We write $C=\bbC_p:=\widehat{\overline{\bbQ_p}}$.

For $F=E$ or $C$ a $p$-adic field and $H$ a $p$-adic manifold, write $\calC_{(c)}^{\sm}(H,F),\calC_{(c)}^{\lan}(H,F)$ and $\calC_{(c)}^{\rm cont}(H,F)$ respectively for the space of (compactly supported) $F$-valued locally constant, locally analytic, and continuous functions on $H$. If $H$ is a $p$-adic Lie group, we write $\Lie H$ for its Lie algebra over $\bbQ_p$.

If $g\in\GL_n$ is a matrix, we write $g^t$ for its transpose.

For $(a,b)\in\bbZ$, we write $V^{(a,b)}=\mathrm{Sym}^{a-b}V^{(1,0)}\otimes\det^{b}$ or $W^{(a,b)}$ for the algebraic representation of $\GL_2$ or its Lie algebra $\mathfrak{gl}_2$ (over a field) with highest weight $(a,b)$ (with respect to the Borel subgroup of upper triangular matrices), where $V^{(1,0)},W^{(1,0)}$ denote the standard $2$-dimensional representation of $\GL_2$. Let $\frh$ be the subalgebra of diagonal matrices in $\mathfrak{gl}_2$. Then we view $(a,b)$ as a character of $\frh$ by sending $\begin{pmatrix}x&0\\0&y
\end{pmatrix}$ to $ax+by$.

If $X$ is a space (scheme, adic space, linear space, etc.) with a left (resp. right) action of a group $H$, then we also view $X$ as a right (resp. left) $H$-space via the involution $g\mapsto g^{-1}$ on $H$. If $\calF$ is an equivariant sheaf on a left $H$-space $X$, then the global sections (or cohomology groups) of $\calF$ naturally form a left $H$-representation via the usual dual actions.

If $M$ is a vector space over a field $F$ with a linear action of a commutative ring $\bbT$, and $\lambda:\bbT\rightarrow F$ is a ring homomorphism, then we write $M[\lambda]$ for the subspace $\{m\in M\mid tm=0,\forall t\in \mathrm{Ker}(\lambda) \}\subset M$. 

Our convention is that the Hodge--Tate weight of the $p$-adic cyclotomic character of $\Gal(\overline{\bbQ_p}/\bbQ_p)$ is $-1$.

We normalize the local reciprocity map $\bbQ_{\ell}^{\times}\rightarrow \Gal(\overline{\bbQ_{\ell}}/\bbQ_{\ell})^{\rm ab}$ by sending $\ell$ to a geometric Frobenius element for any prime number $\ell$.

As in \cite{Pan22,PanII}, the sheaves on adic spaces considered in this paper are equipped with natural topologies (e.g. $\calO_{K^p}^{\lan}$ is a sheaf of LB-spaces, i.e. countable locally convex inductive limits of Banach spaces) and the isomorphisms between cohomology groups are topological isomorphisms. The topology and functional analysis issues have been discussed carefully in \textit{loc. cit.} and can also be treated using condensed mathematics as in \cite{camargo2024locallyanalyticcompletedcohomology,boxer2025modularity}. Since our setting is similar to these previous works and the different settings are not essential to the main results, we will generally omit the discussions on topology in this paper. We note that all completed tensor products $-\widehat{\ox}-$ over $p$-adic fields in this paper can be replaced by solid tensor products.
\subsection*{Acknowledgements}
We thank Yiwen Ding, Gabriel Dospinescu, Eugen Hellmann, Yongquan Hu, Yuanyang Jiang, Lue Pan, Haoran Wang and Liang Xiao for helpful discussions or for their interests in this work. Part of this work was done when the first author was visiting the Morningside Center of Mathematics (MCM) and when the second author was visiting the third author at Universität Münster in the summer of 2025. The first author would like to thank MCM and the second author would like to thank Universität Münster for the hospitalities during their visits.

The first author received funding from the European Union's Horizon 2020 research and innovation programme under the Marie Skłodowska-Curie grant agreement No 945332 and partially funded by the project “Group schemes, root systems, and related representations” founded by the European Union - NextGenerationEU through Romania’s National Recovery and Resilience Plan (PNRR) call no. PNRR-III-C9-2023- I8, Project CF159/31.07.2023. The second author was partially supported by the National Natural Science Foundation of China under agreement No. NSFC-12231001 and No. NSFC-12321001. The third author was funded by the Deutsche Forschungsgemeinschaft (DFG, German Research Foundation) – Project-ID 427320536 – SFB 1442, as well as under Germany's Excellence Strategy EXC 2044 390685587, Mathematics Münster: Dynamics–Geometry–Structure.
\section{Locally analytic cohomologies of Shimura curves}
\subsection{Quaternionic Shimura curve}\label{sec:quaternionicshimuracurve} First, we introduce our perfectoid quaternionic Shimura curve and describe its $p$-adic uniformization, cf. \cite[\S 5, \S 6]{Scholze2018LubinTate}. We also recall the Drinfeld and Lubin--Tate spaces (of dimension $1$) and the period maps on them.
\subsubsection{The Shimura curve}\label{subsec:quaternionicshimuracurve}
Let $D$ be a quaternion algebra over $\bbQ$, which defines a reductive group $G$ over $\bbQ$ by $G(R):=(D\ox_{\bbQ}R)^\times$ for any $\bbQ$-algebra $R$. Assume that $D$ splits over $\bbR$. We assume that $D_p:=D\ox_{\bbQ} \bbQ_p$ is a division algebra (the unique non-split quaternion algebra over $\bbQ_p$), and then $G(\bbQ_p)=D_p^{\times}$. 

Let $K\subset G(\bbA_f)$ be an open compact subgroup of the finite adelic points of $G$. Assume that there is a decomposition $K=K^pK_p$, where $K_p\subset D_p^\times,K^p\subset G(\bbA_f^p)$ are open compact subgroups. We will usually fix $K^p$ throughout this article and will always assume $K^p$ to be sufficiently small (so that to avoid problems of non-representability and to get the desired torsion freeness of some cocompact subgroups of $\GL_2(\bbQ_p)$ defined by $K^p$, see for example \cite[1.5]{MR1141456} and \S\ref{sec:uniformization}). 

Let $X=\bbC\bs\bbR$ be the union of the usual Poincar\'e upper and lower half planes, which can be seen as the $G(\bbR)$-conjugacy class of $h:\bbS=\mathrm{Res}_{\bbC/\bbR}\bbG_m\rightarrow G_{\bbR}=(D\otimes_{\bbQ}\bbR)^{\times}$ sending $a+bi\in \bbC^{\times}=\bbS(\bbR)$ to $\begin{pmatrix}
a&b\\-b&a
\end{pmatrix}\in \GL_2(\bbR)\cong (D\otimes_{\bbQ}\bbR)^{\times}$. Hence $G(\bbR)$ acts on $X$ naturally. The compact Riemann surface 
\[S_K(\bbC):=G(\bbQ)\bs (X\times G(\bbA_f)/K)\]
is known to be canonically defined over $\bbQ$. Let $S_K$ be the (projective smooth) Shimura curve over the reflex field $\bbQ$ associated with the Shimura datum $(G,X)$ and of level $K$.

Fix an algebraic closure $\overline{\bbQ_p}$ of $\bbQ_p$ and denote $C$ to be the completion of $\overline{\bbQ_p}$ for the $p$-adic norm. Let $\calS_K$ be the adic space over $C$ associated with $S_K\times_{\bbQ}C$. As $K_p$ varies in all open compact subgroups of $D_p^\times$, we get a perfectoid space over $C$:
\[
    \calS_{K^p}\sim\ilim_{K_p}\calS_{K^pK_p}
\]
together with the Hodge--Tate period map 
\[\pi_{K^p, \HT}:\calS_{K^p}\to \check\fl,\] 
where $\check\fl$ is the adic space associated with $\bbP^1_{C}$. See \cite{Sch15,caraiani2017generic,She17,Scholze2018LubinTate} for more details. Here $\check{\fl}$ carries the natural $D_p^\times$-action by viewing $\bbP^1_{C}$ as the base-change of the Brauer-Severi variety over $\bbQ_p$ associated to $D_p$ via $\bbQ_p\rightarrow C$, and (roughly speaking) the Hodge--Tate period map is defined by considering the Hodge filtration in the linear dual of the (trivialized) universal Tate module. The Hodge--Tate period map $\pi_{K^p, \HT}$ is $D_p^\times$-equivariant, compatible with the $\Gal(\overline{\bbQ_p}/\bbQ_p)$-action. It is also compatible with the following Hecke actions. Let $S$ be a finite set of prime numbers containing $p$, such that $G_{\bbQ_{\ell}}$ is unramified for $\ell\not\in S$ and $K^p=K^p_S K^S$ where $K^S\subset G(\bbA^S_f)$ is a product of hyperspecial compact open subgroups. Let 
\[\bbT^S=\bbZ[K^S\bs G(\bbA^S_f)/K^S]\]
be the abstract Hecke algebra. Then $\pi_{K^p, \HT}$ is equivariant for the Hecke action of $\bbT^S$ (with the trivial $\bbT^S$-action on $\check{\fl}$). 

\subsubsection{Shimura sets, Drinfeld tower and the $p$-adic uniformization}\label{sec:uniformization}
The $p$-adic uniformization theorem expresses the Shimura curve $\calS_{K^p}$ in terms of quotient of the Drinfeld tower of dimension $1$ at infinite level by some discrete cocompact subgroups of $\GL_2(\bbQ_p)$ (see (\ref{uniformization}) below) which will be important in our study of the completed cohomology of $\calS_{K^p}$. 

First, we introduce the quaternionic Shimura set. Let $\bar D$ be the quaternion algebra over $\bbQ$ obtained from $D$ by changing the local invariants exactly at the archimedean place and $p$. Let $\bar G$ be the reductive group over $\bbQ$ associated to $\bar D$. Hence $\bar G(\bbR)$ is compact modulo the center and the $p$-component $\bar G_{p}:=\bar G(\bbQ_p)$ of $\bar G(\bbA)$ is isomorphic to $\GL_2(\bbQ_p)$ induced by the splitting of $\bar D$ at $p$. We can and we do fix an isomorphism 
\[G(\bbA_{f}^p)\isom \bar G(\bbA_f^{p})\]
so that $K^p$ can be viewed as an open subgroup of $\bar G(\bbA_f^p)$.  The Shimura set
\begin{align}\label{eq:defofZKp}
    Z_{K^p}:= \bar G (\bbQ)\bs \bar G (\bbA_f)/K^p
\end{align}
is a profinite set (by finiteness of class numbers \cite[Theorem 5.1]{borel1963some}) with commuting actions of $\bbT^S$ and $\bar G_p=\GL_2(\bbQ_p)$.

Next, we introduce the Drinfeld tower. The Drinfeld space of dimension $1$ over $C$ of level $0$ is non-canonically given by $\calM_{\Dr,0}\simeq \sqcup_{h\in\bbZ}\bbP^{1,\mathrm{an}}_{C}\setminus \bbP^1(\bbQ_p)$, the base change to $C$ (via $\breve{\bbQ}_p\hookrightarrow C$) of the adic generic fiber of a formal scheme over $\Spf \calO_{\breve{\bbQ}_p}$, which is certain moduli space (Rapoport--Zink space) of $p$-divisible special formal $\calO_{D_p}$-modules, see \cite{Dr76} or \cite[Theorem. 3.72]{RZ96}. Let $V_{\Dr}$ be the \textit{dual} of the $p$-adic Tate module of the universal $p$-divisible group on $\calM_{\Dr,0}$. The Drinfeld tower $(\calM_{\Dr,n})_{n\geq 0}$ consists of \'etale Galois coverings $\calM_{\Dr,n}$ of $\calM_{\Dr,0}$ with the Galois group $(\calO_{D_p}/p^n\calO_{D_p})^\times=\calO_{D_p}^\times/(1+p^n\calO_{D_p})$ trivializing $V_{\Dr}/p^n V_{\Dr}$. By \cite[Theorem 6.5.4]{SW13}, there exists a perfectoid space $\calM_{\Dr,\infty}$ over $C$ such that 
\begin{align*}
    \calM_{\Dr,\infty}\sim\ilim_n\calM_{\Dr,n}
\end{align*}
and $\calM_{\Dr,\infty}$ is equipped with a natural continuous left action of $D_p^{\times}\times \GL_2(\bbQ_p)$. The group $\GL_2(\bbQ_p)$ acts also continuously on each layer $\calM_{\Dr,n}$, $n\ge 0$. The Hodge filtration of the covariant Dieudonn\'e module of the universal $p$-divisible group defines the Gross--Hopkins/Grothendieck--Messing period map \cite[Prop. 6.5.5]{SW13}
\begin{align*}
    \pi_{\Dr,\GM}:\calM_{\Dr,\infty}\to {\fl}
\end{align*}
where $\fl$ means the adic space associated to $\bbP^1_C$ equipped with the natural $\GL_2(\bbQ_p)$-action and the trivial action of $D_p^{\times}$. The Hodge filtration of the dual of the rational Tate module of the universal $p$-divisible group over $\calM_{\Dr,\infty}$ defines the Hodge--Tate period map \cite[Prop. 7.1.1]{SW13}
\begin{align*}
    \pi_{\Dr,\HT}:\calM_{\Dr,\infty}\to \check{\fl}.
\end{align*}
Note that our normalization for the $D_p^{\times}$-action is different from \textit{loc. cit.} by the involution of $D_p$ and follows that of \cite{PanII}, since here the tower $\calM_{\Dr,\infty}$ trivializes the dual of the Tate modules (rather than the Tate modules). See \S\ref{sec:vectorbundles} for relevant discussions. The map $\pi_{\Dr,\HT}$ is equivariant for the $\GL_2(\bbQ_p)\times D_p^\times$-action where $\GL_2(\bbQ_p)$ acts trivially on $\check{\fl}$.

The following theorem gives the $p$-adic uniformization of the perfectoid Shimura curve $\calS_{K^p}$.
\begin{theorem}[\v{C}erednik, Drinfeld]\label{thm:padicuniformization}
We have a $D_p^\times$-equivariant isomorphism of adic spaces over $C$:
\begin{align}\label{eq:padicuniformization}
    \calS_{K^p}\isom \bar G(\bbQ)\bs (\calM_{\Dr,\infty}\times \bar G(\bbA_f^p)/K^p),
\end{align}
which is compatible with the $\breve{\bbQ}_p$-structure, the (effective) Weil descent data to $\bbQ_p$ on both sides and the action of $\bbT^S$.
\end{theorem}
\begin{proof}
See for example \cite{Dr76}, \cite[Theorem III]{RZ96}, \cite{BZ00,boutot2022uniformization} or \cite[\S 6]{Scholze2018LubinTate} for more details. More precisely, for each $n\ge 0$ there exists an isomorphism 
\[ \calS_{K^p,n}\isom \bar G(\bbQ)\bs (\calM_{\Dr,n}\times \bar G(\bbA_f^p)/K^p)\]
where $\calS_{K^p,n}$ is the Shimura curve of level $K^p(1+p^n\calO_{D_p})$. Moreover, after we replace $\calM_{\Dr,\infty}$ by its quotient of a discrete cyclic subgroup of the center of $\bar G_p$, the Weil descent datum on the right hand side is effective and the isomorphism is $\bbQ_p$-rational. See \cite[Proposition 6.16]{RZ96}, \cite[(A.11)]{harris1997supercuspidal} for example. Then we take inverse limits on $n$.
\end{proof}

\begin{remark}\label{remarkuniformizationtwist}
    Here, in (\ref{eq:padicuniformization}), we note that the action of $\bar G(\bbQ)$ on $\calM_{\Dr,\infty}$ is induced from the left action of $\GL_2(\bbQ_p)$ on $\calM_{\Dr,\infty}$ via the map $\bar G(\bbQ)\inj \bar G(\bbQ_p)=\GL_2(\bbQ_p)$ twisted by the Cartan involution of $\GL_2(\bbQ_p)$ (the inverse transpose $g\mapsto (g^{t})^{-1}$) as in \cite[Theorem 5.4.2]{PanII}. The action of $\bar G(\bbQ)$ on $\bar G(\bbA_f^p)$ is given by multiplication from the left. The group $K^p$ acts on $\calM_{\Dr,\infty}$ trivially, and acts on $\bar G(\bbA_f^p)$ by multiplication from the right. 
\end{remark}
We can rewrite the isomorphism (\ref{eq:padicuniformization}) in terms of the Shimura set $Z_{K^p}=\bar G(\bbQ)\bs  \bar G(\bbA_f)/K^p$. There is a canonical isomorphism (cf. \cite[(A.10)]{harris1997supercuspidal} or \cite[4.5.2.1]{fargues2004cohomologie})
\begin{align*}
    \bar G(\bbQ)\bs (\calM_{\Dr,n}\times \bar G(\bbA_f^p)/K^p)&\isom (\calM_{\Dr,n}\times \bar G(\bbQ)\bs\bar G(\bbA_f)/K^p  )/\bar G_p,[x,yK^p]\mapsto [x,\bar G(\bbQ)y K^p]
\end{align*}
where the right $\bar G_p$-action on $\calM_{\Dr,n}\times\bar G(\bbQ)\bs\bar G(\bbA_f)/K^p$ is given by $(x,y^p,y_p)g=(g^{t}x,y^p,y_pg)$ for $x\in \calM_{\Dr,n}$, $y^p \in \bar G(\bbA_f^p)$, $y_p\in \bar G_p\isom \GL_2(\bbQ_p)$, and $g\in \bar G_p$. Under the above isomorphism, we rewrite (\ref{eq:padicuniformization}) as follows: 
\begin{align}\label{eq:realpadicuniformization}
    \calS_{K^p}\isom (\calM_{\Dr,\infty}\times( \bar G(\bbQ)\bs  \bar G(\bbA_f)/K^p))/\bar G_p.
\end{align}
This isomorphism is compatible with the $G_p$-actions on $\calS_{K^p}$ and $\calM_{\Dr,\infty}$, the actions of $\bbT^S$ on $\calS_{K^p}$ and $\bar G(\bbQ)\bs  \bar G(\bbA_f)/K^p$. 
\begin{remark} Under the isomorphism (\ref{eq:realpadicuniformization}), the $G_p=D_p^\times$-equivariant Hodge--Tate period map  $\pi_{\Dr,\HT}:\calM_{\Dr,\infty}\to \check \fl$ induces a $D_p^\times$-equivariant map on $\calS_{K^p}$, which is exactly the Hodge--Tate period map on $\calS_{K^p}$. See \cite[Claim 3.4.2]{She17} for more examples and the discussions below.
\end{remark}
For further applications, it is convenient to rewrite the $D_p^\times$-equivariant isomorphism (\ref{eq:realpadicuniformization}) as
\begin{align}\label{uniformization}
    \calS_{K^p}\isom \sqcup_{x\in  Z_{K^p}/\GL_2(\bbQ_p) }\Gamma_x\bs \calM_{\Dr,\infty}.
\end{align} 
Here, $Z_{K^p}$ has finitely many $\GL_2(\bbQ_p)$-orbits. For every $x\in Z_{K^p}$ represented by $x\in \bar G(\bbA_f)$, we put $\Gamma_x=x^{-1}\bar G(\bbQ)x\cap K^p$ as the stabilizer of $x$ for the right action of $\GL_2(\bbQ_p)$ on $Z_{K^p}$. The groups $\Gamma_x$ are discrete cocompact subgroups of $\GL_2(\bbQ_p)$ since $\bar G(\bbQ)$ is discrete in $\bar G(\bbA_f)$ by \cite[Proposition 1.4]{gross1999algebraic}, and after shrinking $K^p$ if necessary, we may assume all these stabilizer groups $\Gamma_x$ are torsion-free (cf. \cite[5.4]{MR1141456} or take $\ell\neq p$ such that the $\ell$-component $K_{\ell}$ of $K^p$ is torsion-free).

Let $\Gamma$ be one of the above cocompact subgroups $\Gamma_x$ for some $x$. Then $\calS_{\Gamma}:=\Gamma\bs\calM_{\Dr,\infty}$ is an adic space over $C$ (as it is a component of $\calS_{K^p}$). The projection map $\pr_{\Gamma,\infty}:\calM_{\Dr,\infty}\to\Gamma\bs\calM_{\Dr,\infty}$ is an analytic $\Gamma$-covering of adic spaces by the following proposition. 
\begin{proposition}\label{prop:prGamma}
Let $\Gamma$ be a closed discrete cocompact subgroup of $\GL_2(\bbQ_p)$. If $\Gamma$ is torsion-free, then the projection map $\pr_{\Gamma,\infty}:\calM_{\Dr,\infty}\to \Gamma\bs \calM_{\Dr,\infty}$ is a local isomorphism in the following sense: for any $x\in \Gamma\bs \calM_{\Dr,\infty}$, there exists an open affinoid perfectoid open subset $U$ which contains $x$, such that $\pr_{\Gamma,\infty}^{-1}(U)$ is a disjoint union $\sqcup_{\gamma\in\Gamma}\gamma(V)$, with $V\subset \calM_{\Dr,\infty}$ an affinoid perfectoid open subset, such that $\pr_{\Gamma,\infty}|_{V}:V\to U$ is an isomorphism.
\end{proposition}
\begin{proof}
   We only need to prove similar statements for $\calM_{\Dr,0}$ (using \cite[Proposition 4.2(ii)]{CDN20}). The result was proved for subgroups of $\mathrm{PGL}_2(\bbQ_p)$ acting on $\Omega=\bbP^{1,\mathrm{an}}_{C}\setminus \bbP^1(\bbQ_p)$ in \cite[\S 5, Theorem 2]{SS91}. Let $\mathcal{B}$ be the Bruhat-Tits tree for $\mathrm{PGL}_2(\bbQ_p)$. In the proof of \textit{loc. cit.}, there exists an open covering of $\Omega$ by open affinoids $U_{\sigma}^a$ constructed by Drinfeld where $0<a<1$ and $\sigma$ vary in all simplices of $\mathcal{B}$. Moreover, $g(U_{\sigma}^{a})=U_{g(\sigma)}^a$ and $U_{g(\sigma)}\cap U_{\sigma}=\emptyset$ if $g(\sigma)\neq \sigma$. Let $\GL_2(\bbQ_p)$ act on $\mathcal{B}\times \bbZ$ via $g.(x,m)=(g.x,m+v_{p}(\det(g)))$ as in \cite[\S 3.68]{RZ96}. By Drinfeld, see \cite[\S 3.71]{RZ96}, there exists a $\GL_2(\bbQ_p)$-equivariant isomorphism $\calM_{\Dr,0}\simeq \Omega\times\bbZ$ (defined over $\breve{\bbQ}_p$). For a simplex $\sigma=(\sigma',m)$ of $\mathcal{B}\times \bbZ$, let $U_{\sigma}^a$ be the open affinoid $U_{\sigma'}^a\times\{m\}\subset \calM_{\Dr,0}$ via the previous isomorphism. Then $(U_{\sigma}^a)$ form an open covering of $\calM_{\Dr,\infty}$ when $\sigma$ vary in all simplices of $\mathcal{B}\times\bbZ$. We still have $g(U_{\sigma}^{a})=U_{g(\sigma)}^a$ for $g\in\GL_2(\bbQ_p)$ and $U_{g(\sigma)}^a\cap U_{\sigma}^a=\emptyset$ if $g(\sigma)\neq \sigma$. The stabilizer $\Gamma_{\sigma}\subset \Gamma$ of each simplex $\sigma$ in $\mathcal{B}$ is an open compact subgroup of $\GL_2(\bbQ_p)$. Since $\Gamma$ is discrete, we see that each $\Gamma_{\sigma}$ is a finite group, and is furthermore trivial since $\Gamma$ is torsion-free. Hence $U_{\gamma'(\sigma)}^a\cap U_{\gamma(\sigma)}^a=\emptyset$ for all $\gamma\neq\gamma'\in\Gamma$. We can then conclude as in the proof of \cite[\S 5, Theorem 2]{SS91}.
\end{proof}

Moreover, as $\Gamma$ acts trivially on $\check\fl$, the Hodge--Tate period map $\pi_{\Dr,\HT}:\calM_{\Dr,\infty}\to \check{\fl}$ descends to a map of adic spaces 
\[\pi_{\Gamma,\HT}:\calS_{\Gamma}\to \check{\fl}\] which coincides with the restrictions of $\pi_{K^p, \HT}$ via (\ref{uniformization}). 
\begin{lemma}\label{lem:prGammaHT}
There exists a basis $\ffrb_{\Gamma}$ for the analytic topology of $\check{\fl}$, such that each $U\in\ffrb_{\Gamma}$ is an affinoid open subset, $V_\infty=\pi_{\Gamma,\HT}^{-1}(U)$ is an affinoid perfectoid open subset inside $\calS_{\Gamma}$, and there exists a sufficiently large $n$ such that $V_\infty=\pi_{\Gamma,n}^{-1}(V_n)$ for some $V_n\subset \Gamma\bs\calM_{\Dr,n}$ affinoid open subset. Here $\pi_{\Gamma,n}:\Gamma\bs\calM_{\Dr,\infty}\to \Gamma\bs\calM_{\Dr,n}$ is the projection map. 
\end{lemma}
\begin{proof} 
This follows from the $p$-adic uniformization (\ref{uniformization}) and the similar properties of $\pi_{K^p, \HT}:\calS_{K^p}\to \check\fl$, see \cite[Proposition 4.4.53]{boxer2021highercolemantheory}. 
\end{proof}

\subsubsection{The two towers}\label{subsection:twotowers}
We will also need the perfectoid Lubin--Tate space (cf. \cite[\S 6.3]{SW13})
\begin{align*}
    \calM_{\LT,\infty}\sim\ilim_n\calM_{\LT,n}.
\end{align*}
The $0$-th layer $\calM_{\LT,0}$ is the base change to $C$ of the adic generic fiber of of the formal scheme over $\mathrm{Spf}(\breve{\bbZ}_p)$ parametrizing deformations (up to quasi-isogeny) of $1$-dimensional formal $p$-divisible groups of height $2$ over $\bar\bbF_q$, isomorphic to $\bbZ$-copies of the open unit disk (see \cite{GH94} or \cite[\S 3.78]{RZ96}). Each $\calM_{\LT,n}\rightarrow \calM_{\LT,0}$ is an \'etale Galois covering with  Galois group $\GL_{2}(\bbZ_p/p^n\bbZ_p)$. There is a natural continuous (left) action of $\GL_{2}(\bbQ_p)$ on $\calM_{\LT,\infty}$. Also $D_p^\times$ acts on each layer $\calM_{\LT,n}$, $n\ge 0$ (using the moduli description). Similarly as for $\calM_{\Dr,\infty}$, we have the Gross--Hopkins period map $\pi_{\LT,\GM}:\calM_{\LT,\infty}\to {\fl}$ \cite{GH94} and the Hodge--Tate period map $\pi_{\LT,\HT}:\calM_{\LT,\infty}\to \check{\fl}$ \cite{SW13}. Following \cite{PanII}, we twist the usual $\GL_2(\bbQ_p)$-action (in \cite{SW13}) on $\calM_{\LT,\infty}$ by the Cartan involution $g\mapsto (g^{-1})^t$, which corresponds to the usual $\GL_2(\bbQ_p)$-action on $\calM_{\Dr,\infty}$ under the following isomorphism. As a consequence of our normalization, the actions of the centers $\bbQ_p^{\times}$ of $D_p^{\times}$ and $\GL_2(\bbQ_p)$ coincide on $\calM_{\LT,\infty}$. 
\begin{theorem}[{\cite[Theorem 7.2.3]{SW13}}]\label{thm:twotowers}
There is a $\GL_2(\bbQ_p)\times D_p^\times$-equivariant isomorphism 
\begin{align}\label{eq:isomorphismtwotowers}
    \calM_{\LT,\infty}\isom \calM_{\Dr,\infty}
\end{align}
such that under this isomorphism, $\pi_{\LT,\GM}$ is identified with $\pi_{\Dr,\HT}$, and $\pi_{\LT,\HT}$ is identified with $\pi_{\Dr,\GM}$. 
\end{theorem}

The period maps $\pi_{\LT,\GM}=\pi_{\Dr,\HT}$ (resp. $\pi_{\LT,\HT}=\pi_{\Dr,\GM}$) realize $\calM_{\LT,\infty}$ as a pro\'etale $\GL_2(\bbQ_p)$-torsor (resp. $D_p^{\times}$-torsor) over $\check{\fl}$ (resp. over $\bbP^{1,\an}_C\setminus\bbP^1(\bbQ_p)\subset \fl$) \cite[Corollary 23.5.3]{scholze2020berkeley}. Besides, the map $\pi_{\LT,\GM}:\calM_{\LT,\infty}\to \check{\fl}$ factors through 
\begin{align*}
    \calM_{\LT,\infty}\ov{\pi_{\LT,0}}\to\calM_{\LT,0}\ov{\pi_{\LT,\GM,0}}\to  \check{\fl}.
\end{align*}
where $\pi_{\LT,0}$ denotes the projection map $\calM_{\LT,\infty}\surj \calM_{\LT,0}$. The map $\pi_{\LT,0}$ is a $\GL_2(\bbZ_p)$-torsor, and there exists a basis $\ffrb$ of $\calM_{\LT,0}$ consisting of affinoid open subsets, such that for each $U\in\ffrb$, $\pi^{-1}_{\LT,0}(U)$ is affinoid perfectoid (cf. \cite{Wei16}, \cite[Proposition 4.2(i)]{CDN20}). The Gross--Hopkins period map $\pi_{\LT,\GM,0}:\calM_{\LT,0}\to \check{\fl}$ is \'etale, and there exists an open cover $\{U_i\}$ of $\check{\fl}$, such that each $U_i$ is partially proper, and $\pi_{\LT,\GM,0}^{-1}(U_i)$ is a disjoint union of finite \'etale coverings of $U_i$ (cf. \cite[Proposition 7.2]{deJong}). Moreover, the map $\pi_{\LT,\GM,0}$ is surjective with local sections \cite[Lemma 6.1.4]{SW13}. We note that there are infinitely many components of $\pi_{\LT,\GM,0}^{-1}(U_i)$, as the geometric fibers of $\pi_{\LT,\GM,0}$ are $\GL_2(\bbQ_p)/\GL_2(\bbZ_p)$. Besides, the Hodge--Tate period map $\pi_{\Dr,\HT}$ decomposes as the composition $\pi_{\Gamma,\HT}\comp \pr_{\Gamma,\infty}$, where $\pr_{\Gamma,\infty}:\calM_{\Dr,\infty}\to \calS_{\Gamma}$ is a $\Gamma$-torsor for the analytic topology and a local isomorphism in the notation and in the sense of Proposition \ref{prop:prGamma}, and $\pi_{\Gamma,\HT}:\calS_{\Gamma}\to \check{\fl}$ is ``affinoid'' in the sense of Lemma \ref{lem:prGammaHT}. We summarize the picture as follows
$$
\begin{tikzcd}
& {\calM_{\LT,\infty}\isom \calM_{\Dr,\infty}} \arrow[ld, "\pi_{\LT,0}"'] \arrow[rd, "\pr_{\Gamma,\infty}"] &        \\
{\calM_{\LT,0}} \arrow[rd, "{\pi_{\LT,\GM,0}}"'] &  & { \calS_{\Gamma}} \arrow[ld, "{\pi_{\Gamma,\HT}}"] \\& \check{\fl} &                                                             
\end{tikzcd}.
$$
\subsection{Locally analytic vectors à la Pan}\label{subsec:locallyanalyticvectors}
Let $\calO_{\calM_{\Dr,\infty}}\simeq \calO_{\calM_{\LT,\infty}}$ be the completed structure sheaf on the infinite level perfectoid Drinfeld/Lubin--Tate space under the isomorphism in Theorem \ref{thm:twotowers}. We recall some basic properties and constructions about the sheaf of locally analytic sections in $\calO_{\calM_{\Dr,\infty}}$ and $\calO_{\calM_{\LT,\infty}}$ as in \cite{Pan22,PanII}.

Let $\check\frg:=\Lie D_p^\times\ox_{\bbQ_p}C$. We fix a splitting throughout this article 
\begin{align}\label{equationsplittingD}
    D_p\ox_{\bbQ_p}C\isom \Mat_2(C).
\end{align} 
Using this splitting of $D_p$, we get a Lie algebra isomorphism $\check\frg\isom \frg:=\gl_2(C)$, where $\frg$ is the Lie algebra of $\GL_2(C)$. We will identify them via this isomorphism. 
\subsubsection{Equivariant line bundles on the flag varieties}\label{sec:vectorbundles}
The splitting (\ref{equationsplittingD}) induces an isomorphism $\check\fl=\bbP^1_C\simeq (\GL_2/\bar B)_{C}$ where $\overline{B}$ is the opposite of the Borel subgroup of upper triangular matrices $B\subset \GL_2$ and we will also identify $\fl$ with $(\GL_{2}/\bar B)_C$. For any weight $(a,b)\in \bbZ^2$ of the Cartan subalgebra $\check{\frh}\simeq \frh$ of diagonal matrices, we write $\omega_{\check{\fl}}^{(a,b)}$ for the $\GL_2$-equivariant line bundle on $\check{\fl}$ associated to the character $(a,b):\bar B\to C, \left( \begin{matrix} x&0\\y&z \end{matrix} \right)\mapsto x^az^b$. And we similarly write $\omega_{\fl}^{(a,b)}$ for the line bundle on $\fl$ associated to $(a,b)$ (also using the Borel subgroup consisting of lower triangular matrices). Hence the dualizing sheaves are given by $\Omega^1_{\check{\fl}}=\omega_{\check{\fl}}^{-2\rho},\Omega^1_{\fl}=\omega_{\fl}^{-2\rho}$ where $2\rho=(1,-1)$. In this normalization, we have (see Remark \ref{remarknormalizationLT} below)
\begin{equation}\label{equationpullbackHTGM}
    \pi_{\Dr,\HT}^*\omega_{\check\fl}^{(a,b)}\simeq \pi_{\Dr,\GM}^*\omega_{\fl}^{(-a,-b)}\otimes_{\calO_{\calM_{\Dr,\infty}}}\calO_{\calM_{\Dr,\infty}}(-a)
\end{equation}
for $(a,b)\in\bbZ^2$, where $(-a)$ denotes the $-a$th Tate twist.

\begin{remark}\label{remarknormalizationLT}
    Let $V^{(1,0)}$ be the standard representation over $C$ of $\GL_2$. The pullback along $\pi_{\LT,\HT}$ of the the short exact sequence of equivariant bundles on $\fl$ induced from the $\bar B$-sequence $0\rightarrow (0,1)\rightarrow V^{(1,0)}|_{\bar B}\rightarrow (1,0)\rightarrow 0$ 
    \begin{align}\label{equationHTseqLT}
        0\to \pi_{\LT,\HT}^*\omega_{\fl}^{(0,1)}\to V^{(1,0)}\ox_C\calO_{\calM_{\LT,\infty}}\to \pi_{\LT,\HT}^*\omega_{\fl}^{(1,0)}\to 0
    \end{align}
    gives the (dual) Hodge--Tate sequence (cf. (\ref{equationpullbackHTGM})):
    \[0\rightarrow \Lie(\calG^{\vee})\rightarrow V_p(\calG)^{\vee}\otimes_{\bbQ_p}\calO_{\LT,\infty}\rightarrow \Lie(\calG)^{-1}(-1)\rightarrow 0\]
    on $\calM_{\LT,\infty}$, where we denote by $\calG$ the univeral $p$-divisible group on $\calM_{\LT,\infty}$, $V_p(\calG)$ the rational Tate module and $\Lie(\calG)$ the Lie algebra. 
    
    Let $W^{(1,0)}$ denote the standard representation over $C$ of $\GL_2$ on which on which $D_p^{\times}$ acts via (\ref{equationsplittingD}). The $\bar B$-filtration on $W^{(1,0)}|_{\bar B}$ together with $\pi_{\LT,\GM}$ induces the Grothendieck--Messing sequence on the infinite level $\calM_{\LT,\infty}$, which is a $D_p^{\times}$-equivariant exact sequence 
    \begin{align}\label{equationHTseqGM}
         0\to \pi_{\LT,\GM}^*\omega_{\check\fl}^{(0,1)}\to W^{(1,0)}\ox_C\calO_{\calM_{\LT,\infty}}\to \pi_{\LT,\GM}^*\omega_{\check\fl}^{(1,0)}\to 0.
    \end{align}
    The above exact sequence is identified with the pullback of the sequence
    \begin{align*}
         0\to \Lie(\calG^{\vee})^{-1}\to M(\calG)\to \Lie(\calG) \to 0
    \end{align*}
    where $M (\calG)$ denotes the covariant Dieudonn\'e crystal of $\calG$ on $\calM_{\LT,0}$.

    Let $\check{\calG}$ be the univeral $p$-divisible special formal $\calO_{D_p}$-module over $\calM_{\Dr,\infty}$. Let $N$ be the isocrystal over $\breve{\bbQ}_p$ of slope $-\frac{1}{2}$ so that $D_p=\End(N)$. Under the isomorphism (\ref{eq:isomorphismtwotowers}), we have (see proof of \cite[Theorem 7.2.3]{SW13} or \cite[\S 8.1.2, \S 8.3.2]{fargues2019courbes})  $D_p$-equivariant isomorphisms \[M(\check{\calG})\simeq \Hom_{\bbQ_p}(V_p(\calG),N\otimes_{\breve{\bbQ}_p}\calO_{\calM_{\Dr,\infty}}), V_p(\check{\calG})\otimes_{\bbQ_p}\calO_{\calM_{\Dr,\infty}}\simeq \Hom(M(\calG),N\otimes_{\breve{\bbQ}_p}\calO_{\calM_{\Dr,\infty}})\] and 
    \[\Lie(\check{\calG})=\Hom(\Lie(\calG)(1),N\otimes_{\breve{\bbQ}_p}\calO_{\calM_{\Dr,\infty}}).\] There are similar Hodge--Tate and Grothendieck--Messing sequences on $\calM_{\Dr,\infty}$ for $V_p(\check{\calG})^{\vee}$ and $M(\check{\calG})$ corresponding to (\ref{equationHTseqGM}) and (\ref{equationHTseqLT}) respectively. 
\end{remark}

Let 
\[\omega_{\calM_{\Dr,\infty}}^{(a,b)}:=\pi_{\Dr,\GM}^*\omega_{\fl}^{(a,b)}=\pi_{\LT,\HT}^*\omega_{\fl}^{(a,b)}\] and 
\[\omega_{\calM_{\LT,\infty}}^{(a,b)}:=\pi_{\LT,\GM}^*\omega_{\check\fl}^{(a,b)}=\pi_{\Dr,\HT}^*\omega_{\check\fl}^{(a,b)}\]
be the pullback of the line bundles on $\calM_{\Dr,\infty}\simeq \calM_{\LT,\infty}$ for any $(a,b)\in\bbZ^2$. Then there is a $\GL_2(\bbQ_p)\times D_p^\times$-equivariant and Galois equivariant isomorphism \[\omega_{\calM_{\Dr,\infty}}^{(a,b)}=\omega_{\calM_{\LT,\infty}}^{(-a,-b)}(-a)\] 
by (\ref{equationpullbackHTGM}). The Grothendieck--Messing maps factor through finite levels: $\calM_{\Dr,\infty}\stackrel{\pi_{\Dr,n}}{\rightarrow}\calM_{\Dr,n}\stackrel{\pi_{\Dr, \GM,n}}{\rightarrow} \fl$ and $\calM_{\LT,\infty}\stackrel{\pi_{\LT,n}}{\rightarrow}\calM_{\LT,n}\stackrel{\pi_{\LT, \GM,n}}{\rightarrow} \check\fl$. Let 
\begin{align*}
    &\omega_{\calM_{\Dr,n}}^{(a,b)}:=\pi_{\Dr,\GM,n}^{*}\omega_{\fl}^{(a,b)},\\
    &\omega_{\calM_{\LT,n}}^{(a,b)}:=\pi_{\Dr,\GM,n}^{*}\omega_{\check\fl}^{(a,b)}
\end{align*} 
and set 
\begin{align*}
&\omega_{\calM_{\Dr,\infty}}^{(a,b),\sm}:=\colim_n\pi_{\Dr,n}^{-1}\omega_{\calM_{\Dr,n}}^{(a,b)},\\
&\omega_{\calM_{\LT,\infty}}^{(a,b),\sm}:=\colim_n\pi_{\LT,n}^{-1}\omega_{\calM_{\LT,n}}^{(a,b)}.
\end{align*}

\subsubsection{Locally analytic vectors and geometric Sen thoery}\label{sec:locallyanalyticvectors}
Recall that $\calO_{\calM_{\Dr,\infty}}=\calO_{\calM_{\LT,\infty}}$ denotes the completed structure sheaf of $\calM_{\Dr,\infty}=\calM_{\LT,\infty}$. Let $\calO_{\calM_{\Dr,\infty}}^{\lan}$ (reps. $ \calO^{\lan}_{\calM_{\LT,\infty}}$) be the subsheaf of $\calO_{\calM_{\Dr,\infty}}$ consisting of locally analytic sections for the $D_p^\times$-action (resp. $\GL_2(\bbQ_p)$-action). By \cite[Corollary 5.3.9]{PanII}, \cite[Proposition 4.3.6]{QS24} or \cite[Corollary 5.1.9]{dospinescu2024jacquetlanglandsfunctorpadiclocally}, we have
\begin{align*}
    \calO_{\calM_{\Dr,\infty}}^{\lan}=\calO_{\calM_{\LT,\infty}}^{\lan}.
\end{align*}

Similarly we define $\calO_{\calM_{\Dr,\infty}}^{\sm}$ (resp. $\calO_{\calM_{\LT,\infty}}^{\sm}$) as the subsheaf consisting of $D_p^{\times
}$-smooth (resp. $\GL_2(\bbQ_p)$-smooth) sections. Also let $\calO_{\calM_{\Dr,\infty}}^{\lalg}\subset \calO_{\calM_{\Dr,\infty}}^{\lan}$ (resp. $ \calO_{\calM_{\LT,\infty}}^{\lalg}\subset \calO_{\calM_{\LT,\infty}}^{\lan}$) be the subsheaf consisting of locally algebraic sections for the $D_p^\times$-action (resp. $\GL_2(\bbQ_p)$-action). By \cite[Corollary 4.2.7]{Eme17}, we know 
\begin{align*}
   \calO_{\calM_{\Dr,\infty}}^{\lalg}\isom \bigoplus_{W} \calO_{\calM_{\Dr,\infty}}^{W\-\lalg},
\end{align*}
where $W$ varies in the set of irreducible algebraic representations of $D_p^\times$ over $C$. Similar decomposition result also holds for $\calO_{\calM_{\LT,\infty}}^{\lalg}$.

For $(a,b)\in\bbZ^2$, we also have subsheaves $\omega_{\calM_{\Dr,\infty}}^{(a,b),\lalg}
\subset \omega_{\calM_{\Dr,\infty}}^{(a,b),\lan}$ 
of the $\GL_2(\bbQ_p)\times D_p^{\times}$-equivariant sheaf $\omega_{\calM_{\Dr,\infty}}^{(a,b)}=\calO_{\calM_{\Dr,\infty}}\otimes_{\calO_{\calM_{\Dr,\infty}}^{\sm}}\omega_{\calM_{\Dr,\infty}}^{(a,b),\sm}$ defined similarly as in \S\ref{sec:vectorbundles}. We have $\omega_{\calM_{\Dr,\infty}}^{(a,b),\lan}=\calO_{\calM_{\Dr,\infty}}^{\lan}\otimes_{\calO_{\calM_{\Dr,\infty}}^{\sm}}\omega_{\calM_{\Dr,\infty}}^{(a,b),\sm}$. Again, similar results hold on the Lubin--Tate side.

We recall below results from the geometric Sen theory, developed in \cite{Pan22,PanII,camargo2022geometric,dospinescu2024jacquetlanglandsfunctorpadiclocally}. As in \cite[\S 4.2.6]{Pan22} and following Beilinson–Bernstein, let $\check \frg^0:=\check \frg\ox_{C}\calO_{\check\fl},\frg^0:= \frg\ox_{C}\calO_{\fl}$, which are Lie algebroids over $\calO_{\check\fl}$ or $\calO_{\fl}$. Let $\frh\simeq  \check{\frh}\subset\check{\frg}\simeq \frg=\gl_2$ be the Cartan subalgebra consisting of diagonal matrices. Let $\check\frn^0\subset \check\frb^0\subset  \check\frg^0$ be the universal nilpotent subalgebra and the univeral Borel subalgebra over $\check{\fl}$, with $(\check\frn^0)^{\vee}$ the dual of $\check\frn^0$. Similarly, let $\frn^0\subset\frb^0 \subset \frg^0$ be the univeral nilpotent subalgebra and the universal Borel subalgebra over $\fl$. The preimage of these Lie algebroids, e.g. $\pi_{\Dr,\HT}^{-1}\check\frg^0$, acts on the locally analytic vectors $\calO_{\calM_{\Dr,\infty}}^{\lan}$ by derivations.

Let $\calZ(\check{\frg})$ be the center of the universal enveloping algebra of $\check\frg$ which is isomorphic to $S(\check{\frh})^W$ via the Harish-Chandra isomorphism, where $W$ is the Weyl group for $\check{\frg}$ and $W$ acts on $S(\check{\frh})$ via the dot action. Recall that $\calZ(\check{\frg})$ is a $C$-algebra generated by $z=\left( \begin{matrix}1&0\\0&1\end{matrix} \right)$ and the Casimir operator $\Omega=u^+u^-+u^-u^++\frac{1}{2}h^2\in Z(U(\check{\frg}))$, where $u^+=\left( \begin{matrix}0&1\\0&0\end{matrix} \right)$, $u^-=\left( \begin{matrix}0&0\\1&0\end{matrix} \right)$ and $h=\left( \begin{matrix}1&0\\0&-1\end{matrix} \right)$. Here we used our fixed isomorphism $\check{\frg}\isom \frg$ to define these elements $z$ and $\Omega$. 

As in \cite[\S 5.1.4]{Pan22} or \cite[Theorem 6.3.6]{camargo2024locallyanalyticcompletedcohomology}, there is an action of the arithmetic Sen operator $\Theta_{\Sen}$ on $\calO_{K^p}^{\lan}$. We also have the following relation between the arithmetic Sen operator, the horizontal action $\theta_{\check\frh}$ of the Cartan subalgebra $\check\frh$ (defined below), and the action of the center $\calZ(\check{\frg})$.
\begin{theorem}\label{thm:Senaction}
The following statements hold true:
\begin{enumerate}[(i)]
    \item The locally analytic vectors $\calO_{\calM_{\Dr,\infty}}^{\lan}\simeq \calO_{\calM_{\LT,\infty}}^{\lan}$ is killed by the ``geometric Sen operators'' $\pi_{\Dr,\HT}^{-1}\check{\frn}^0$ and $\pi_{\LT,\HT}^{-1}\frn^0$. 
    \item Let $\theta_{\check\frh}$ (resp. $\theta_{\check\frh}$) be the horizontal action of $\check\frh\simeq \frh$ on $\calO_{\calM_{\Dr,\infty}}^{\lan}$ (resp. $\calO_{\calM_{\LT,\infty}}^{\lan}$) via $\check\frb^0/\check\frn^0\simeq \check\frh\otimes_C\calO_{\check\fl}$ (resp. $\frb^0/\frn^0\simeq \frh\otimes_C\calO_{\fl}$). Then $\theta_{\check\frh}=\theta_{\frh}$.  
    \item The action of $z=\left( \begin{matrix}1&0\\0&1\end{matrix} \right)$ and $\theta_{\check{\frh}}(z)$ on $\calO_{\calM_{\Dr,\infty}}^{\lan}$ are the same.
    \item We have $\Theta_{\Sen}=\theta_{\check\frh}\left( \begin{matrix} 1&0\\0&0 \end{matrix} \right)$ and $(2\Theta_{\Sen}-z+1)^2-1=2\Omega$ as operators on $\calO_{\calM_{\Dr,\infty}}^{\lan}$.
\end{enumerate}
\begin{proof}
The proof is similar to \cite[Theorem 5.1.8]{Pan22}, \cite[Proposition 6.5.4]{PanII}, and the results also follow by \cite[Theorem 5.2.5, Theorem 6.3.6]{camargo2024locallyanalyticcompletedcohomology}. 
Here, we note that in \cite{Pan22} Pan uses the upper triangular matrices to define the flag variety, and we use the lower triangular matrices. So that there is a small difference about the description of results on horizontal actions of $\check{\frh}$.
\end{proof}
\end{theorem}

Some of the above results on various operators on $\calO_{\calM_{\Dr,\infty}}^{\lan}\simeq \calO_{\calM_{\LT,\infty}}^{\lan}$ also generalize to certain twists of these sheaves. As the nilpotent Lie algebra $\pi_{\Dr,\HT}^{-1}\check\frn^0$ kills $\omega_{\calM_{\Dr,\infty}}^{(a,b),\lan}=\calO_{\calM_{\Dr,\infty}}^{\lan}\otimes_{\calO_{\calM_{\Dr,\infty}}^{\sm}}\omega_{\calM_{\Dr,\infty}}^{(a,b),\sm}$, there is also a horizontal action $\theta_{\check\frh}$ of $\check{\frh}$ on $\omega_{\calM_{\Dr,\infty}}^{(a,b),\lan}$. For $(a,b),(a',b')\in\bbZ^2$, let $\calO_{\calM_{\Dr,\infty}}^{\lan,(a',b')}$ (resp. $\omega_{\calM_{\Dr,\infty}}^{(a,b),\lan,(a',b')}$) be the subsheaf of $ \calO_{\calM_{\Dr,\infty}}^{\lan}$ (resp. $\omega_{\calM_{\Dr,\infty}}^{(a,b),\lan}$) such that the action of $\theta_{\check \frh}$ is given by the weight $(a',b')$. Using the fact that $\theta_{\check{\frh}}$ acts on $\omega_{\check{\fl}}^{(a',b')}$ via $(a',b')$ and (\ref{equationpullbackHTGM}), we have the following formula relating these twists:
\begin{align}\label{lemmatwistomegaDr1}
    \calO_{\calM_{\Dr,\infty}}^{\lan,(a,b)}\ox_{\calO_{\calM_{\Dr,\infty}}^{\sm}}\omega_{\calM_{\Dr,\infty}}^{(a',b'),\sm}&\isom \omega_{\calM_{\Dr,\infty}}^{(a',b'),\lan,(a,b)},\\\label{lemmatwistomegaDr2}
    \calO_{\calM_{\Dr,\infty}}^{\lan,(a,b)}\ox_{\pi_{\Dr,\HT}^{-1}\calO_{\check{\fl}}}\pi_{\Dr,\HT}^{-1}\omega_{\check{\fl}}^{(a',b')}&\isom \omega_{\calM_{\Dr,\infty}}^{(-a',-b'),\lan,(a+a',b+b')}(-a').
\end{align}
Using the above formulae, one can check that
\begin{align*}
    \pi_{\Dr,\HT}^{-1}\Omega_{\check{\fl}}^{1}\ox_{\pi_{\Dr,\HT}^{-1}\calO_{\check{\fl}}}\calO_{\calM_{\Dr,\infty}}^{\lan,(0,0)}\ox_{\calO_{\calM_{\Dr,\infty}}^{\sm}}\Omega_{\calM_{\Dr,\infty}}^{1,\sm}\isom \calO_{\calM_{\Dr,\infty}}^{\lan,(-1,1)}(1).
\end{align*}

We can calculate some $W$-locally algebraic vectors explicitly for some algebraic representation $W$ of $D_p^\times$ over $C$.
\begin{proposition}\label{prop:Wlalg}
For $(a,b)\in\bbZ^2,a\geq b$, let $W^{(a,b)}$ be the algebraic representation of $D_p^{\times}\subset \GL_2(C)$ over $C$ of highest weight $(a,b)$ (with respect to the Borel subgroup consisting of upper triangular matrices). Then 
\[\calO_{\calM_{\Dr,\infty}}^{\lalg,(a,b)}=\calO_{\calM_{\Dr,\infty}}^{W^{(a,b)}\-\lalg}\simeq W^{(a,b)}\otimes_C \omega_{\calM_{\Dr,\infty}}^{(a,b),\sm}(a).\]
\begin{proof}
See \cite[\S 3.4]{QS24} for example. We briefly illustrate the case when $(a,b)=(0,-1)$. The subsheaf $\calO_{\calM_{\Dr,\infty}}^{\lalg,(a,b)}\subset \calO_{\calM_{\Dr,\infty}}^{\lalg}$ is the subsheaf of $\calO_{\calM_{\Dr,\infty}}^{\lan,(a,b)}$ consisting of sections that are locally algebraic for the $D_p^\times$-action.  The first equality can be obtained by explicit descriptions of the sheaf $\calO_{\calM_{\Dr,\infty}}^{\lan,(a,b)}$ in terms of certain power series expansions as in \cite[Theorem 3.2.2]{PanII}. Next, using (\ref{equationHTseqGM}) together with (\ref{equationpullbackHTGM}), we deduce an exact sequence 
\[
    0\to \omega_{\calM_{\Dr,\infty}}^{(0,-1)}\to W^{(1,0)}\otimes_C \calO_{\calM_{\Dr,\infty}}\to \omega_{\calM_{\Dr,\infty}}^{(-1,0)}(-1)\to 0.
\]
Then the term $\omega_{\calM_{\Dr,\infty}}^{(0,-1),\sm}$ will contribute to the locally $(W^{(1,0)})^*$-algebraic part of $\calO_{\calM_{\Dr,\infty}}$. The general cases follow similarly.
\end{proof}
\end{proposition}

\subsubsection{The differential operators}\label{sec:differentialoperators}
There are various differential operators on 
\[\calO_{\calM_{\Dr,\infty}}^{\lan,(0,-k)}=\calO_{\calM_{\LT,\infty}}^{\lan,(0,-k)}, k\geq 0\] 
that are constructed in \cite[\S 4]{PanII}, which extend the connections on the corresponding algebraic parts.

The differential operator (see \cite[Remark 5.3.18]{PanII} or \cite[Proposition 4.3.10]{QS24})
\[d_{\Dr}:\calO_{\calM_{\Dr,\infty}}^{\lan,(0,-k)}\rightarrow \calO_{\calM_{\Dr,\infty}}^{\lan,(0,-k)}\otimes (\Omega^{1,\sm}_{\calM_{\Dr,\infty}})^{\otimes k+1}\] 
extends the connection on $\calO_{\calM_{\Dr,\infty}}^{\lalg,(0,-k)}\simeq W^{(0,-k)}\otimes_C \omega_{\calM_{\Dr,\infty}}^{(0,-k),\sm}$, or more classically the Gauss--Manin connections  
\begin{equation}\label{equationGMconnectionDr}
    \omega_{\calM_{\Dr,n}}^{(0,-k)}\rightarrow \omega_{\calM_{\Dr,n}}^{(0,-k)}\otimes (\Omega_{\calM_{\Dr,n}}^{1})^{\otimes(k+1)}\simeq \omega_{\calM_{\Dr,n}}^{(-k-1,1)}.    
\end{equation}

\begin{remark}\label{remarktranslation}
    These connections are pullbacks of the connections on $\omega_{\fl}^{(0,-k)}$ via the \'etale map $\calM_{\Dr,n}\to \fl$. We note that we use the Borel subgroup consisting of lower triangular matrices to define the flag variety $\fl$ (in particular $\omega_{\fl}^{(1,0)}$ is ample). When $k=0$, it's clear that $d_{\Dr}$ is the pullback of the de Rham complex on $\fl$. For general $k$, the connection can also be constructed by translations of the de Rham complex as in \cite[Lemma 5.6.6]{PanII} or as in the usual BGG construction \cite{faltings2006cohomology}. We use the notation of Remark \ref{remarknormalizationLT}. Since we have trivialized the univeral Dieudonn\'e module $M(\check{\calG})$ on $\calM_{\Dr,0}$ \cite[Proposition 5.15]{RZ96} which coincides with $V^{(1,0)}\ox_C \calO_{\calM_{\Dr,n}}$ up to a multiplicity space given by a finite-dimensional $D_p^\times$-space, the connection on $V^{(1,0)}\ox_C \calO_{\calM_{\Dr,n}}$ is trivial on $V^{(1,0)}$ (cf. \cite[Lemma 3.11]{vanhaecke2024cohomologie}). Taking a symmetric power of the dual of the Dieudonn\'e module, we see the connection 
    \begin{align}\label{equationconnectionVk}
        V^{(0,-k)}\ox \calO_{\calM_{\Dr,n}}\to V^{(0,-k)}\ox \Omega^1_{\calM_{\Dr,n}}    
    \end{align}
    is trivial on $V^{(0,-k)}$ and is given by the usual differential map on the structure sheaf $\calO_{\calM_{\Dr,n}}$. Besides, the map (\ref{equationconnectionVk}) is equivariant for the $\GL_2(L)$-action. Taking the part of (\ref{equationconnectionVk}) with infinitesimal character equaling to the infinitesimal character of $V^{(0,-k)}$ (i.e. the translation of the usual differential map $\calO_{\calM_{\Dr,n}}\to \Omega^1_{\calM_{\Dr,n}}$ from the trivial weight to the weight $(0,-k)$), we obtain the desired map (\ref{equationGMconnectionDr}), which is quasi-isomorphic to (\ref{equationconnectionVk}) as complexes by \cite[Lemma 5.6.6]{PanII}. On the other hand, there is a Hodge filtration on $V^{(0,-k)}\ox \calO_{\calM_{\Dr,n}}$ (given by the dual of the Hodge-de Rham sequence (\ref{equationHTseqLT}) if $k=1$) so that the map (\ref{equationconnectionVk}) satisfies the Griffiths transversality. The induced map given by (\ref{equationconnectionVk}) is an isomorphism on all graded pieces of the Hodge filtration except the top one $\omega_{\calM_{\Dr,n}}^{(0,-k)}$ and induces (\ref{equationGMconnectionDr}) (cf. \cite[Lemma 9.8]{vanhaecke2024cohomologie}). 
\end{remark} 
The differential operator (\cite[\S 5.3.16]{PanII})
\[\bar d_{\Dr}:\calO_{\calM_{\Dr,\infty}}^{\lan,(0,-k)}\rightarrow \calO_{\calM_{\Dr,\infty}}^{\lan,(0,-k)}\ox_{\pi_{\Dr,\HT}^{-1}\calO_{\check \fl}}(\pi_{\Dr,\HT}^{-1}\Omega_{\check \fl}^1)^{\ox k+1}\] 
arises from the action of the differential operators on the flag variety $\check\fl$, extending a twist of the pullback of the connection
\begin{equation}\label{equationconnectioncheckfl}
    \omega_{\check\fl}^{(0,-k)}\rightarrow \omega_{\check\fl}^{(0,-k)}\otimes (\Omega_{\check{\fl}}^{1})^{\otimes(k+1)}\simeq \omega_{\check{\fl}}^{(-k-1,1)}.    
\end{equation}
There are also twists of these differential operators, which we denote them by $d_{\Dr}',\bar d_{\Dr}'$. We summarize the results as follows.
\begin{theorem}\label{theoremdifferentialoperatorsDr}
Let $k\ge 0$ be an integer. There is a commutative diagram 
\[
\begin{tikzcd}
    \calO_{\calM_{\Dr,\infty}}^{\lan,(0,-k)} \arrow[r, "d_{\Dr}"] \arrow[d, "\bar d_{\Dr}"] & \calO_{\calM_{\Dr,\infty}}^{\lan,(0,-k)}\ox_{\calO_{\calM_{\Dr,\infty}}^{\sm}}(\Omega^{1,\sm}_{\calM_{\Dr,\infty}})^{\ox k+1} \arrow[d, "\bar d_{\Dr}'"] \\
    \calO_{\calM_{\Dr,\infty}}^{\lan,(0,-k)}\ox_{\pi_{\Dr,\HT}^{-1}\calO_{\check \fl}}(\pi_{\Dr,\HT}^{-1}\Omega_{\check \fl}^1)^{\ox k+1} \arrow[r, "d_{\Dr}'"]           &  \calO_{\calM_{\Dr,\infty}}^{\lan,(-k-1,1)}(k+1)       
\end{tikzcd}
\]
such that:
\begin{enumerate}[(i)]
    \item The operator $d_{\Dr}$ is $\pi_{\Dr,\HT}^{-1}\calO_{\check{\fl}}$-linear, and is given by the Gauss--Manin connection (\ref{equationGMconnectionDr}) when restricted to $\calO_{\calM_{\Dr,\infty}}^{\lalg,(0,-k)}$.
    \item The operator $\bar d_{\Dr}$ is $\calO_{\calM_{\Dr,\infty}}^{\sm}$-linear. If we twist $\bar d_{\Dr}$ by $\omega_{\calM_{\Dr,\infty}}^{(0,-k),\sm}$, its restriction on $\pi_{\Dr,\HT}^{-1}\omega_{\check\fl}^{(0,-k)}$ is induced by the differential (\ref{equationconnectioncheckfl}). 
    \item The map $\bar d_{\Dr}$ is surjective, with kernel $\calO_{\calM_{\Dr,\infty}}^{\lalg,(0,-k)}$.
    \item $d'_{\Dr}$ is a twist of $d_{\Dr}$ by $\pi_{\Dr,\HT}^{-1}(\Omega_{\check\fl}^1)^{\ox k+1}$, and $\bar d'_{\Dr}$ is a twist of $\bar d_{\Dr}$ by $(\Omega_{\calM_{\Dr,\infty}}^{1,\sm})^{\ox k+1}$.
\end{enumerate}
\end{theorem}
\begin{proof}
This is \cite[Theorem 5.2.15, Theorem 5.2.16, Theorem 5.3.17]{PanII} except that we don't take the direct image along the period maps. We note that by (\ref{lemmatwistomegaDr1}) and (\ref{lemmatwistomegaDr2}) there is an isomorphism 
\[ (\pi_{\Dr,\HT}^{-1}\Omega_{\check \fl}^1)^{\ox k+1}\ox_{\pi_{\Dr,\HT}^{-1}\calO_{\check \fl}}\calO_{\calM_{\Dr,\infty}}^{\lan,(0,-k)}\ox_{\calO_{\calM_{\Dr,\infty}}^{\sm}}(\Omega^{1,\sm}_{\calM_{\Dr,\infty}})^{\ox k+1}\simeq \calO_{\calM_{\Dr,\infty}}^{\lan,(-k-1,1)}(k+1).\] 
The construction of these differential operators follow by the explicit description of $\calO_{\calM_{\Dr,\infty}}^{\lan,(0,-k)}$ in terms of power series expansions as in \cite[Theorem 3.2.2]{PanII}. 
\end{proof}
The above theorem also holds \textit{mutatis mutandis} for the differential operators $d_{\LT},\bar d_{\LT}$ on $\calO_{\calM_{\LT,\infty}}^{\lan,(0,-k)}$ (and their twists) in \cite[\S 5.2.14]{PanII}, replacing everywhere ``$\Dr$'' by ``$\LT$'' and ``$\check\fl$'' by ``$\fl$''. In particular, we also have a commutative diagram
\[
\begin{tikzcd}
    \calO_{\calM_{\LT,\infty}}^{\lan,(0,-k)} \arrow[r, "d_{\LT}"] \arrow[d, "\bar d_{\LT}"] & \calO_{\calM_{\LT,\infty}}^{\lan,(0,-k)}\ox_{\calO_{\calM_{\LT,\infty}}^{\sm}}(\Omega^{1,\sm}_{\calM_{\LT,\infty}})^{\ox k+1} \arrow[d, "\bar d_{\LT}'"] \\
    \calO_{\calM_{\LT,\infty}}^{\lan,(0,-k)}\ox_{\pi_{\LT,\HT}^{-1}\calO_{ \fl}}(\pi_{\LT,\HT}^{-1}\Omega_{ \fl}^1)^{\ox k+1} \arrow[r, "d_{\LT}'"]           &  \calO_{\calM_{\LT,\infty}}^{\lan,(-k-1,1)}(k+1)
\end{tikzcd}
\]
where $\bar d_{\LT}$ is surjective with kernel $\calO_{\calM_{\LT,\infty}}^{\lalg,(0,-k)}$. It turns out that (up to a non-zero constant) $d_{\LT}$ coincides with $\bar d_{\Dr}$ and $\bar d_{\LT}$ coincides with $d_{\Dr}$, and the above diagram is essentially the same as the one in Theorem \ref{theoremdifferentialoperatorsDr}.
\begin{theorem}\label{theoremdifferentialoperatorsDrLT}
    Under the isomorphism $\calO_{\calM_{\Dr,\infty}}^{\lan,(0,-k)}=\calO_{\calM_{\LT,\infty}}^{\lan,(0,-k)}$, we have $d_{\Dr}=c\bar d_{\LT},\bar d_{\Dr}=c'd_{\LT}$ for some invertible constants $c,c'$.
\end{theorem}
\begin{proof}
    This is \cite[Theorem 5.3.20]{PanII} and \cite[Theorem 4.3.12]{QS24}.
\end{proof}
The identification of these operators allows us to compute the cohomology of the differential operator $d_{\Dr}$.
\begin{corollary}\label{cor:kerdDr}
The differential operator $d_{\Dr}:\calO_{\calM_{\Dr,\infty}}^{\lan,(0,-k)} \to  \calO_{\calM_{\Dr,\infty}}^{\lan,(0,-k)}\ox_{\calO_{\calM_{\Dr,\infty}}^{\sm}}(\Omega^{1,\sm}_{\calM_{\Dr,\infty}})^{\ox k+1}$ is surjective, with kernel $\calO_{\calM_{\LT,\infty}}^{\lalg,(0,-k)}$. Similar results hold for $d_{\Dr}'$.
\end{corollary}
From the above corollary we obtain a short exact sequence
\begin{align}\label{equationdDr}
    0\rightarrow \calO_{\calM_{\LT,\infty}}^{\lalg,(0,-k)}\rightarrow \calO_{\calM_{\Dr,\infty}}^{\lan,(0,-k)} \stackrel{d_{\Dr}}{\rightarrow}  \calO_{\calM_{\Dr,\infty}}^{\lan,(0,-k)}\ox_{\calO_{\calM_{\Dr,\infty}}^{\sm}}(\Omega^{1,\sm}_{\calM_{\Dr,\infty}})^{\ox k+1}\rightarrow 0
\end{align}
which is $\GL_2(\bbQ_p)\times D_p^{\times}$-equivariant. This sequence will be crucial for our study of the locally analytic representations of $D_p^{\times}$ from the compactly supported cohomology groups of these sheaves and the completed cohomology of the quaternionic Shimura curve.
\subsubsection{Passage to the Shimura curve and the flag variety}\label{sec:passageShimura}
So far we have discussed some equivariant line bundles on $\calM_{\Dr,\infty}$ in \S\ref{sec:vectorbundles}, subsheaves of $\calO_{\calM_{\Dr,\infty}}^{\lan}$ \S\ref{sec:locallyanalyticvectors} and the differential operators in \S\ref{sec:differentialoperators}. These objects and constructions are $\GL_2(\bbQ_p)$-equivariant, hence descend to the analytic quotients $\calS_{\Gamma}=\Gamma \bs\calM_{\Dr,\infty}$ where $\Gamma\subset \GL_2(\bbQ_p)$ is a closed cocompact subgroup satisfying the assumption for Proposition \ref{prop:prGamma}.  Via the $p$-adic uniformization $\calS_{K^p}=\sqcup_{x\in  Z_{K^p}/\GL_2(\bbQ_p) }\calS_{\Gamma_x}$ in the notation of (\ref{uniformization}), we can obtain corresponding objects on the Shimura curve $\calS_{K^p}$. Note that we can make the same constructions directly on $\calS_{K^p}$ as is done in \cite{PanII} or \cite{QS24}. 

Recall that $\pr_{\Gamma,\infty}:\calM_{\Dr,\infty}\rightarrow \calS_{\Gamma}$ is the quotient map and $\pi_{K^p,\HT}:\calS_{K^p}\to \check\fl$ is the Hodge--Tate map for $\calS_{K^p}$. Let $(a,b)\in\bbZ^{\oplus 2}$ and $k\geq 0$. In summary, there are
\begin{itemize}
    \item automorphic line bundles $\omega_{\calS_{K^p}}^{(a,b)}=\pi_{K^p,\HT}^*\omega_{\check\fl}^{(-a,-b)}(-a)$ and the smooth vectors $\omega_{\calS_{K^p}}^{(a,b),\sm}$ for $(a,b)\in\bbZ^2$ such that 
    \[\pr_{\Gamma,\infty}^{-1}(\omega_{\calS_{K^p}}^{(a,b)}|_{\calS_{\Gamma}})=\omega_{\calM_{\Dr,\infty}}^{(a,b)},\quad\pr_{\Gamma,\infty}^{-1}(\omega_{\calS_{K^p}}^{(a,b),\sm}|_{\calS_{\Gamma}})=\omega_{\calM_{\Dr,\infty}}^{(a,b),\sm}.\]
    In particular, $\Omega^{1,\sm}_{\calS_{K^p}}$ is naturally identified with $\omega_{\calS_{K^p}}^{(-1,1),\sm}$.
    \item subsheaves of \[\calO_{\calS_{K^p}}^{W\-\lalg}\subset\calO_{\calS_{K^p}}^{\lalg}
    \subset \calO_{\calS_{K^p}}^{\lan}\subset \calO_{\calS_{K^p}},\]
    and similarly subsheaves of $\omega_{\calS_{K^p}}^{(a,b), \lan}$, whose pullbacks to $\calM_{\Dr,\infty}$ via $\pr_{\Gamma,\infty}$ coincides with the corresponding sheaves on $\calM_{\Dr,\infty}$ in \S\ref{sec:locallyanalyticvectors}. These sheaves are killed by the nilpotent Lie algebra $\pi_{K^p,\HT}^{-1}\check\frn^0$ and also admit horizontal actions of $\check{\frh}$ still denoted by $\theta_{\check\frh}$.
    \item the subsheaves $\calO_{\calS_{K^p}}^{\lan,(a',b')},\omega_{\calS_{K^p}}^{(a,b),\lan,(a',b')}$, etc. where $\check{\frh}$ acts via the weight $(a',b')\in\bbZ^2$ whose pullback to $\calM_{\Dr,\infty}$ coincide with $\calO_{\calM_{\Dr,\infty}}^{\lan,(a',b')},\omega_{\calM_{\Dr,\infty}}^{(a,b),\lan,(a',b')}$, etc.
    \item the differential operators $d_{\calS_{K^p}},\bar d_{\calS_{K^p}},d_{\calS_{K^p}}',\bar d_{\calS_{K^p}}'$ on $\calO_{\calS_{K^p}}^{\lan,(0,-k)}$ for $k\geq 0$, which satisfy the results of Theorem \ref{theoremdifferentialoperatorsDr} replacing everywhere ``$\calM_{\Dr,\infty}$'' by ``$\calS_{K^p}$'' and ``$\Dr$'' by ``$\calS_{K^p}$''. All these differential operators are surjection of sheaves as in (iii) of Theorem \ref{theoremdifferentialoperatorsDr} and Corollary \ref{cor:kerdDr}, which will be discussed in Proposition \ref{prop:dKvsurj} below.
\end{itemize}
\begin{remark}
    The sheaves $\omega_{\calS_{K^p}}^{(a,b),\sm}$ we constructed above coincide with the usual automorphic vector bundles on the Shimura varieties constructed using the Grothendieck--Messing period maps, see \cite[Remark 5.2.11, Corollary 5.3.3]{DLLZ}. 
\end{remark}
Finally we can pushforward the above sheaves along the ``affinoid'' map $\pi_{K^p,\HT}:\calS_{K^p}\to \check\fl$ to obtain sheaves on $\check\fl$ and the corresponding differential operators. We set, exactly as in \cite{Pan22,PanII} or \cite{QS24},
\begin{itemize}
    \item $\calO_{K^p}:=\pi_{K^p,\HT,*}\calO_{\calS_{K^p}},\calO_{K^p}^{\sm}:=\pi_{K^p,\HT,*}\calO_{\calS_{K^p}}^{\sm},\omega_{K^p}^{(a,b)}:=\pi_{K^p,\HT,*}\omega_{\calS_{K^p}}^{(a,b)}$, etc.
    \item the subsheaf of locally analytic or locally algebraic vectors 
    \[ \calO_{K^p}^{W\-\lalg}\subset\calO_{K^p}^{\lalg}
    \subset \calO_{K^p}^{\lan}\]
    of the $D_p^{\times}$-equivariant sheaf $\calO_{K^p}$. There are similar subsheaves of $\omega_{K^p}^{(a,b),\lan}$. These sheaves are killed by the nilpotent Lie algebra $\check\frn^0$ and also admit horizontal actions of $\check{\frh}$ still denoted by $\theta_{\check\frh}$, hence are equivariant twisted $\calD$-modules in a suitable sense (e.g. \cite[\S 3]{boxer2025modularity}). 
    \item the subsheaves $\calO_{K^p}^{\lan,(a',b')}\subset \calO_{K^p}^{\lan},\omega_{K^p}^{(a,b),\lan,(a',b')}\subset \omega_{K^p}^{(a,b),\lan}$, etc. where $\check{\frh}$ acts via the weight $(a',b')\in\bbZ^2$.
    \item there are differential operators as in Theorem \ref{theoremdifferentialoperatorsDr}
    \begin{equation}\label{equationdKv}
     \begin{tikzcd}
    \calO_{K^p}^{\lan,(0,-k)} \arrow[r, "d_{K^p}"] \arrow[d, "\bar d_{K^p}"] & \calO_{K^p}^{\lan,(0,-k)}\ox_{\calO_{K^p}^{\sm}}(\Omega^{1,\sm}_{K^p})^{\ox k+1} \arrow[d, "\bar d_{K^p}'"] \\
    \calO_{K^p}^{\lan,(0,-k)}\ox_{\calO_{\check\fl}}(\Omega^1_{\check\fl})^{\ox k+1} \arrow[r, "d'_{K^p}"]           &  \calO_{K^p}^{\lan,(-k-1,1)}(k+1)       
    \end{tikzcd}
   \end{equation}
   for $k\geq 0$ integers.
\end{itemize}

\begin{proposition}\label{prop:dKvsurj}
    Let $d_{K^p},\bar d_{K^p},d'_{K^p},\bar d'_{K^p}$ be the differential operators on $\calO_{K^p}^{\lan,(0,-k)}$ in (\ref{equationdKv}). Then the following statements hold:
\begin{enumerate}[(1)]
    \item The map $\bar d_{K^p}:\calO_{K^p}^{\lan,(0,-k)}\to \calO_{K^p}^{\lan,(0,-k)}\ox_{\calO_{\check\fl}}(\Omega^1_{\check\fl})^{\ox k+1}$ is surjective, with kernel 
    \begin{align}\label{equationOKplalg}
        \calO_{K^p}^{\lalg,(0,-k)}\isom W^{(0,-k)}\ox_C \omega_{K^p}^{(0,-k),\sm}.
    \end{align} 
    Similarly, the map $\bar d'_{K^p}:\calO_{K^p}^{\lan,(0,-k)}\ox_{\calO_{K^p}^{\sm}}(\Omega_{K^p}^{1,\sm})^{\ox k+1}\to \calO_{K^p}^{\lan,(-k-1,1)}(k+1)$ is surjective, with kernel $\calO_{K^p}^{\lalg,(0,-k)}\ox_{\calO_{K^p}^{\sm}}(\Omega_{K^p}^{1,\sm})^{\ox k+1}\isom W^{(0,-k)}\ox_C \omega_{K^p}^{(-k-1,1),\sm}$.
    \item The maps $d_{K^p}:\calO_{K^p}^{\lan,(0,-k)}\to \calO_{K^p}^{\lan,(0,-k)}\ox_{\calO_{K^p}^{\sm}}(\Omega_{K^p}^{1,\sm})^{\ox k+1}$ and $d'_{K^p}:\calO_{K^p}^{\lan,(0,-k)}\ox_{\calO_{\check\fl}}(\Omega^1_{\check\fl})^{\ox k+1} \to  \calO_{K^p}^{\lan,(-k-1,1)}(k+1)$ are surjective.
\end{enumerate}
\end{proposition}

\begin{proof}
(1) The statement for $\bar d_{K^p}$ can be proved using similar arguments as in \cite[Proposition 4.2.9]{PanII} and $\bar d'_{K^p}$ is a twist of $\bar d_{K^p}$.

(2) Since $\calS_{K^p}=\sqcup_{x\in  Z_{K^p}/\GL_2(\bbQ_p) }\calS_{\Gamma_x}$ and each map $\calM_{\Dr,\infty}\rightarrow \calS_{\Gamma_x}$ is a local isomorphism (Proposition \ref{prop:prGamma}), the differential operator $d_{\calS_{K^p}}:\calO_{\calS_{K^p}}^{\lan,(0,-k)}\to \calO_{\calS_{K^p}}^{\lan,(0,-k)}\ox_{\calO_{\calS_{K^p}}^{\sm}}(\Omega_{K^p}^{1,\sm})^{\ox k+1}$ is surjective. To show that $d_{K^p}$ is surjective, we show that there exists a covering of $\check\fl$ by affinoid open subset $V\subset \check{\fl}$ such that the map
\begin{equation}\label{equationdGammaV}
    d_{\calS_{\Gamma}}:\calO_{\calS_{\Gamma}}^{\lan,(0,-k)}(\pi_{\Gamma,\HT}^{-1}(V))\rightarrow (\calO_{\calS_{\Gamma}}^{\lan,(0,-k)}\ox_{\calO_{\calS_{\Gamma}}^{\sm}}(\Omega_{K^p}^{1,\sm}|_{\calS_{\Gamma}})^{\ox k+1})(\pi_{\Gamma,\HT}^{-1}(V))    
\end{equation}
is surjective for any $\Gamma$ satisfying the assumption of Proposition \ref{prop:prGamma}. Let $V\subset \check{\fl}$ be an affinoid open subset such that 
\[\pi_{\LT,\GM,0}^{-1}(V)=\sqcup_{g\in \GL_2(\bbQ_p)/\GL_2(\bbZ_p)}gU,\] 
where $\pi_{\LT,\GM,0}:\calM_{\LT,0}\to \check{\fl}$ is the Gross--Hopkins map discussed in the end of \S\ref{subsection:twotowers} and $U\to \calM_{\LT,0}$ is an affinoid open subset such that 
$\pi_{\LT,\GM,0}|_{U}$ is finite \'etale. We claim that for such open subset $V$ the map (\ref{equationdGammaV}) is surjective.

By Corollary \ref{cor:kerdDr}, the differential operator $d_{\rm Dr}$ is surjective on $\calM_{\Dr,\infty}$, with kernel $\calO_{\calM_{\LT,\infty}}^{\lalg,(0,-k)}$. For an affinoid open $U_\infty\subset \calM_{\Dr,\infty}\simeq \calM_{\LT,\infty}$ that is the pullback of an open affinoid subset of $\calM_{\LT,0}$, we have $H^1(U_\infty,\calO_{\calM_{\LT,\infty}}^{\lalg,(0,-k)})=0$ using Tate's acyclicity (see the proof of Lemma \ref{lem:RpOLTsm} below). We see that the differential
\begin{align}\label{equationdDrUinfty}
 d_{\Dr}:\calO_{\calM_{\Dr,\infty}}^{\lan,(0,-k)}(U_\infty)\to (\calO_{\calM_{\Dr,\infty}}^{\lan,(0,-k)}\ox_{\calO_{\calM_{\Dr,\infty}}^{\sm}}(\Omega_{\calM_{\Dr,\infty}}^{1,\sm})^{\ox k+1})(U_\infty)
\end{align}
is still surjective after taking sections on such affinoid open subsets $U_{\infty}$. 

Let $U_\infty$ be the pre-image of $U$ inside $\calM_{\LT,\infty}$ via $\calM_{\LT,\infty}\to \calM_{\LT,0}$. Then $\GL_2(\bbQ_p)$-equivariantly
\[\pi_{\LT,\GM}^{-1}(V)=\sqcup_{g\in \GL_2(\bbQ_p)/\GL_2(\bbZ_p)}gU_\infty.\] Let $\Gamma\subset \GL_2(\bbQ_p)$ be a discrete cocompact torsion-free subgroup. The assumption on $\Gamma$ implies that $\Gamma\cap c\GL_2(\bbZ_p)c^{-1}=\{1\}$ for any $c\in \GL_2(\bbQ_p)$. Hence 
\begin{align*}
    \pi_{\LT,\GM}^{-1}(V)=\sqcup_{g\in \GL_2(\bbQ_p)/\GL_2(\bbZ_p)}gU_\infty=\sqcup_{\gamma\in\Gamma} \sqcup_{c\in \Gamma\backslash\GL_2(\bbQ_p)/\GL_2(\bbZ_p)}\gamma c U_\infty.
\end{align*}
Therefore, $\pi_{\Gamma,\HT}^{-1}(V)=\pr_{\Gamma,\infty}\pi_{\LT,\GM}^{-1}(V)\isom  \sqcup_{c\in \Gamma\backslash \GL_2(\bbQ_p)/\GL_2(\bbZ_p)}c\cdot  U_\infty$ is isomorphic to a finite copy of $U_{\infty}$. As a consequence, the map (\ref{equationdGammaV}) can be identified with a finite direct sum of copies of the map (\ref{equationdDrUinfty}). In particular, the map (\ref{equationdGammaV}) is surjective, equivalently, the map 
\begin{align*}
    d_{K^p}:\calO_{K^p}^{\lan,(0,-k)}(V)\to (\calO_{K^p}^{\lan,(0,-k)}\ox_{\calO_{K^p}^{\sm}}(\Omega_{{K^p}}^{1,\sm})^{\ox k+1})(V)
\end{align*}
is surjective.

Since $\check{\fl}$ can be covered by $V$ such that (\ref{equationdGammaV}) is surjective and $\calS_{K^p}$ is a disjoint union of finitely many components of the form $\calS_{\Gamma}$, we conclude that $d_{K^p}$ is a surjective map of sheaves. Finally, as $d'_{K^p}$ is a twist of $d_{K^p}$, we deduce that $d'_{K^p}$ is also surjective.
\end{proof}

\subsubsection{The intertwining operator on $\calO_{K^p}^{\lan,(0,-k)}$}
Let $k\geq 0$. We define the intertwining operator on $\calO_{K^p}^{\lan,(0,-k)}$ as \cite[Definition 4.3.1]{PanII}.

\begin{definition}\label{def:intertwiningoperator}
We denote $I:=d'_{K^p}\comp\bar d_{K^p}=\bar d'_{K^p}\comp d_{K^p}$ to be the composition of these operators in (\ref{equationdKv}), and we call $I$ the intertwining operator.
\end{definition}

In the same way as \cite[\S 6]{PanII}, we can interpret the intertwining operator Galois-theoretically as the Fontaine operator \cite{fontaine2004arithmetique}. We sketch the construction of the Fontaine operator. Let $\bbB_{\dR,\calS_{K^p}}^+$ be the positive de Rham period sheaf on $\calS_{K^p}$ (we mean the slice of the pro-\'etale sheaf on the analytic site of $\calS_{K^p}$). Define $\bbB_{\dR}^+:=\pi_{K^p,\HT,*}\bbB_{\dR,\calS_{K^p}}$, which has a decreasing filtration given by $\Fil^i\bbB_{\dR}^+=\pi_{K^p,\HT,*}\Fil^i\bbB_{\dR,\calS_{K^p}}$, with graded pieces $\gr^i\bbB_{\dR}^+\isom \calO_{K^p}(i)$. Here we use $R^j\pi_{K^p,\HT,*}\calO_{\calS_{K^p}}=0$ for $j\ge 1$. For $k\ge 0$ an integer, put $\bbB_{\dR,k}^+:=\bbB_{\dR}^+/\Fil^k\bbB_{\dR}^+$. For $(a,b)\in\bbZ^2$, let $\tilde\chi_{(a,b)}$ be the character of $Z(\check\frg)$ such that $Z(\check\frg)$ acts on the Verma module $U(\check\frg)\ox_{U(\check \frb)}(a,b)$ via $\tilde\chi_{(a,b)}$. There is an induced filtration on $\bbB_{\dR,k+2}^{+,\lan,\tilde\chi_{(0,-k)}}$ from $\bbB_{\dR}^{+}$ and one can show that $\bbB_{\dR,k+2}^{+,\lan,\tilde\chi_{(0,-k)}}$ has $i$-th graded pieces given by $\calO_{K^p}^{\lan,\tilde\chi_{(0,-k)}}(i)$, where $i=0,1,...,k+1$. Besides, the arithmetic Sen operator $\Theta_{\Sen}$ acts on $\calO_{K^p}^{\lan,\tilde\chi_{(0,-k)}}(i)$ via $i,-k-1+i$. In \cite[\S 6]{PanII}, Pan developed a Sen theory for colimits of Banach spaces which are also $B_{\dR}^+/\Fil^k B_{\dR}^+$-modules. For an open affinioid $U\subset \check\fl$ as in Lemma \ref{lem:prGammaHT}, we may assume there exists a finite extension $K$ of $\bbQ_p$ such that $U$ is defined over $K$. Let $\zeta_{p^n}$ be a primitive $p^n$-th root of unity, and let $K_\infty\subset \dlim_nK(\zeta_{p^n})$ be the maximal $\bbZ_p$-extension. For $W$ an LB $B_{\dR}^+/\Fil^k B_{\dR}^+$-module, let $W^K\subset W$ be the subspace of $\Gal(\bar K/K_\infty)$-invariant, $\Gal(K_\infty/K)$-analytic vectors in $W$. It turns out $\bbB_{\dR,k+2}^{+,\lan,\tilde\chi_{(0,-k)}}(U)$ satisfies the assumption in \cite[6.1.17]{PanII}, and $\bbB_{\dR,k+2}^{+,\lan,\tilde\chi_{(0,-k)}}(U)^K$ has $i$-th graded pieces given by $\calO_{K^p}^{\lan,\tilde\chi_{(0,-k)}}(U)^K(i)$. Then the action of $\Lie\Gamma=\Lie\Gal(K_\infty/K)$ induces a map 
\begin{align}
    \calO_{K^p}^{\lan,(0,-k)}(U)^K\to \calO_{K^p}^{\lan,(-k-1,1)}(U)^K(k+1).
\end{align}
This map is functorial, and after we base change to $C$ it induces a natural sheaf morphism 
\begin{align}\label{eq:fontaineoperator}
    N:\calO_{K^p}^{\lan,(0,-k)}\to \calO_{K^p}^{\lan,(-k-1,1)}(k+1)
\end{align}
which we call it the Fontaine operator. Roughly speaking, the Fontaine operator measures the nilpotency of the arithmetic Sen action on $\bbB_{\dR,k+2}^{+,\lan,\tilde{\chi}_{(0,-k)}}$. Using the machinery developed in \cite[\S 6]{PanII}, similar arguments as for \cite[Theorem 6.2.6]{PanII} give the following result.
\begin{theorem}\label{thm:NequalsI}
Up to an invertible constant, the Fontaine operator $N$ equals to the intertwining operator $I$.
\end{theorem}

\subsubsection{Completed cohomology of quaternionic Shimura curves}\label{subsectioncompletedcohomology}
Recall that $S_{K^pK_p}$ is the quaternionic Shimura curve of level $K^pK_p\subset G(\bbA^\infty)$, which is a complete algebraic curve over $\bbQ$. The completed cohomology groups associated to the tower $\{S_{K^pK_p}\}_{K_p\subset D_p^\times}$ are given by
\begin{align*}
    \tilde{H}^i(K^p,\bbQ_p)&:=(\ilim_{n}\dlim_{K_p\subset D_p^\times}H^i(S_{K^pK_p}(\bbC),\bbZ/p^n\bbZ))[\frac{1}{p}]\\
    &\simeq(\ilim_{n}\dlim_{K_p\subset D_p^\times}H^i_{\et}(S_{K^pK_p,\overline{\bbQ}},\bbZ/p^n\bbZ))[\frac{1}{p}]
\end{align*}
for $i\geq 0$, which carry commuting actions of $D_p^\times$, the Galois group $\Gal(\bar \bbQ/\bbQ)$ and the Hecke algebra $\bbT^S$. Note that $\tilde{H}^1(K^p,\bbQ_p)=\widehat{H}^1(K^p,\bbQ_p)$ is a unitary Banach representation of $D_p^\times$ (cf. \cite[\S 4.1]{Emerton2006interpolation}). Given any Banach space $W$ over $\bbQ_p$, let 
\[
    \tilde{H}^i(K^p,W):=\tilde{H}^i(K^p,\bbQ_p)\hat\ox_{\bbQ_p}W.
\]
Recall that $\calO_{K^p}=\pi_{K^p,\HT,*}\calO_{\calS_{K^p}}$ with $\calO_{\calS_{K^p}}$ the completed structure sheaf on $\calS_{K^p}$, and $\calO_{K^p}^{\lan}$ the subsheaf consisting of locally analytic sections. 
\begin{theorem}\label{thm:completedcohomology}
For $i\ge 0$, we have $D_p^\times\times \Gal(\overline{\bbQ_p}/\bbQ_p)\times \bbT^S$-equivariant isomorphisms 
\begin{align*}
    \tilde{H}^i(K^p,C)&\isom H^i(\check\fl,\calO_{K^p})\\
    \tilde{H}^i(K^p,C)^{\lan}&\isom H^i(\check\fl,\calO_{K^p}^{\lan}).
\end{align*}

\begin{proof}
The first one is the primitive comparison theorem \cite{scholze_p-adic_2013}, cf. \cite[\S 4.2]{Sch15}. The second one follows from \cite[Theorem 4.4.6]{Pan22} or \cite[Theorem 1.1.8]{camargo2024locallyanalyticcompletedcohomology}. 
\end{proof}
\end{theorem}

Let $E$ be a finite extension of $\bbQ_p$. For any compact open subgroup $K_p\subset D_p^\times$, we define $\bbT(K^pK_p)\subset \End_{\bbZ_p}(H^1(S_{K^pK_p}(\bbC),\bbZ_p))$ as the image of $\bbT^S$. Then we set
\begin{align*}
    \bbT(K^p):=\ilim_{K_p}\bbT(K^pK_p).    
\end{align*}
It acts faithfully on $\tilde{H}^1(K^p,E)$ and commutes with the action of $\Gal(\bar \bbQ/\bbQ)\times D_p^\times$. Moreover, the action of $\bbT(K^p)$ and $\Gal(\bar \bbQ/\bbQ)$ on $\tilde{H}^1(K^p,E)$ satisfy the usual Eichler-Shimura congruence relation: 
\begin{align}\label{equationcongruencerelation}
    \mathrm{Frob}_{\ell}^2 - T_{\ell}\circ\mathrm{Frob}_{\ell} + \ell S_{\ell}=0
\end{align}
where $\ell\notin S$ is a prime number, $\mathrm{Frob}_{\ell}$ denotes the geometric Frobenius at $\ell$ and 
\[T_{\ell}=[\GL_2(\bbZ_{\ell})\begin{pmatrix}
\ell & 0 \\ 0 & 1
\end{pmatrix}\GL_2(\bbZ_{\ell})],S_{\ell}=[\GL_2(\bbZ_{\ell})\begin{pmatrix}
\ell & 0 \\ 0 & \ell
\end{pmatrix}\GL_2(\bbZ_{\ell})]\] 
are the usual Hecke operators (cf. \cite[\S 10.3, Remarque 0.3]{carayol1986mauvaise}, while our normalization follows those in \cite{emerton2011local,Pan22,PanII} and the choice of the $G(\bbR)$-conjugacy class $h$ in \S\ref{subsec:quaternionicshimuracurve}, hence the Galois action (the canonical model), is different from those in \cite{carayol1986mauvaise}).

Let $\rho$ be a $2$-dimensional continuous $E$-linear absolutely irreducible representation of $\Gal(\bar \bbQ/\bbQ)$. Set 
\begin{align}\label{eq:defofTrho}
    \check\Pi(\rho):=\Hom_{\Gal(\bar \bbQ/\bbQ)}(\rho,\tilde{H}^1(K^p,E)).
\end{align}
Suppose that $\check\Pi(\rho)\neq 0$. This is a unitary Banach representation of $D_p^\times$. From $\rho$, we can associate a character $\lambda:\bbT(K^p)\to E$ using the congruence relation (\ref{equationcongruencerelation}), such that (cf. \cite[Prop. 6.1.12]{emerton2011local})
\begin{align*}
    \Hom_{\Gal(\bar \bbQ/\bbQ)}(\rho,\tilde{H}^1(K^p,E))=\Hom_{\Gal(\bar \bbQ/\bbQ)}(\rho,\tilde{H}^1(K^p,E)[\lambda]).
\end{align*}
Moreover, by the main result of \cite{BLR}, as $\rho$ is $2$-dimensional and absolutely irreducible, we know that $\tilde{H}^1(K^p,E)[\lambda]$ is $\rho$-isotypic: 
\begin{align}\label{eq:rhoisotypic}
    \tilde{H}^1(K^p,E)[\lambda]\isom \rho\ox_E \check\Pi(\rho).
\end{align}

Similar to the treatment in \cite[Theorem 7.2.2]{PanII} and \cite[Theorem 5.4.2]{QS24}, we get the following result on a first step towards the description of the locally analytic part of the isotypic space $\check\Pi(\rho)$. Let $I^1$ be the operator induced by the intertwining operator $I$ on the level of cohomology groups 
\begin{align*}
    I^1:H^1(\check\fl,\calO_{K^p}^{\lan,(0,-k)})\to H^1(\check\fl,\calO_{K^p}^{\lan,(-k-1,1)}(k+1)).
\end{align*}
\begin{theorem}\label{thm:PicheckrhoLankerI1}
Let $\rho$ be a $2$-dimensional $E$-linear continuous absolutely irreducible representation of $\Gal(\bar \bbQ/\bbQ)$. Suppose:
\begin{enumerate}[(i)]
    \item $\rho$ appears in $\tilde{H}^1(K^p,E)^{\lan}$, namely $\Hom_{\Gal(\bar \bbQ/\bbQ)}(\rho,\tilde{H}^1(K^p,E)^{\lan})\neq 0$.
    \item $\rho|_{\Gal(\overline{\bbQ_p}/\bbQ_p)}$ is de Rham of Hodge--Tate weight $0,k+1$ with $k\in \bbZ_{\ge 0}$.
\end{enumerate}
Then after fixing an embedding $E\hookrightarrow C$, we have
\begin{align*}
    \check{\Pi}(\rho)^{\lan}\widehat{\ox}_{E}C\isom \ker I^1[\lambda].
\end{align*}
where $\lambda:\bbT^S\rightarrow C$ is the Hecke eigensystem associated to $\rho$.
\end{theorem}
\begin{proof}
    The proof is literally the same as in \cite[Theorem 7.2.2]{PanII} or \cite[Theorem 5.4.2]{QS24}. Roughly speaking, the regular de Rhamness condition of $\rho$ and the assumption $\check\Pi(\rho)^{\lan}\neq 0$ ensures that the multiplicity space $\check\Pi(\rho)^{\lan}$ is killed by the Fontaine operator (\ref{eq:fontaineoperator}), which can be identified with the intertwining operator using Theorem \ref{thm:NequalsI}.
\end{proof}
\subsection{A product formula}\label{sec:productformula}
Motivated by Theorem \ref{thm:PicheckrhoLankerI1}, we proceed to study the cohomology of the intertwining operator $I$ on $\calO_{K^p}^{(0,-k)}$.  Let $k\geq 0$ and let
\begin{align*}
    d_{K^p}:\calO_{K^p}^{\lan,(0,-k)}\to \calO_{K^p}^{\lan,(0,-k)}\ox_{\calO_{K^p}^{\sm}}(\Omega_{K^p}^{1,\sm})^{\ox k+1}
\end{align*}
be the differential map we constructed in \S\ref{sec:passageShimura}. Being part of the intertwining operator for which we need to understand the cohomology, we will need to compute the cohomology of $\ker d_{K^p}$ and determine the Hecke action on it (Proposition \ref{prop:productformula}). 

The lemma below shows that the cohomology of $\ker d_{\calS_{K^p}}$ on $\calS_{K^p}$ is the same as the cohomology of $\ker d_{K^p}$ on $\check\fl$.
\begin{lemma}\label{lemmakerdKvdSKv}
There is a natural isomorphism 
\begin{align}
    R\Gamma(S_{K^p},\ker d_{\calS_{K^p}})\isom R\Gamma(\check{\fl},\ker d_{K^p}).
\end{align}
\begin{proof}
By Proposition \ref{prop:dKvsurj} (2), we know $\calO_{K^p}^{\lan,(0,-k)}\stackrel{d_{K^p}}{\to} \calO_{K^p}^{\lan,(0,-k)}\ox_{\calO_{K^p}^{\sm}}(\Omega_{K^p}^{1,\sm})^{\ox k+1}$ is surjective. Hence the sheaves $\ker d_{\calS_{K^p}},\ker d_{K^p}$ admit resolutions: 
\begin{align*}
    &0\to \ker d_{\calS_{K^p}}\to \calO_{\calS_{K^p}}^{\lan,(0,-k)}\stackrel{d_{\calS_{K^p}}}{\to} \calO_{\calS_{K^p}}^{\lan,(0,-k)}\ox_{\calO_{\calS_{K^p}}^{\sm}}(\Omega_{\calS_{K^p}}^{1,\sm})^{\ox k+1}\to 0,\\
    &0\to \ker d_{K^p}\to \calO_{K^p}^{\lan,(0,-k)}\stackrel{d_{K^p}}{\to} \calO_{K^p}^{\lan,(0,-k)}\ox_{\calO_{K^p}^{\sm}}(\Omega_{K^p}^{1,\sm})^{\ox k+1}\to 0.
\end{align*}
By definition, $\pi_{K^p,\HT,*}\calO_{\calS_{K^p}}^{\lan,(0,-k)}= \calO_{K^p}^{\lan,(0,-k)}$. Also $R^j\pi_{K^p,\HT,*}\calO_{\calS_{K^p}}^{\lan,(0,-k)}=0$ for $j\ge 1$, the proof of which is similar to the one in \cite[Proposition 3.2.8]{PanII}. Indeed, using explicit power series expansions of $\calO_{\calS_{K^p}}^{\lan,(0,-k)}$, one can locally exhibit $\calO_{\calS_{K^p}}^{\lan,(0,-k)}$ as colimits of sheaves isomorphic to ``direct products'' of $\omega_{\calS_{K^p}}^{(0,-k),\sm}$. More precisely, for $U\subset \calS_{K^p}$ a good affinoid perfectoid open subset (which comes from the pullback of an affinioid open subset from finite level and an affinioid open subset from the flag variety along $\pi_{\HT}$), then $\calO_{\calS_{K^p}}^{\lan,(0,-k)}(U)$ can be written as a filtered colimit (with injective transition maps) $\dlim_n A^n$, with each $A^n\isom (\prod^\infty_{i=0}\calO^+_{\calS_{K^pK_n'}}(U_{K_n'}))\ox_{\bbZ_p}\bbQ_p$ (for some open compact subgroup $K_n'\subset D_p^\times$ and $U_{K_n'}$ is an affinioid open in $\calS_{K^pK_n'}$ whose preimage is $U$) by sending the power series expansion to its coefficients. Then one deduces the desired vanishing result using Tate's acyclicity of coherent sheaves on affinioid open subsets. Similar results are valid for the third terms in the above two exact sequences. Hence $R\pi_{K^p,\HT,*}\ker d_{\calS_{K^p}}=\ker d_{K^p}[0]$ and the Leray spectral sequence degenerates, giving the desired isomorphism. 
\end{proof}
\end{lemma}

Instead of computing $\ker d_{K^p}$ directly, using the $p$-adic uniformization, we first pass these sheaves from $\calS_{K^p}$ to $\calM_{\Dr,\infty}$. Let $\Gamma\subset \GL_2(\bbQ_p)$ be a cocompact torsion-free discrete subgroup appearing in (\ref{uniformization}). Let $d_{\calS_{\Gamma}}:=d_{\calS_{K^p}}|_{\calS_{\Gamma}}$ and similarly $d_{\calS_{\Gamma}}':=d_{\calS_{K^p}}'|_{\calS_{\Gamma}}$ where $\calS_{\Gamma}=\calM_{\Dr,\infty}/\Gamma$. By the construction and an analogue of Proposition \ref{prop:Wlalg}, we have 
\begin{equation}\label{equationkerdDr}
    \pr_{\Gamma,\infty}^{-1}\ker d_{\calS_{\Gamma}}=\ker d_{\Dr}=\calO_{\calM_{\LT,\infty}}^{\lalg,(0,-k)}=V^{(0,-k)}\otimes\omega_{\cal_{\calM_{\LT,\infty}}}^{(0,-k),\sm}.
\end{equation}
Then we descend via the $\Gamma$-torsor $\pr_{\Gamma,\infty}:\calM_{\Dr,\infty}\surj \calS_{\Gamma}$ and exhibit the results using $\Gamma$-(co)homology for cocompact torsion-free discrete subgroups $\Gamma\subset \GL_2(\bbQ_p)$ (Lemma \ref{Lem:RGammacLTsmGamma} below). We find that it is better to descend compactly supported cohomology instead of cohomology of sheaves on $\calM_{\Dr,\infty}$, partly because that the cover $\calM_{\Dr,\infty}\to  \calS_{\Gamma}$ is an ``infinite-sheet'' \'etale covering map. In order to handle the compactly supported cohomology groups, we develop some results on proper pushforwards of abelian sheaves on adic spaces in Appendix \ref{section:properpushforward}.

\begin{lemma}\label{Lem:RGammacLTsmGamma}
    The $\Gamma$-covering $\calM_{\LT,\infty}\simeq \calM_{\Dr,\infty}\rightarrow \calM_{\Dr,\infty}/\Gamma$ induces $D_p^{\times}$-isomorphisms
    \begin{align*}
        R\Gamma_c(\calM_{\LT,\infty},\calO_{\calM_{\LT,\infty}}^{\lalg,(0,-k)})\otimes_{\bbZ[\Gamma]}^L\bbZ&\simeq R\Gamma(\calS_{\Gamma},\ker d_{\calS_{\Gamma}}),\\
        R\Gamma_c(\calM_{\LT,\infty},\omega_{\calM_{\LT,\infty}}^{(-k-1,k+1),\lalg,(0,-k)})\otimes_{\bbZ[\Gamma]}^L\bbZ&\simeq R\Gamma(\calS_{\Gamma},\ker d_{\calS_{\Gamma}}').
    \end{align*}
\end{lemma}
\begin{proof}
    Let $*/\Gamma$ is the classifying stack of the group $\Gamma$, let $p:\calM_{\LT,\infty}/\Gamma\to */\Gamma$  and $f:*/\Gamma\to *$ be the natural maps. The desired isomorphism for $\ker d_{\calS_{\Gamma}}$ arises from that the composition of ``proper pushforwards":
    \[Rf_! Rp_! \calO_{\calM_{\LT,\infty}}^{\lalg,(0,-k)}=R(f\circ p)_!\calO_{\calM_{\LT,\infty}}^{\lalg,(0,-k)}.\]
    For a rigorous proof, we apply Proposition \ref{prop:Gammaequivariantpushforward} for $\sF=\ker d_{\calS_{\Gamma}}, X=\calM_{\Dr,\infty}, Y=\calM_{\Dr,\infty}/\Gamma$ and $g=f\circ p$. Note that $\Gamma$ has finite cohomological dimension \cite[\S 6.1]{borel1976cohomologie}.
    
    The map $d'_{\Dr}$ is a twist of $d_{\Dr}$ by $ \pi_{\Dr,\HT}^{-1}(\Omega_{\check\fl}^1)^{\otimes k+1}=\pi_{\LT,\GM}^{-1}(\Omega_{\check\fl}^1)^{\otimes k+1}=\pi_{\LT,\GM}^{-1}\omega_{\check\fl}^{ (-k-1,k+1)}$ over $\pi_{\LT,\GM}^{-1}\calO_{\check\fl}$ by (iv) of Theorem \ref{theoremdifferentialoperatorsDr}. Hence 
    \[\ker d'_{\Dr}=\calO_{\calM_{\LT,\infty}}^{\lalg,(0,-k)}\otimes_{\pi_{\LT,\GM}^{-1}\calO_{\check\fl}}\pi_{\LT,\GM}^{-1}\omega_{\check\fl}^{ (-k-1,k+1)}=\omega_{\calM_{\LT,\infty}}^{(-k-1,k+1),\lalg,(0,-k)}.\]
    The result for $\ker d_{\calS_{\Gamma}}'$ follows similarly using Proposition \ref{prop:Gammaequivariantpushforward} again. 
\end{proof}
When $k=0$, we have $\calO_{\calM_{\LT,\infty}}^{\lalg,(0,-k)}=\calO_{\calM_{\LT,\infty}}^{\sm} $ and $\omega_{\calM_{\LT,\infty}}^{(-k-1,k+1),\lalg,(0,-k)}=\omega_{\calM_{\LT,\infty}}^{(-1,1),\sm}=\Omega^{1,\sm}_{\calM_{\LT,\infty}}$. The cohomologies $R\Gamma_c(\calM_{\LT,\infty},\calO_{\calM_{\LT,\infty}}^{\sm})$ and $R\Gamma_c(\calM_{\LT,\infty},\omega_{\calM_{\LT,\infty}}^{(-1,1),\sm})$ are (derived) smooth representations of $G=\GL_2(\bbQ_p)$ by Lemma \ref{lem:RpOLTsm} below. Hence they are naturally (left) modules over the smooth Hecke algebra $\calH(G)=\dlim_K\calH(G,K)$, where $K$ runs through open compact subgroups of $G$ and $\calH(G,K)=\calC^{\sm}_c(K\bs G/K,C)$. See for example \cite[II.3.12, III.1.4]{renard2010representations} or \cite[\S 5.1]{mantovan2004certain}. 

For general $k\geq 0$, by (\ref{equationkerdDr}), 
\begin{align}\label{equationRGammackerdDrlocallyalgebraic}
    R\Gamma_c(\calM_{\LT,\infty},\ker d_{\Dr})&=R\Gamma_c(\calM_{\LT,\infty},\omega_{\calM_{\LT,\infty}}^{(0,-k),\sm})\otimes_C V^{(0,-k)}\\
    R\Gamma_c(\calM_{\LT,\infty},\ker d_{\Dr}')&=R\Gamma_c(\calM_{\LT,\infty},\omega_{\calM_{\LT,\infty}}^{(-k-1,1),\sm})\otimes_C V^{(0,-k)}    
\end{align} 
are locally $V^{(0,-k)}$-algebraic representations of $G$. 

\begin{lemma}\label{lem:RpOLTsm}
For $(a,b)\in\bbZ^2$, we have 
\begin{align*}
    R\Gamma_c(\calM_{\LT,\infty},\omega_{\calM_{\LT,\infty}}^{(a,b),\sm})\isom \dlim_n R\Gamma_c(\calM_{\LT,n},\omega^{(a,b)}_{\calM_{\LT,n}})=\dlim_n H^1_c(\calM_{\LT,n},\omega^{(a,b)}_{\calM_{\LT,n}})[-1].
\end{align*}
\begin{proof}
Recall that $\omega_{\calM_{\LT,\infty}}^{(a,b),\sm}$ can be equivalently defined as follows: given an affinoid open subset $U_\infty\subset \calM_{\LT,\infty}$, which we assume that there is a sufficiently large $n$ such that $U_\infty$ is the pre-image of an affinoid open subset $U_n\subset \calM_{\LT,n}$. Then $\omega_{\calM_{\LT,\infty}}^{(a,b),\sm}(U_{\infty})=\dlim_{m\ge n} \omega_{\calM_{\LT,m}}^{(a,b)}(\pi_{m,n}^{-1}(U))$ with $\pi_{m,n}:\calM_{\LT,m}\to \calM_{\LT,n}$ the natural projection. Let $U_0\subset \calM_{\LT,0}$ be an affinoid open subset, and we denote its pre-image in $\calM_{\LT,n}$ being $U_n$ and its pre-image in $\calM_{\LT,\infty}$ being $U_\infty$. By \cite[Proposition 14.9]{scholze2022etale} applied to $Y_0=U_n$ with $Y_i$'s being the \'etale coverings $\{U_m\}_{m\ge n}$ and $\calF_0=\omega_{\calM_{\LT,n}}^{(a,b)}|_{U_n}$. Then 
\begin{align*}
    R\Gamma(U_\infty,\omega_{\calM_{\LT,\infty}}^{(a,b),\sm})&=R\Gamma(U_\infty,\dlim_n \pi_n^{-1}\omega_{\calM_{\LT,n}}^{(a,b)})=\dlim_n R\Gamma(U_\infty,\pi_n^{-1}\omega_{\calM_{\LT,n}}^{(a,b)})\\
    &=\dlim_n\dlim_{m\ge n}R\Gamma(U_m,\pi_{m,n}^{-1}\omega_{\calM_{\LT,n}}^{(a,b)})= \dlim_m \dlim_{n\le m}R \Gamma(U_n,\pi_{m,n}^{-1}\omega_{\calM_{\LT,n}}^{(a,b)})=\dlim_m R\Gamma(U_m,\omega_{\calM_{\LT,m}}^{(a,b)})
\end{align*}
where $\pi_n$ is the natural projection $\calM_{\LT,\infty}\to \calM_{\LT,n}$. In particular, this implies that $R\pi_{0,*}\omega_{\calM_{\LT,\infty}}^{(a,b),\sm}\isom \dlim_nR\pi_{n,0,*}\omega_{\calM_{\LT,n}}^{(a,b)}$.

Next we verify the identification of compactly supported cohomologies. By definition and Proposition \ref{prop:compositionofexceptionalpushforward}, 
\begin{align*}
    R\Gamma_c(\calM_{\LT,\infty},\omega_{\calM_{\LT,\infty}}^{(a,b),\sm})=R\Gamma(\check{\fl},R\pi_{\LT,\GM,!}\omega_{\calM_{\LT,\infty}}^{(a,b),\sm})
\end{align*}
where $\pi_{\LT,\GM}:\calM_{\LT,\infty}\to \check{\fl}$ is the composition of the projection $\pi_{0}:\calM_{\LT,\infty}\to \calM_{\LT,0}$ and the Gross--Hopkins period map $\pi_{\GM,0}:\calM_{\LT,0}\to \check{\fl}$. The map $\pi_{\GM,0}$ is partially proper, the map $\pi_{0}$ is proper, and $\check{\fl}$ is proper. Hence
\begin{align*}
    R\Gamma(\check{\fl},R\pi_{\LT,\GM,!}\omega_{\calM_{\LT,\infty}}^{(a,b),\sm})&=R\Gamma(\check{\fl},R\pi_{\GM,0,!}R\pi_{0,*}\omega_{\calM_{\LT,\infty}}^{(a,b),\sm})\\
    &=R\Gamma(\check{\fl},R\pi_{\GM,0,!}\dlim_n R\pi_{n,0,*}\omega_{\calM_{\LT,n}}^{(a,b)})\\
    &=\dlim_n R\Gamma(\check{\fl},R\pi_{\GM,0,!} R\pi_{n,0,*}\omega_{\calM_{\LT,n}}^{(a,b)})\\
    &=\dlim_n R\Gamma(\check{\fl},R\pi_{\GM,n,!} \omega_{\calM_{\LT,n}}^{(a,b)})\\
    &=\dlim_n R\Gamma_c(\calM_{\LT,n},\omega_{\calM_{\LT,n}}^{(a,b)}),
\end{align*}
where for the third isomorphism above we used (3) of Lemma \ref{lemmalowershriekcolimit}. Finally, the last isomorphism in the statement follows from the higher vanishing of the coherent comologies of the Stein spaces $\calM_{\LT,n}$ and the Serre duality.
\end{proof}
\end{lemma}

We now temporarily suspend our discussion on Lubin--Tate spaces and turn our attention to the Shimura set that arises in the $p$-adic uniformization theorem. First we introduce the space of classical algebraic automorphic forms for $\bar G$ in \S\ref{sec:uniformization} with fixed weights and of the level $K^p$ outside $p$ (cf. \cite{gross1999algebraic} and \cite[\S 3]{loeffler2011overconvergent}). 
\begin{definition}\label{def:algebraicautomorphicform}
Suppose that $V$ is an irreducible algebraic representation of $\bar G_p=G=\GL_2(\bbQ_p)$ over $C$. Let $\calA^{K^p}_{\bar G,V}$ denote the set of maps
\begin{align}\label{eq:defofAKvGk}
    f:Z_{K^p}=\bar G(\bbQ)\bs\bar G(\bbA_f)/K^p  \to V
\end{align}
such that $f(zg_p)=g_p^{-1}f(z)$ for any $z\in Z_{K^p}$ and all $g_p$ in some open compact subgroup of $\bar G_p$. When $V=V^{(0,-k)}$, we write $\calA^{K^p}_{\bar G,k}:=\calA^{K^p}_{\bar G,V^{(0,-k)}}$.    
\end{definition}
If $k=0$, $\calA^{K^p}_{\bar G,0}=\calC^{\sm}(Z_{K^p},C)$ is the space of $C$-valued  locally constant functions on the $p$-adic manifold $Z_{K^p}$. The space $\calA^{K^p}_{\bar G,k}$ admits a smooth action of $\bar G_p$ via $(g_p.f)(z)=g_p.f(zg_p)$. Both $\calA^{K^p}_{\bar G,k}$ and $R\Gamma_c(\calS_{K^p},\ker d_{\calS_{K^p}})$ are equipped with the action of the Hecke algebra $\bbT^S$ defined in \S\ref{subsec:quaternionicshimuracurve} via the isomorphism $G(\bbA_f^p)\simeq \bar G(\bbA_f^p)$ and the Hecke correspondences.

\begin{remark}
    We recall the definiton of the action of $\bbT^S$ on $R\Gamma_c(\calS_{K^p},\ker d_{\calS_{K^p}})$. For $g\in \bar G(\bbA_f^S)$ and the Hecke operator $[K^SgK^S]$, we consider the Hecke correspondence 
    \begin{center}
        \begin{tikzcd}
            &\calS_{K^p\cap gK^pg^{-1}}\stackrel{\times g}{\rightarrow} \calS_{g^{-1}K^pg\cap K^p}\arrow[rd,"h_2"]\arrow[ld,"h_1"]&\\
            \calS_{K^p}& &\calS_{K^p}.
        \end{tikzcd}
    \end{center}
    Via the $p$-adic uniformization of $\calS_{K^p}$ (\ref{uniformization}), the maps $h_1,h_2$ are unions of the maps of the form $\Gamma'_{x'}\backslash\calM_{\Dr,\infty}\rightarrow \Gamma_x\backslash\calM_{\Dr,\infty}$ where $(K^p)'=K^p\cap gK^pg^{-1}$ or $g^{-1}K^pg\cap K^p$, $x'\in Z_{(K^p)'}$, $x$ is the image of $x'$ under the projection $Z_{(K^p)'}\rightarrow Z_{K^p}$, $\Gamma_x= \bar G (\bbQ)\cap x^{-1}K^px\supset \Gamma_{x'}:= \bar G (\bbQ)\cap (x')^{-1}(K^p)'x'$. By Proposition \ref{prop:prGamma}, both $h_1,h_2$ are analytic finite coverings. We have an isomorphism $(\times g)^{-1}\ker d_{\calS_{g^{-1}K^pg\cap K^p}}\stackrel{\sim}{\rightarrow} \ker d_{\calS_{K^p\cap gK^p}g^{-1}}$. The action of the Hecke operator $[K^SgK^S]$ is the composition
    \begin{align}\label{eq:HeckeonRGammakerdS}
    R\Gamma(\calS_{K^p},\ker d_{\calS_{K^p}})\stackrel{}{\rightarrow} R\Gamma(\calS_{K^p},h_{1,*}\ker d_{\calS_{K^p\cap gK^pg^{-1}}})=R\Gamma(\calS_{K^p\cap gK^pg^{-1}},\ker d_{\calS_{K^p\cap gK^pg^{-1}}})\nonumber \\
    \stackrel{\sim}{\rightarrow} R\Gamma(\calS_{g^{-1}K^pg\cap K^p},\ker d_{\calS_{g^{-1}K^pg\cap K^p}})=R\Gamma(\calS_{K^p},h_{2,!}\ker d_{\calS_{g^{-1}K^pg\cap K^p}})\rightarrow R\Gamma(\calS_{K^p},\ker d_{\calS_{K^p}})
    \end{align}
    where the first and last maps are respectively induced by the natural maps $\ker d_{\calS_{K^p}}\rightarrow h_{1,*}h_1^{-1}\ker d_{\calS_{K^p}}=h_{1,*}\ker d_{\calS_{ K^p\cap gK^pg^{-1}}}$ and $h_{2,!}\ker d_{\calS_{g^{-1}K^pg\cap K^p}}=h_{2,!}h_1^{-1}\ker d_{\calS_{K^p}}\rightarrow \ker d_{\calS_{K^p}}$.
\end{remark}
The following proposition is our product formula. 
\begin{proposition}\label{prop:productformula}
    There exist $\bbT^S$-equivariant isomorphisms of $D_p^{\times}$-representations
    \begin{align}\label{eq:RGammakerdSk}
        R\Gamma(\calS_{K^p},\ker d_{\calS_{K^p}})&\isom R\Gamma_c(\calM_{\LT,\infty},\omega_{\calM_{\LT,\infty}}^{(0,-k),\sm})\ox_{\calH(G)}^L\calA^{K^p}_{\bar G,k}\\
        R\Gamma(\calS_{K^p},\ker d_{\calS_{K^p}}')&\isom R\Gamma_c(\calM_{\LT,\infty},\omega_{\calM_{\LT,\infty}}^{(-k-1,1),\sm})\ox_{\calH(G)}^L\calA^{K^p}_{\bar G,k}.
    \end{align}
    Here the action of $G=\GL_2(\bbQ_p)$ on $R\Gamma_c(\calM_{\LT,\infty},\omega_{\calM_{\LT,\infty}}^{(0,-k),\sm})$ or $R\Gamma_c(\calM_{\LT,\infty},\omega_{\calM_{\LT,\infty}}^{(-k-1,1),\sm})$ is induced by the left $G$-action on $\calM_{\LT,\infty}$ in \S\ref{subsection:twotowers} twisted by the inverse transpose of $G$ as in Remark \ref{remarkuniformizationtwist}.
\end{proposition}
\begin{proof}
    We will only prove the first isomorphism, and the second one follows similarly. 

    The Hecke algebra $\calH(G)\simeq \calC_c^{\rm sm}(G,C)$ is itself a smooth $G\times G^{\op}$-modules, where the action of $(g_1,g_2)\in G\times G^{\op}$ on $f\in \calC^\sm_c(G,C)$ is given by $(g_1,g_2).f(-)=f(g_2-g_1)$. By \cite[Lemma A.6(a)]{Casselmansmoothcohomology}, for a discrete subgroup $\Gamma\subset G$, $\calC^{\sm}_c(G,C)\isom \bbZ[\Gamma]\otimes_{\bbZ}\calC^{\sm}_c(\Gamma\backslash G,C)$ as $\Gamma$-modules where $\Gamma$ acts on the $\bbZ[\Gamma]$-factor of the tensor product. Hence $\calC^{\sm}_c(G,C)\ox^L_{\bbZ[\Gamma]}V=\calC^{\sm}_c(G,C)\ox_{\bbZ[\Gamma]}V$ for any $\bbZ[\Gamma]$-module, i.e. higher Tor groups vanish. 

    Let $V=V^{(0,-k)}$ and let $C^{\rm sm}(\Gamma\backslash G,V)$ be the space of functions $f:\Gamma\backslash G\to V$ such that $f(gk)=k^{-1}f(g)$ for any $k$ in some open compact subgroup of $G$, with a $G$-action given $(g.f)(z)=gf(zg)$ for any $z,g\in G$. There exists a natural map $\calC_c^{\rm sm}(G,C)\otimes_C V\rightarrow \calC^{\rm sm}(\Gamma\backslash G,V):f\otimes v\mapsto (g\in G\mapsto \sum_{\gamma\in\Gamma}f(\gamma g)g^{-1}\gamma^{-1}.v)$. This map is equivariant for the left $G\times \Gamma$-actions (where $G$ acts on $\calC_c^{\rm sm}(G,C)\otimes_C V$ only on the first factor via right translations and $\Gamma$ acts on it diagonally via the left translation on $\calC_c^{\rm sm}(G,C)$) and induces a natural $G$-isomorphism 
    \[ \calC^{\sm}_c(G,C)\ox^L_{C[\Gamma]}V=\calC^{\sm}_c(G,C)\ox_{C[\Gamma]}V\stackrel{\sim}{\rightarrow}C^{\rm sm}(\Gamma\backslash G,V).\]
    We write $Z_{K^p}=\sqcup_{x\in Z_{K^p}/\GL_2(\bbQ_p)} \Gamma_x\backslash \GL_2(\bbQ_p)$ so that we have $\calS_{K^p}=\sqcup_{\Gamma_x}\calM_{\Dr,\infty}/\Gamma_x$. By the above discussions,
    \[\calA^{K^p}_{\bar G,k}=\oplus_xC^{\rm sm}(\Gamma_x\backslash G,V)=\oplus_x \calC_c^{\rm sm}(G,C)\otimes_{C[\Gamma_x]}^LV=\oplus_x\calH(G)\otimes_{C[\Gamma_x]}^LV\]
    as smooth left $\calH(G)$-modules. Hence by Lemma \ref{Lem:RGammacLTsmGamma} and (\ref{equationRGammackerdDrlocallyalgebraic}) we have 
    \begin{align*}
        R\Gamma(\calS_{K^p},\ker d_{\calS _{K^p}})&\isom \oplus_x R\Gamma_c(\calM_{\Dr,\infty}/\Gamma_x,\ker d_{\calS_{\Gamma_x}})\\
        &\isom \oplus_x  R\Gamma_c(\calM_{\Dr,\infty},\omega_{\calM_{\LT,\infty}}^{(0,-k),\sm})\otimes_{C[\Gamma_x]}^LV\\
        &\isom  R\Gamma_c(\calM_{\Dr,\infty},\omega_{\calM_{\LT,\infty}}^{(0,-k),\sm})\otimes_{\calH(G)}^L(\oplus_x\calH(G)\otimes_{C[\Gamma_x]}^LV)\\
        &\isom R\Gamma_c(\calM_{\LT,\infty},\omega_{\calM_{\LT,\infty}}^{(0,-k),\sm})\ox_{\calH(G)}^L\calA^{K^p}_{\bar G,k}.
    \end{align*}
   where for the third isomorphism, we used the $\calH(G)$-isomorphism (\cite[(I.2.3.3)]{renard2010representations}) 
    \begin{align*}
        R\Gamma_c(\calM_{\LT,\infty},\omega_{\calM_{\LT,\infty}}^{(0,-k),\sm})\isom R\Gamma_c(\calM_{\LT,\infty},\omega_{\calM_{\LT,\infty}}^{(0,-k),\sm})\ox_{\calH(G)}^L\calH(G).
    \end{align*}

    We remain to show that the isomorphism (\ref{eq:RGammakerdSk}) we just built is $\bbT^S$-equivariant. For this, we show that the composition
    \begin{align}\label{eq:compositionRGammac}
         R\Gamma_c(\calM_{\LT,\infty},\omega_{\calM_{\LT,\infty}}^{(0,-k),\sm})\ox_{C}^L\calA^{K^p}_{\bar G,k}\to R\Gamma_c(\calM_{\LT,\infty},\omega_{\calM_{\LT,\infty}}^{(0,-k),\sm})\ox_{\calH(G)}^L\calA^{K^p}_{\bar G,k}\rightarrow R\Gamma_c(\calS_{K^p},\ker d_{\calS _{K^p}})
    \end{align} 
    is $\bbT^S$-equivariant, where we identify $R\Gamma_c(\calM_{\LT,\infty},\omega_{\calM_{\LT,\infty}}^{(0,-k),\sm})\ox_{\calH(G)}^L\calA^{K^p}_{\bar G,k}$ as the $G$-coinvariant (in the derived category of smooth $\calH(G)$-modules) of $R\Gamma_c(\calM_{\LT,\infty},\omega_{\calM_{\LT,\infty}}^{(0,-k),\sm})\ox_{C}^L\calA^{K^p}_{\bar G,k}$ (cf. \cite[(III.1.15.1)]{renard2010representations}). For this purpose we will write (\ref{eq:compositionRGammac}) in a more geometric way as the descent of certain cohomology along a $G$-torsor.

    First, the $p$-adic uniformization (\ref{eq:realpadicuniformization}) is equivariant for the $\bbT^S$-action. This gives a $\bbT^S$-equivariant maps 
    \begin{align*}
        \pi_{G}:\calM_{\LT,\infty}\times Z_{K^p}=\sqcup_{x} \calM_{\LT,\infty}\times(\Gamma_x\backslash \GL_2(\bbQ_p))\rightarrow \calS_{K^p}=\sqcup_{x} \calM_{\LT,\infty}/\Gamma_x.
    \end{align*}
    Here we identify $Z_{K^p}$ with the perfectoid space over $C$ associated to the profinite set $Z_{K^p}$. Consider the action map $\mathrm{act}:\calM_{\Dr,\infty}\times\GL_2(\bbQ_p)\rightarrow \calM_{\Dr,\infty}:(x,g)\mapsto g.x$ which is $\Gamma_x$-equivariant where $\Gamma_x$ acts on $\calM_{\Dr,\infty}\times\GL_2(\bbQ_p)$ by the left multiplications on the $\GL_2(\bbQ_p)$-factor. Taking the quotient by $\Gamma_x$ we get a Cartesian diagram
    \[\begin{tikzcd}
        \calM_{\Dr,\infty}\times \GL_2(\bbQ_p)\arrow[r]\arrow[d,"\mathrm{act}"]&\calM_{\Dr,\infty}\times \Gamma_x\backslash\GL_2(\bbQ_p)\arrow[d,"\pi_G"]  \\
        \calM_{\Dr,\infty}\arrow[r]&\calM_{\Dr,\infty}/\Gamma_x 
    \end{tikzcd}\] 
    We see that the map $\pi_G$ is a $\GL_2(\bbQ_p)$-torsor locally trivial for the analytic topology.
    
    Let $V^*$ be the $C$-dual of $V$ and let $\calO_{\calM_{\LT,\infty}\times Z_{K^p}}^{(V,V^*)\-\lalg}$ be the subsheaf of $\calO_{\calM_{\LT,\infty}\times Z_{K^p}}$ consisting of sections which are locally $(V,V^*)$-algebraic under the left $G\times G$-action. Note that $\calC^{\cont}(Z_{K^p},C)^{V^*\-\lalg}=V^*\otimes_C\calA^{K^p}_{\bar G,k}$. Let $\calO_{\calM_{\LT,\infty}\times Z_{K^p}}^{(V,V^*)\-\lalg,\sm}\subset \calO_{\calM_{\LT,\infty}\times Z_{K^p}}^{(V,V^*)\-\lalg}$ be the subsheaf consisting of sections which are smooth under the diagonal action of $\Delta G\subset G\times G$, corresponding to the one-dimensional summand $(V\otimes_C V^*)^G\subset V\otimes_C V^*$. 

    By the same proof for Lemma \ref{lem:RpOLTsm}, we have a natural $\bbT^S$-isomorphism 
    \begin{align*}
        R\Gamma_c(\calM_{\LT,\infty}\times Z_{K^p},\calO_{\calM_{\LT,\infty}\times Z_{K^p}}^{(V,V^*)-\lalg})&\isom\colim_{n,K_p} R\Gamma_c(\calM_{\LT,n},\omega_{\calM_{\LT,n}}^{(0,-k)})\ox_C^LV\otimes_C^L V^* \ox_C^L(\calA^{K^p}_{\bar G,k})^{K_p}
        \\
        &=R\Gamma_c(\calM_{\LT,\infty},\omega_{\calM_{\LT,\infty}}^{(0,-k),\sm})\ox_{C}^L\calA^{K^p}_{\bar G,k}\ox_C^L\End_C(V)
    \end{align*}
    where $K_p$ runs through open compact subgroups of $G$. Taking the smooth part for the diagonal $G$-action, we get
    \[ R\Gamma_c(\calM_{\LT,\infty}\times Z_{K^p},\calO_{\calM_{\LT,\infty}\times Z_{K^p}}^{(V,V^*)-\lalg,\sm})=R\Gamma_c(\calM_{\LT,\infty},\omega_{\calM_{\LT,\infty}}^{(0,-k),\sm})\ox_{C}^L\calA^{K^p}_{\bar G,k}\]
    Moreover, we have 
    \[R\pi_{G,!}\calO_{\calM_{\LT,\infty}\times Z_{K^p}}^{(V,V^*)\-\lalg}=\pi_{G,!}\calO_{\calM_{\LT,\infty}\times Z_{K^p}}^{(V,V^*)\-\lalg}\] 
    where $\pi_{G,!}\calO_{\calM_{\LT,\infty}\times Z_{K^p}}^{(V,V^*)\-\lalg}(U)=(\ker d_{\Dr}(U))\otimes_C\calC^{\cont}_c(G,C)^{V^*\-\lalg}$ for $U\subset \calM_{\Dr,\infty}/\Gamma$ such that $\pi_G^{-1}(U)\isom U\times \GL_2(\bbQ_p)$. (Here we note that the constant sheaf (also the structure sheaf) on the locally profinite group $G$ is flasque, as given a section $s$ on a compact open, one can also extend $s$ to a larger compact open via extension by zero. Hence by Lemma \ref{flasqueacyclic} there are no higher pushforwards on the $G$ component, see also \cite[Lemma 7.2]{scholze2022etale}). And there is a natural map $\pi_{G,!}\calO_{\calM_{\LT,\infty}\times Z_{K^p}}^{(V,V^*)\-\lalg}\rightarrow \ker d_{\calS_{K^p}}$, which can be obtained from a similar map $\mathrm{act}_{!}\calO_{\calM_{\LT,\infty}\times \GL_2(\bbQ_p)}^{(V,V^*)\-\lalg}\rightarrow \ker d_{\Dr}$, identiying $\ker d_{\calS_{K^p}}$ as the $G$-coinvariant of the locally algebraic $G$-module $\pi_{G,!}\calO_{\calM_{\LT,\infty}\times Z_{K^p}}^{(V,V^*)\-\lalg}$, or the smooth $G$-module $\pi_{G,!}\calO_{\calM_{\LT,\infty}\times Z_{K^p}}^{(V,V^*)\-\lalg,\sm}$. 
    
    Restricting to a component $\calS_{\Gamma}=\calM_{\Dr,\infty}/\Gamma$, the $\GL_2(\bbQ_p)$-smooth part of the map 
    \[ R\Gamma_c(\calM_{\LT,\infty}\times \Gamma\backslash \GL_2(\bbQ_p),\calO_{\calM_{\LT,\infty}\times Z_{K^p}}^{(V,V^*)\-\lalg})=R\Gamma_c(\calM_{\Dr,\infty}/\Gamma,R\pi_{G,!}\calO_{\calM_{\LT,\infty}\times Z_{K^p}}^{(V,V^*)\-\lalg})\rightarrow R\Gamma_c(\calM_{\Dr,\infty}/\Gamma,\ker d_{\calS_\Gamma}),\]
    contributes to (\ref{eq:compositionRGammac}). The above map is the (derived) $\Gamma$-coinvariant of the map
    \[R\Gamma_c(\calM_{\Dr,\infty}\times\GL_2(\bbQ_p),\calO_{\calM_{\LT,\infty}\times Z_{K^p}}^{(V,V^*)\-\lalg})\rightarrow R\Gamma_c(\calM_{\Dr,\infty},\calO^{\lalg,(0,-k)}_{\calM_{\LT,\infty}})\] by the isomorphisms in Lemma \ref{Lem:RGammacLTsmGamma}.
    Hence the map
    \[ R\Gamma_c(\calM_{\LT,\infty}\times Z_{K^p},\calO_{\calM_{\LT,\infty}\times Z_{K^p}}^{(V,V^*)\-\lalg,\sm})\rightarrow R\Gamma_c(\calM_{\LT,\infty}\times Z_{K^p},\calO_{\calM_{\LT,\infty}\times Z_{K^p}}^{(V,V^*)\-\lalg})\rightarrow R\Gamma_c(\calS_{K^p},\ker d_{\calS _{K^p}})\]
    induced by $R\pi_{G,!}\calO_{\calM_{\LT,\infty}\times Z_{K^p}}^{(V,V^*)\-\lalg,\sm}\subset \pi_{G,!}\calO_{\calM_{\LT,\infty}\times Z_{K^p}}^{(V,V^*)\-\lalg}\rightarrow \ker d_{\calS_{K^p}}$ coincides with (\ref{eq:compositionRGammac}). By the definition of the Hecke actions on the (compactly supported) cohomologies as in (\ref{eq:HeckeonRGammakerdS}), we can check that (\ref{eq:compositionRGammac}) is $\bbT^{S}$-equivariant. 
\end{proof}

\begin{remark}
In the $\ell$-adic cohomology setting, the Mantovan product formula, e.g. \cite{mantovan2004certain, mantovan2005cohomology, caraiani2017generic, zhang2023peltypeigusastackpadic, daniels2024igusastackscohomologyshimura} expresses the compactly supported $\ell$-adic cohomology of a fixed Newton stratum of the tower of certain Shimura varieties in terms of the cohomology of some local Shimura varieties and the cohomology of some Igusa varieties. The results in Proposition \ref{prop:productformula} extend the known product formula for de Rham cohomology of the Shimura curve (cf. \cite[\S 4.2]{DLB17})
\begin{align*}\label{remarkuniformizationdeRhamcohomology}
    R\Gamma_{\dR,c}(\calM_{\LT,\infty})\ox_{\calH(G)}^L\calA_{\bar G,0}^{K^p}\isom R\Gamma_{\dR,c}(\calS_{K^p}).
\end{align*}
that will be discussed in the next subsection. 
\end{remark}

\begin{remark}\label{remarkproductformula}
    It is natural to expect a similar product formula for $R\Gamma(\calS_{K^p},\calO_{\calS_{K^p}})$ or $R\Gamma(\calS_{K^p},\calO_{\calS_{K^p}}^{\lan})$. Let $C[[G]]=C[G]\otimes_{\calO_C[\GL_2(\bbZ_p)]}\calO_C[[\GL_2(\bbZ_p)]]$ be the Iwasawa algebra for $G$ and let $\calC^{\cont}(X,C)$ denote the space continuous functions on a $p$-adic manifold $X$. We have
    \[R\Gamma_c(\calM_{\Dr,\infty}\times Z_{K^p},\calO_{\calM_{\Dr,\infty}\times Z_{K^p}})=R\Gamma_c(\calM_{\Dr,\infty},\calO_{\calM_{\Dr,\infty}})\widehat{\otimes}_C^L\calC^{\cont}(Z_{K^p},C).\]
    However, its (derived continuous) $G$-coinvariant does not give $R\Gamma(\calS_{K^p},\calO_{\calS_{K^p}})$. This can be seen already on the sheaf level: $\pi_{G,*}\calO_{\calM_{\Dr,\infty}\times Z_{K^p}}\widehat{\otimes}^L_{C[[G]]}C\neq \calO_{\calS_{K^p}}$ since $\calC_c^{\cont}(G,C)\widehat{\otimes}^L_{C[[G]]}C=C[\dim G]$ using \cite[Proposition 3.2.11]{jacinto2023solid}. But the same arguments as in the proof of Proposition \ref{prop:productformula} for $\Gamma$-descent shows that
    \begin{align}\label{equationproductformulafinitelevel}
        R\Gamma(\calS_{K^p},\calO_{\calS_{K^p}})&\isom R\Gamma_c(\calM_{\LT,\infty},\calO_{\calM_{\LT,\infty}})\widehat{\otimes}^L_{C[[G]]}\calD^{\cont}(Z_{K^p},C)\\
        &\isom R\Gamma_c(\calM_{\LT,\infty},\calO_{\calM_{\LT,\infty}})\widehat{\otimes}^L_{C[[G]]}\calC^{\cont}(Z_{K^p},C)[-\dim G]\nonumber
    \end{align}
    where $\calD^{\cont}(Z_{K^p},C)=\Hom^{\cont}_{C}(\calC^{\cont}(Z_{K^p},C),C)$ is the space of continuous measures on $Z_{K^p}$. This matches with the heuristic that $\pi_G^!\calO_{\calS_{K^p}}=\calO_{\calM_{\Dr,\infty}\times Z_{K^p}}[\dim G]$.
\end{remark}

\subsection{Cohomology of Intertwining operator I}\label{sec:cohomologyI}
Let $k\geq 0$ be an integer. Following \cite[\S 5]{PanII}, we focus on the spectral decomposition of $\ker I^1$ with respect to the Hecke action, where
\begin{align*}
    I^1:H^1(\check \fl,\calO_{K^p}^{\lan,(0,-k)})\stackrel{H^1(I)}{\rightarrow} H^1(\check\fl,\calO_{K^p}^{\lan,(-k-1,1)}(k+1)).
\end{align*}
and the intertwining operator $I$ is (Definition \ref{def:intertwiningoperator}):
\begin{align*}
    I=\bar d'_{K^p}\circ d_{K^p}: \calO_{K^p}^{\lan,(0,-k)}\stackrel{d}{\rightarrow}\calO_{K^p}^{\lan,(0,-k)}\ox_{\calO_{K^p}^{\sm}}(\Omega^{1,\sm}_{K^p})^{\ox k+1} \stackrel{\bar d'}{\rightarrow }\calO_{K^p}^{\lan,(-k-1,1)}(k+1).
\end{align*}

As in \cite[\S 5]{PanII}, as well as \cite[Theorem 4.7.2]{QS24}, we will express $\ker I^1$ in terms of cohomology groups of some (de Rham) complexes associated to the differentials operators $d_{K^p},\bar d_{K^p}$. Put 
\begin{align*}
    &\DR^{\lalg}:=[\calO^{\lalg,(0,-k)}_{K^p}\stackrel{d_{K^p}}{\rightarrow} \calO^{\lalg,(0,-k)}_{K^p}\ox_{\calO_{K^p}^{\sm}}(\Omega_{K^p}^{1,\sm})^{\ox k+1}]\\
    &\DR:=[\calO_{K^p}^{\lan,(0,-k)}\stackrel{d_{K^p}}{\to} \calO_{K^p}^{\lan,(0,-k)}\ox_{\calO_{K^p}^{\sm}}(\Omega_{K^p}^{1,\sm})^{\ox k+1}]\simeq \ker d_{K^p}\\
    &\DR':=[\calO_{K^p}^{\lan,(0,-k)}\ox_{\calO_{\check\fl}}(\Omega^1_{\check\fl})^{\ox k+1}  \stackrel{d'_{K^p}}{\rightarrow }\calO_{K^p}^{\lan,(-k-1,1)}(k+1)]\simeq \ker d_{K^p}'
\end{align*} 
which are complexes of abelian sheaves on $\check\fl$ concentrating in cohomological degrees $[0,1]$. The quasi-isomorphisms $\DR\simeq \ker d_{K^p}, \DR'\simeq \ker d_{K^p}'$ follow from the surjectivity of $d_{K^p}$ and $d_{K_p}'$ by Proposition \ref{prop:dKvsurj}.

Recall that we write $\calA^{K^p}_{\bar G,k}$ (Definition \ref{def:algebraicautomorphicform}) for the space of algebraic automorphic forms of $p$-adic algebraic weight $V^{(0,-k)}$ for $k\geq 0$ on the Shimura set $Z_{K^p}$ (defined in \S\ref{sec:uniformization}). The sheaves on $\check\fl$ we considered above and their cohomology groups carry natural $\bbT^S$-actions, and morphisms between them are Hecke equivariant. It will follow from the results below that the Hecke eigensystems appearing in those cohomology groups are subsets of the eigensystems in the classical automorphic forms.
\begin{definition}\label{def:setHeckeeigenvalues}
    Let $\sigma_{k}^{K^p}$ be the set of systems of Hecke eigenvalues (i.e. characters $\bbT^S\rightarrow C$) that appear in $\calA^{K^p}_{\bar G,k}$ (i.e. $\calA_{\bar G,k}^{K^p}[\lambda]\neq 0$). Let $\sigma_{k,c}^{K^p}$ be the subset of $\sigma_k^{K^p}$ consisting of Hecke eigensystems such that the corresponding automorphic representations of $\bar G(\bbA)$ are cuspidal, i.e., don't factor through the reduced norm of $\bar G$.
\end{definition}
If $\lambda\in \sigma_{k,c}^{K^p}$, let $\pi_\lambda$ be the irreducible automorphic representation of $\bar G$ corresponding to $\lambda$. Write $\pi_\lambda=\pi_{\lambda,\infty}\ox\pi_{\lambda}^\infty$ where $\pi_{\lambda,\infty}$ is the archimedean component and $\pi^\infty_{\lambda}$ is the non-archimedean component. Decompose further $\pi^\infty=\pi^{\infty,p}_{\lambda}\ox_C\pi_{\lambda,p}$, where $\pi_{\lambda,p}$ is a smooth irreducible representation of $\bar G_p=\GL_2(\bbQ_p)$. 
\begin{theorem}\label{thm:decompositionautomorphicforms}
    The action of $\bbT^S$ on $\calA^{K^p}_{\bar G,k}$ is semisimple and there is a decomposition
    \begin{align}\label{AKvdecomp}
        \calA_{\bar G,k}^{K^p}=\bigoplus_{\lambda\in \sigma_{k}^{K^p}}\calA_{\bar G,k}^{K^p}[\lambda].
    \end{align}
    Moreover, $\calA_{\bar G,k}^{K^p}[\lambda]=(\pi^{\infty,p}_{\lambda})^{K^p}\ox_C\pi_{\lambda,p}$ and $\pi_{\lambda,p}$ is infinite-dimensional for $\lambda\in \sigma_{k}^{K^p}$. 
\end{theorem}
\begin{proof}
    The representation $V=V^{(0,-k)}=V_E\otimes_EC$ of $\bar G\otimes_{\bbQ}C$, as well as the space $\calA^{K^p}_{\bar G,k}=\calA^{K^p}_{\bar G,k,E}\otimes_EC$ can be defined over a finite extension $E$ of $\bbQ_p$. If we base change to $\bbC$ along an embedding $E\hookrightarrow\overline{E}\simeq \bbC$, then there is a $\bbT^S$-equivariant decomposition (cf. \cite[Proposition 3.8.1]{loeffler2011overconvergent}) 
    \[\calA^{K^p}_{\bar G,k,E}\otimes_E\bbC=\bigoplus_{\pi}m(\pi)(\pi^p)^{K^p}\otimes \pi_{p}\otimes (\pi_{\infty}\otimes (V_E\otimes_{E}\bbC))^{\bar G(\bbR)}\]
    where $\pi=\pi^p\otimes \pi_p\otimes \pi_{\infty}$ runs through all automorphic representations of $\bar G$ such that $\pi_{\infty}$ is the linear dual of $V_E\otimes_{E}\bbC$ and $m(\pi)$ is the multiplicity of $\pi$. If $\dim_{\bbC} \pi>1$, then $m(\pi)=1$ by the classical Jacquet--Langlands correspondence and the multiplicity one theorem for $\GL_2$ (cf. \cite[Proposition 11.1.1]{langlands1970automorphic}), and in this case the corresponding system of Hecke eigenvalues $\lambda:\bbT^S\rightarrow\bbC$ can be defined over a finite extension of $E$. Moreover, $\pi_{\lambda,p}$ is infinite-dimensional being generic, see \cite[\S 11]{langlands1970automorphic}.
\end{proof}

\begin{lemma}\label{lem:cohoofomegasm}
Let $(a,b)\in\bbZ^2$ be integers. Then
\begin{align}
    H^i(\check{\fl},\omega_{K^p}^{(a,b),\sm})\isom \dlim_{K_p\subset D_p^\times}H^i(\calS_{K^pK_p},\omega_{\calS_{K^pK_p}}^{(a,b)}).
\end{align}
\begin{proof}
The proof is similar to \cite[Lemma 5.3.5]{Pan22}.
\end{proof}
\end{lemma}
\begin{proposition}\label{prop:dRcohOKvalg}
    The following statements hold.
    \begin{enumerate}[(1)]
        \item There are natural isomorphisms 
        \[\begin{array}{c}
        \bbH^0(\check\fl,\DR^{\lalg})\stackrel{\sim}{\rightarrow} H^0(\check\fl,\calO_{K^p}^{\lalg,(0,-k)}),\\ 
        H^1(\check\fl,\calO_{K^p}^{\lalg,(0,-k)}\ox_{\calO_{K^p}^{\sm}}(\Omega_{K^p}^{1,\sm})^{\ox k+1})\stackrel{\sim}{\to} \bbH^2(\check\fl,\DR^{\lalg})        
        \end{array}\]
        and an exact sequence 
        \begin{align}\label{eq:Hodgefiltration}
            0\to H^0(\check\fl,\calO_{K^p}^{\lalg,(0,-k)}\ox_{\calO_{K^p}^{\sm}}(\Omega_{K^p}^{1,\sm})^{\ox k+1})\to \bbH^1(\check\fl,\DR^{\lalg})\to H^1(\check\fl,\calO_{K^p}^{\lalg,(0,-k)})\to 0.
        \end{align}
        Moreover, $\bbH^0(\check\fl,\DR^{\lalg})=\bbH^2(\check\fl,\DR^{\lalg})=0$ if $k>0$.
        \item The actions of $\bbT^{S}$ on $\bbH^i(\check\fl,\DR^{\lalg})$ are semisimple. And the actions of $D_p^{\times}$ on $\bbH^i(\check\fl,\DR^{\lalg})$ factors through the reduced norm for $i=0,2$.
    \end{enumerate}
\end{proposition}
\begin{proof}
    (1) For $k=0$, the hypercohomology of $\DR^{\lalg}$ is essentially the algebraic de Rham cohomology of the (finite level) Shimura curves by Lemma \ref{lem:cohoofomegasm} and Proposition \ref{prop:Wlalg}: 
    \[\bbH^i(\check\fl,\DR^{\lalg})=\colim_{K_p\subset D_p^{\times}}H^i_{\dR}(\calS_{K_pK^p})\]
    The proposition follows from the degeneration of the Hodge-de Rham spectral sequence. For $k>0$, the statements follow from (2) of Lemma \ref{lem:H0vanish} below.

    (2) Since $\DR^{\lalg}=W^{(0,-k)}\otimes_C[\omega^{(0,-k),\sm}_{K^p}\stackrel{}{\rightarrow} \omega^{(-k-1,1),\sm}_{K^p}]$, it shall be enough to consider the hypercohomology of the complex $[\omega^{(0,-k),\sm}_{K^p}\stackrel{}{\rightarrow} \omega^{(-k-1,1),\sm}_{K^p}]$. Fix a level $K=K^pK_p$. Let $S_{K,\overline{\bbQ_p}}$ denote the base change to $\overline{\bbQ_p}$ of the $\bbQ$-scheme in \S\ref{subsec:quaternionicshimuracurve}. Let $(V_{\rm dR},\nabla,\Fil^{\bullet}V_{\rm dR})$ be the filtered connexion corresponding to the \'etale local system  $\underline{V}^{(k,0)}_{\overline{\bbQ_p}}$ on $S_{K,\overline{\bbQ_p}}$ attached to the algebraic representation of $D_p^{\times}$ of the highest weight $(k,0)$ over $\overline{\bbQ_p}$ by the $p$-adic Riemann-Hilbert correspondence in \cite[Proposition 5.2.17]{DLLZ}. The $p$-adic correspondence in \textit{loc. cit.} is compatible with the complex correspondence \cite[Theorem 5.3.1]{DLLZ} and the pullback to $\calM_{\Dr,K_p}$ \cite[Theorem 3.9]{liu2017rigidity}. Hence, the pullback and restriction of $V_{\rm dR}$ to $\calM_{\Dr,\infty}$ coincides with the $\GL_2(\bbQ_p)$-equivariant bundle $V^{(0,-k)}\otimes_{C}\calO_{\calM_{\Dr,\infty}}$ (there is a dual due to our convention of the $D_p^{\times}$-action on $\calM_{\Dr,\infty}$, see Remark \ref{remarknormalizationLT}). Hence, by the usual dual BGG construction (cf. \cite{faltings2006cohomology}, \cite[Remark 6.1.9]{lan2023rham} and Remark \ref{remarktranslation}), we have
    \[R\Gamma(S_{K,\overline{\bbQ_p}},[\omega_{S_{K,\overline{\bbQ_p}}}^{(0,-k)}\rightarrow \omega_{S_{K,\overline{\bbQ_p}}}^{(-k-1,1)}])\otimes_{\overline{\bbQ_p}}\bbC\simeq R\Gamma(S_{K}(\bbC),[\omega_{S_{K}(\bbC)}^{(0,-k)}\rightarrow \omega_{S_{K}(\bbC)}^{(-k-1,1)}])\simeq R\Gamma(S_{K}(\bbC), \underline{V}_{\bbC}^{(k,0)}).\]   
    By Matsushima's formula (see \cite[\S VII Theorem 5.2]{borel2000continuous}, \cite[\S 15.5]{getz2024introduction} as well as \cite[\S 2.2.4]{carayol1994formes})
    \[H^i(S_{K}(\bbC), \underline{V}_{\bbC}^{(k,0)})=\oplus_{\pi}(\pi^{\infty})^K\otimes H^i(\mathfrak{sl}_{2,\bbR},\mathrm{SO}_2(\bbR),(\pi_{\infty}\otimes V_{\bbC}^{(k,0)})^{\bbR_{>0}})\]
    where $\pi$ runs through irreducible automorphic representations of $G(\bbQ)$ such that the central character of $\pi_{\infty}$ has the same restriction on $\mathbb{R}_{>0}$ with that of $V^{(0,-k)}$ and has the same (regular) infinitesimal character as $V^{(0,-k)}$. It's then clear that the action of $\bbT^S$ on $\bbH^i(\check\fl,\DR^{\lalg})$ is semi-simple.

    By the proof of Lemma \ref{lem:H0vanish}, the action of $D_p^{\times}$ on $\bbH^0(\check\fl,\DR^{\lalg})$ factors through the reduced norm. Then the same statement for $\bbH^2(\check\fl,\DR^{\lalg})$ follows from the Poincaré duality.
\end{proof}

We will also need the algebraic de Rham cohomology of the Lubin--Tate space. Set 
\begin{align*}
    R\Gamma_{\dR,c}(\calM_{\LT,\infty},\omega_{\calM_{\LT,\infty}}^{(0,-k),\sm}):=\mathrm{fib}(R\Gamma_{c}(\calM_{\LT,\infty},\omega_{\calM_{\LT,\infty}}^{(0,-k),\sm})\to R\Gamma_{c}(\calM_{\LT,\infty},\omega_{\calM_{\LT,\infty}}^{(-k-1,1),\sm})).
\end{align*}
\begin{lemma}\label{lemmaderhamlubintate}
    We have isomorphisms of locally algebraic $D_p^{\times}$-representations
    \begin{align*}
    H^i_{\dR,c}(\calM_{\LT,\infty},\omega_{\calM_{\LT,\infty}}^{(0,-k),\sm})= H^i(R\Gamma_{\dR,c}(\calM_{\LT,\infty},\calO_{\calM_{\LT,\infty}}^{\sm}))\ox W^{(0,-k)}\isom  H^i_{\dR,c}(\calM_{\LT,\infty})\ox W^{(0,-k)}
    \end{align*}
    for $i=1,2$. 
\end{lemma}
\begin{proof}
    The cohomology groups of $R\Gamma_{\dR,c}(\calM_{\LT,\infty},\omega_{\calM_{\LT,\infty}}^{(0,-k),\sm})$ vanish on all degrees other than $i=1,2$ by acyclicity results for coherent sheaves on quasi-Stein spaces and Serre duality. The complex $\omega_{\calM_{\LT,\infty}}^{(0,-k),\sm}\to \omega_{\calM_{\LT,\infty}}^{(-k-1,1),\sm}$ can be seen as the translation of the standard complex $\calO_{\calM_{\LT,\infty}}^{\sm}\to \Omega^{1,\sm}_{\calM_{\LT,\infty}}$ from the trivial weight $(0,0)$ to the weight $(0,-k)$ (see \cite[Lemma 5.6.6]{PanII} for the Drinfeld case and Remark \ref{remarktranslation}). Note that the compactly supported cohomology groups of $[\calO_{\calM_{\LT,\infty}}^{\sm}\to \Omega^{1,\sm}_{\calM_{\LT,\infty}}]$ are smooth for the $D_p^\times$-action \cite[Th\'eor\`eme 4.1]{CDN20}, so that the translations of these cohomology groups equal to the tensor products with the algebraic representation $W^{(0,-k)}$.   
\end{proof}
\subsection{Cohomology of Intertwining operator II}\label{sec:cohomologyII}
The goal of this subsection is to prove  our main results: Theorem \ref{thm:H1DR}, Theorem \ref{thm:kerI1classical} and Corollary \ref{cor:kerI1DRDRprime}. We have defined complexes $\DR^{\lalg},\DR,\DR'$ in \S\ref{sec:cohomologyI}. By Proposition \ref{prop:dKvsurj}, we have the following commutative diagram with exact rows:
\begin{equation}\label{eq:DRdiagram}
     \begin{tikzcd}
        \DR^{\lalg} \arrow[r,rightarrow]&  \DR\arrow[r,rightarrow]& \DR' \\    
        \calO_{K^p}^{\lalg,(0,-k)} \arrow[r,hookrightarrow] \arrow[d, "d_{K^p}"] & \calO_{K^p}^{\lan,(0,-k)} \arrow[r, "\bar d_{K^p}",twoheadrightarrow] \arrow[d, "d_{K^p}",twoheadrightarrow] & \arrow[d, "d_{K^p}'",twoheadrightarrow] \calO_{K^p}^{\lan,(0,-k)}\ox_{\calO_{\check\fl}}(\Omega^1_{\check\fl})^{\ox k+1}   \\
        \calO_{K^p}^{\lalg,(0,-k)}\ox_{\calO_{K^p}^{\sm}}(\Omega^{1,\sm}_{K^p})^{\ox k+1}\arrow[r,hookrightarrow]& \calO_{K^p}^{\lan,(0,-k)}\ox_{\calO_{K^p}^{\sm}}(\Omega^{1,\sm}_{K^p})^{\ox k+1}  \arrow[r, "\bar d'_{K^p}",twoheadrightarrow]   &  \calO_{K^p}^{\lan,(-k-1,1)}(k+1)       
    \end{tikzcd}
\end{equation}
We write for short $\bbH^i(\calF)$ for the hypercohomology $\bbH^i(\check\fl,\calF)$ for a complex of sheaves on $\check{\fl}$ and $H^i(\calF)$ for $H^i(\check\fl,\calF)$ if $\calF$ is an abelian sheaf. Recall that the analytic space $\check\fl$ has cohomological (Krull) dimension $1$ (see \cite[Proposition 2.5.8]{de1996etale} and also the proof of \cite[Corollary IV.2.2]{Sch15}).

We collect some lemmas for later uses.
\begin{lemma}\label{lem:H0vanish}
    Let $k\geq 0$.
    \begin{enumerate}[(1)]
        \item For $(a,b)\in\bbZ^2$, we have $H^0(\calO^{\lan,(a,b)}_{K^p})=0$ if $a\neq b$.
        \item If $k>0$, then $H^0(\calO_{K^p}^{\lalg,(0,-k)})=0$ and $H^1(\calO_{K^p}^{\lalg,(0,-k)}\ox_{\calO_{K^p}^{\sm}}(\Omega_{K^p}^{1,\sm})^{\ox k+1})=0$.
        \item The natural map $H^0(\calO_{K^p}^{\lalg,(0,-k)})\to H^0(\calO_{K^p}^{\lan,(0,-k)})$ is an isomorphism. 
        \item The natural map $H^0(\calO_{K^p}^{\lalg,(0,-k)}\ox_{\calO_{K^p}^{\sm}}(\Omega_{K^p}^{1,\sm})^{\ox k+1})\to H^0(\calO_{K^p}^{\lan,(0,-k)}\ox_{\calO_{K^p}^{\sm}}(\Omega_{K^p}^{1,\sm})^{\ox k+1})$ is an isomorphism.
    \end{enumerate}
\end{lemma}
\begin{proof}
    (1) This follows from the same arguments as \cite[Corollary 5.1.3 (i)]{Pan22}, and we briefly sketch a proof here. By Theorem \ref{thm:completedcohomology}, there is a natural isomorphism $H^0(\check \fl,\calO^{\lan}_{K^p})\isom \tilde{H}^0(K^p,C)^{\lan}$. The action of $D_p^\times$ on $\tilde{H}^0(K^p,C)$ factors through the reduced norm map. (For example, by definition $\tilde{H}^0(K^p,C)$ is the completed cohomology of the tower of connected components of finite level Shimura curves, then using the $p$-adic uniformization (Theorem \ref{thm:padicuniformization}) we are reduced to the description of the geometric connected components of Drinfeld towers of dimension $1$ over $\bbQ_p$, and then we are reduced to the Lubin--Tate side by the isomorphism of the two towers, which follows from the main result of \cite{strauch2008geometrically}.) From the explicit description of the $\theta_{\check{\frh}}$-action, we see $\theta_{\check{\frh}}\left( \begin{matrix} 1&0\\0&-1 \end{matrix} \right)\in \Lie D_p^\times\ox_{\bbQ_p}C$ acts trivially on $H^0(\check \fl,\calO^{\lan,(a,b)}_{K^p})$. But we know $\theta_{\check{\frh}}\left( \begin{matrix} 1&0\\0&-1 \end{matrix} \right)$ acts on $\calO^{\lan,(a,b)}_{K^p}$ via $a-b$, so that it also acts on $H^0(\check \fl,\calO^{\lan,(a,b)}_{K^p})$ via $a-b$. Hence $H^0(\check \fl,\calO^{\lan,(a,b)}_{K^p})=0$ if $a-b\neq 0$.

    (2) Since $\calO_{K^p}^{\lalg,(0,-k)}=W^{(0,-k)}\otimes \omega_{K^p}^{(0,-k),\sm}$ and $\calO_{K^p}^{\lalg,(0,-k)}\ox_{\calO_{K^p}^{\sm}}(\Omega_{K^p}^{1,\sm})^{\ox k+1}=W^{(0,-k)}\otimes \omega_{K^p}^{(-k-1,1),\sm}$( cf. Proposition \ref{prop:Wlalg}), it's enough to show that $H^0(\omega_{K^p}^{(0,-k),\sm})=0$ and $H^1(\omega_{K^p}^{(-k-1,1),\sm})=0$. Using Lemma \ref{lem:cohoofomegasm} and by the usual Serre duality on the proper spaces $\calS_{K^pK_p}$, it suffices to show $H^0(\omega_{K^p}^{(0,-k),\sm})=0$, equivalently $H^0(\calO_{K^p}^{\lalg,(0,-k)})=0$, if $k>0$. But there is an injection $H^0(\calO_{K^p}^{\lalg,(0,-k)})\inj H^0(\calO_{K^p}^{\lan,(0,-k)})$ and $H^0(\calO_{K^p}^{\lan,(0,-k)})$ vanishes for $k>0$ by (1).    

    (3) If $k\neq 0$, then $H^0(\calO_{K^p}^{\lan,(0,-k)})=0$ by (1) and (2). Hence it suffices to handle the case when $k=0$. In this case, as in (1), the action of $D_p^\times$ on $H^0(\calO_{K^p}^{\lan,(0,-k)})$ factors through the reduced norm map. Besides, the central element $\left( \begin{matrix} 1&0\\0&1\end{matrix} \right)\in \Lie(D_p^\times)\ox_{\bbQ_p}C$ acts trivially on $H^0(\calO_{K^p}^{\lan,(0,-k)})$ by Theorem \ref{thm:Senaction}. Hence the $D_p^\times$-action on $H^0(\calO_{K^p}^{\lan,(0,-k)})$ is smooth, which shows $H^0(\calO_{K^p}^{\lalg,(0,-k)})\isom  H^0(\calO_{K^p}^{\lan,(0,-k)})$.

    (4) By Proposition \ref{prop:dKvsurj}, $\bar d'_{K^p}:\calO_{K^p}^{\lan,(0,-k)}\ox_{\calO_{K^p}^{\sm}}(\Omega_{K^p}^{1,\sm})^{\ox k+1}\to \calO_{K^p}^{\lan,(-k-1,1)}(k+1)$ is surjective with kernel $\calO_{K^p}^{\lalg,(0,-k)}\ox_{\calO_{K^p}^{\sm}}(\Omega_{K^p}^{1,\sm})^{\ox k+1}$. The desired isomorphism follows from that $H^0(\calO_{K^p}^{\lan,(-k-1,1)}(k+1))=0$ by (1).
\end{proof}

\begin{lemma}\label{lem:injH1O0Ola}
The natural map $H^1(\calO_{K^p}^{\lalg,(0,-k)})\to H^1(\calO_{K^p}^{\lan})$ is injective.
\end{lemma}

\begin{proof}
    By Theorem \ref{thm:completedcohomology}, we can identify $H^1(\calO_{K^p}^{\lan})$ with $\tilde{H}^1(K^p,C)^{\lan}$. By Lemma \ref{lem:cohoofomegasm} and Proposition \ref{prop:Wlalg}, we can identify $H^1(\calO_{K^p}^{\lalg,(0,-k)})$ with (a translation of) the cohomology of automorphic line bundles on $\calS_{K^pK_p}$. Therefore, using the $p$-adic \'etale comparison theorem applied to $\dlim_{K_p}H^1_{\et}(\calS_{K^pK_p},W)$, where $W$ is a local system associated to the algebraic representation $W^{(k,0)}$ of $D_p^\times$, it suffices to show some non-completed cohomology groups is a subspace (the Sen weight $0$ part of the locally $W^{(0,-k)}$-algebraic vectors) of the completed cohomology group of the Shimura curves. This follows from the isomorphism
    \[\dlim_{K_p}H^1_{\et}(\calS_{K^pK_p},W)=\Hom_{\mathfrak{\gl}_2}(W^{(0,-k)},H^1(\calO_{K^p}^{\lalg,(0,-k)}))\]
    which can be proved using similar arguments as for the modular curves \cite[(4.3.4)]{Emerton2006interpolation}.

    We also sketch a (slightly different) argument for the $k=0$ case. We have $R\Gamma(\check\fl,R\Gamma(\gl_2,\calO_{K^p}^{\lan}))=R\Gamma(\gl_2,R\Gamma(\check\fl,\calO_{K^p}^{\lan}))$ (cf. \cite[Theorem 4.1]{gee2025modularity}). Hence there is a Grothendieck spectral sequence 
    \[E_2^{ij}=H^i(\check\fl,H^j(\gl_2,\calO_{K^p}^{\lan}))\Rightarrow \bbH^{i+j}(R\Gamma(\gl_2,R\Gamma(\check\fl,\calO_{K^p}^{\lan}))).\]
    Since $H^0(\gl_2,\calO_{K^p}^{\lan})=\calO_{K^p}^{\sm}$, we get an injection $ H^1(\check\fl,\calO_{K^p}^{\sm})\to \bbH^1(R\Gamma(\frg,R\Gamma(\check\fl,\calO_{K^p}^{\lan})))=H^0(\frg,H^1(\check\fl,\calO_{K^p}^{\lan}))$ where the last equality can be obtained using a version of \cite[Corollary 4.3.2]{Emerton2006interpolation} for Shimura curves as in the proof for \cite[(4.3.4)]{Emerton2006interpolation}. 
\end{proof}

Recall that $\sigma_{k,c}^{K^p}$ (Definition \ref{def:setHeckeeigenvalues}) is the set of cuspidal Hecke eigensystems that appear in the space of algebra automorphic forms of weight $V^{(k,0)}$ in $\calA_{\bar G,k}^{K^p}$ defined in \S\ref{sec:productformula}. We specialize to the Hecke eigenspaces for $\lambda\in \sigma_{k,c}^{K^p}$.
\begin{lemma}\label{lem:H0H1lambda}
Let $\lambda\in \sigma^{K^p}_{k,c}$. 
\begin{enumerate}[(1)]
    \item The natural map $H^1(\calO_{K^p}^{\lalg,(0,-k)})[\lambda]\to H^1(\calO_{K^p}^{\lan,(0,-k)})[\lambda]$ is injective.
    \item $H^0(\calO_{K^p}^{\lan,(0,-k)}\ox_{\calO_{\check\fl}}(\Omega^1_{\check\fl})^{\ox k+1})[\lambda]=0$.
    \item $H^0(\ker d'_{K^p})[\lambda]=0$.
\end{enumerate}

\end{lemma}
\begin{proof}

(1) Using similar arguments as in the proof for \cite[Corollary 5.1.3]{Pan22}, there is a short exact sequence 
\[0\rightarrow\Ext^1_{C[\frh]}(\chi,H^0(\calO_{K^p}^{\lan}))\rightarrow H^1(\calO_{K^p}^{\lan,\chi})\rightarrow H^1(\calO_{K^p}^{\lan})^{\chi}.\]
where $\chi$ denotes the character of $C[\frh]$ corresponding to the weight $(0,-k)$ of $\frh$. As in the proof for (1) of Lemma \ref{lem:H0vanish}. $H^0(\check \fl,\calO^{\lan}_{K^p})\isom \tilde{H}^0(K^p,C)^{\lan}$ and the action of $G(\bbA_f)$ on $\colim_{K^p}\tilde{H}^0(K^p,C)^{\lan}$ factors through the reduced norm of $G$ (cf. \cite[(4.2)]{Emerton2006interpolation}). Hence $H^0(\calO_{K^p}^{\lan})$ vanishes after being localized at the ideal of $\bbT^S$ corresponding to the cuspidal eigensystem $\lambda$. It's then not hard to see that $\Ext^1_{C[\frh]}(\chi,H^0(\calO_{K^p}^{\lan}))[\lambda]=0$. Then, by the above short exact sequence, we see that the map $H^1(\calO_{K^p}^{\lan,(0,-k)})[\lambda]\rightarrow H^1(\calO_{K_p}^{\lan})$ is an injection. Hence the map $H^1(\calO_{K^p}^{\lalg,(0,-k)})[\lambda]\to H^1(\calO_{K^p}^{\lan,(0,-k)})[\lambda]\subset H^1(\calO_{K^p}^{\lan})[\lambda]$ is injective by Lemma \ref{lem:injH1O0Ola}.

(2) Taking cohomology for the short exact sequence in the second row of (\ref{eq:DRdiagram}) for $\bar d_{K^p}$, we see that the statement is a direct consequence of (1) together with (3) of Lemma \ref{lem:H0vanish}. 

(3) As $\ker d'_{K^p}\subset \calO_{K^p}^{\lan,(0,-k)}\ox_{\calO_{\check\fl}}(\Omega^1_{\check\fl})^{\ox k+1}$, we deduce $H^0(\ker d'_{K^p})[\lambda]=0$ from (2).
\end{proof}

By (2) of Proposition \ref{prop:dKvsurj}, there is an exact sequence 
\begin{align}\label{equationsequencedKv}
    0\to \ker d_{K^p}\to \calO_{K^p}^{\lan,(0,-k)}\to \calO_{K^p}^{\lan,(0,-k)}\ox_{\calO_{K^p}^{\sm}}(\Omega_{K^p}^{1,\sm})^{\ox k+1}\to 0.
\end{align}
\begin{lemma}\label{lem:cohoexactseqdKv}
    There exists a $\bbT^S$-equivariant isomorphism of $D_p^{\times}$-representations:
    \begin{align*}
        H^0(\ker d_{K^p})\stackrel{\sim}{\rightarrow} H^0(\calO_{K^p}^{\lalg,(0,-k)}) 
    \end{align*}
    And $H^1(\ker d_{K^p})$ fits into a $\bbT^S$-equivariant short exact sequence of $D_p^{\times}$-representations:
    \begin{align*}
        0\to H^0(\calO_{K^p}^{\lalg,(0,-k)}\ox_{\calO_{K^p}^{\sm}}(\Omega_{K^p}^{1,\sm})^{\ox k+1}) \to H^1(\ker d_{K^p}) \to \ker H^1(d_{K^p})\to 0.
    \end{align*}
    where $ H^1(d_{K^p}):H^1(\calO_{K^p}^{\lan,(0,-k)})\to H^1(\calO_{K^p}^{\lan,(0,-k)}\ox_{\calO_{K^p}^{\sm}}(\Omega_{K^p}^{1,\sm})^{\ox k+1})$.
\end{lemma}
\begin{proof}
    By (3) of Lemma \ref{lem:H0vanish}, we know $H^0(\calO_{K^p}^{\lan,(0,-k)})=H^0(\calO_{K^p}^{\lalg,(0,-k)})$. In particular, $H^0(d_{K^p}):H^0(\calO_{K^p}^{\lan,(0,-k)})\to H^0(\calO_{K^p}^{\lan,(0,-k)}\ox_{\calO_{K^p}^{\sm}}(\Omega_{K^p}^{1,\sm})^{\ox k+1})$ is zero by Proposition \ref{prop:dRcohOKvalg} and the commutative diagram (\ref{eq:DRdiagram}).
    Taking the long exact sequence of (\ref{equationsequencedKv}) after applying $R\Gamma(\check{\fl},-)$, we see there is an isomorphism
    \[H^0(\ker d_{K^p})=H^0(\calO_{K^p}^{\lalg,(0,-k)})\]
    and a short exact sequence
    \[0\rightarrow H^0(\calO_{K^p}^{\lan,(0,-k)}\ox_{\calO_{K^p}^{\sm}}(\Omega_{K^p}^{1,\sm})^{\ox k+1})\rightarrow H^1(\ker d_{K^p})\rightarrow \ker H^1(d_{K^p})\rightarrow 0.\]
    Finally,
    \[H^0(\calO_{K^p}^{\lan,(0,-k)}\ox_{\calO_{K^p}^{\sm}}(\Omega_{K^p}^{1,\sm})^{\ox k+1})=H^0(\calO_{K^p}^{\lalg,(0,-k)}\ox_{\calO_{K^p}^{\sm}}(\Omega_{K^p}^{1,\sm})^{\ox k+1})\] 
    by (4) of Lemma \ref{lem:H0vanish}. 
\end{proof}

We consider also the short exact sequence
\begin{align*}
    0\to \ker d'_{K^p}\to \calO_{K^p}^{\lan,(0,-k)}\ox_{\calO_{\check\fl}}(\Omega^1_{\check\fl})^{\ox k+1}\to \calO_{K^p}^{\lan,(-k-1,1)}(k+1)\to 0.
\end{align*}
which is a twist of (\ref{equationsequencedKv}).
\begin{lemma}\label{lem:cohoexactseqdKvprime}
    There exist $\bbT^S$-equivariant isomorphisms of $D_p^{\times}$-representations:
    \begin{align*}
        &H^0(\ker d_{K^p}')\stackrel{\sim}{\rightarrow} H^0(\calO_{K^p}^{\lan,(0,-k)}\ox_{\calO_{\check\fl}}(\Omega^1_{\check\fl})^{\ox k+1})\\
        &H^1(\ker d_{K^p}')\stackrel{\sim}{\rightarrow} \ker H^1(d_{K^p}')
    \end{align*}
    where $ H^1(d_{K^p}'):H^1(\calO_{K^p}^{\lan,(0,-k)}\ox_{\calO_{\check\fl}}(\Omega^1_{\check\fl})^{\ox k+1})\to H^1(\calO_{K^p}^{\lan,(-k-1,1)}(k+1))$ is surjective.
\end{lemma}
\begin{proof}
    The proof is similar to the previous lemma using that $H^0(\calO_{K^p}^{\lan,(-k-1,1)}(k+1))=0$ for $k\geq 0$ by (1) of Lemma \ref{lem:H0vanish}. 
\end{proof}

\begin{proposition}\label{prop:Torvanishing}
Let $\lambda\in \sigma^{K^p}_{k,c}$ be a system of cuspidal Hecke eigenvalues. For $k\ge 0$, we have 
\begin{align*}
\mathrm{Tor}_i^{\calH(G)}(\calA_{\bar G,k}^{K^p}[\lambda],H^1_c(\calM_{\LT,\infty},\omega_{\calM_{\LT,\infty}}^{(0,-k),\sm}))=0,\\
\mathrm{Tor}_i^{\calH(G)}(\calA_{\bar G,k}^{K^p}[\lambda],H^1_c(\calM_{\LT,\infty},\omega_{\calM_{\LT,\infty}}^{(-k-1,1),\sm}))=0
\end{align*}
for $i\ge 1$.
\begin{proof}
By Proposition \ref{prop:productformula}, together with the vanishing results in Lemma \ref{lem:RpOLTsm} and the fact that the action of $\bbT^S$ in the isomorphism of Proposition \ref{prop:productformula} is semisimple (Theorem \ref{thm:decompositionautomorphicforms}), we see that 
\begin{align*}
    \mathrm{Tor}_i^{\calH(G)}(\calA_{\bar G,k}^{K^p}[\lambda],H^1_c(\calM_{\LT,\infty},\omega_{\calM_{\LT,\infty}}^{(0,-k),\sm}))=H^{1-i}(\calS_{K^p},\ker d_{\calS_{K^p}})[\lambda]
\end{align*}
and the similar result holds for $\omega_{\calM_{\LT,\infty}}^{(-k-1,1),\sm}$. This implies that the $\mathrm{Tor}$ groups vanish for $i\ge 2$. 

When $i=1$, by Lemma \ref{lem:H0H1lambda} (3), and Lemma \ref{lem:cohoexactseqdKv}, we see that $H^0(\calS_{K^p},\ker d_{\calS_{K^p}})[\lambda]=0$ and $H^0(\calS_{K^p},\ker d'_{\calS_{K^p}})[\lambda]=0$. This shows that the $\mathrm{Tor}$ groups also vanish for $i=1$.
\end{proof}
\end{proposition}

We can now describe the cohomology of $\DR^{\lalg},\DR,\DR'$ after localizing at a Hecke eigensystem $\lambda\in \sigma_{k,c}^{K^p}$.
\begin{theorem}\label{thm:H1DR}
Let $\lambda\in \sigma_{k,c}^{K^p}$. There is a $D_p^\times$-equivariant exact sequence
\begin{align*}    
    0\to \bbH^1(\check\fl,\DR^{\lalg})[\lambda]\to \bbH^1(\check\fl,\DR)[\lambda]\to \bbH^1(\check\fl,\DR')[\lambda]\to 0
\end{align*}
where we have
\begin{align*}
    \bbH^1(\check\fl,\DR)[\lambda]&\isom H^1_c(\calM_{\LT,\infty},\omega_{\calM_{\LT,\infty}}^{(0,-k),\sm})\ox_{\calH(G)} \calA_{\bar G,k}^{K^p}[\lambda],\\
    \bbH^1(\check\fl,\DR')[\lambda]&\isom H^1_c(\calM_{\LT,\infty},\omega_{\calM_{\LT,\infty}}^{(-k-1,1),\sm})\ox_{\calH(G)} \calA_{\bar G,k}^{K^p}[\lambda].\\
    \bbH^1(\check\fl,\DR^{\lalg})[\lambda]&\isom W^{(0,-k)}\ox_C R\Gamma_{\dR,c}(\calM_{\LT,\infty},\calO_{\calM_{\LT,\infty}}^{\sm})[1]\ox_{\calH(G)}^L \calA_{\bar G,k}^{K^p}[\lambda].
\end{align*}
\end{theorem}
\begin{proof}
    By Lemma \ref{lemmakerdKvdSKv}, Lemma \ref{lem:RpOLTsm} and Proposition \ref{prop:productformula}, there exists an isomorphism
    \[R\Gamma(\check\fl,\ker d_{K^p})=\calA_{\bar G,k}^{K^p} \ox_{\calH(G)}^L H^1_c(\calM_{\LT,\infty},\omega_{\calM_{\LT,\infty}}^{(0,-k),\sm})[-1].\]
    Hence $\bbH^1(\DR)=H^1(\ker d_{K^p})=\calA_{\bar G,k}^{K^p} \ox_{\calH(G)} H^1_c(\calM_{\LT,\infty},\omega_{\calM_{\LT,\infty}}^{(0,-k),\sm})$. Similarly, $\bbH^1(\DR')=H^1(\ker d_{K^p}')=\calA_{\bar G,k}^{K^p} \ox_{\calH(G)} H^1_c(\calM_{\LT,\infty},\omega_{\calM_{\LT,\infty}}^{(-k-1,1),\sm})$. And $\bbH^2(\DR)=H^2(\ker d_{K^p})=0$. By Lemma \ref{lem:cohoexactseqdKvprime}, 
    \[\bbH^0(\DR')=H^0(\ker d_{K^p}')=H^0(\calO_{K^p}^{\lan,(0,-k)}\ox_{\calO_{\check\fl}}(\Omega^1_{\check\fl})^{\ox k+1}).\]
    By (2) of Lemma \ref{lem:H0H1lambda}, $\bbH^0(\DR')[\lambda]=0$ for $\lambda\in\sigma_{k,c}^{K^p}$.

    Taking the long exact sequence of the hypercohomology for the complex $0\rightarrow \DR^{\lalg}\rightarrow\DR\rightarrow\DR'\rightarrow 0$, we get an exact sequence 
    \begin{align*}
        \bbH^0(\DR')\to \bbH^1(\DR^{\lalg})\to \bbH^1(\DR)\to \bbH^1(\DR')\to \bbH^2(\DR^{\lalg})\to 0.
    \end{align*}
    Using the decomposition (\ref{AKvdecomp}), the vanishing of $\bbH^2(\DR^{\lalg})[\lambda]$ (obtained from the vanishing of $H^0(\DR^{\lalg})[\lambda']$, when $\lambda'$ is assumed to be cuspidal, and duality, see Proposition \ref{prop:dRcohOKvalg}) and $\bbH^0(\DR')[\lambda]$, we get the desired short exact sequence $0\to \bbH^1(\DR^{\lalg})[\lambda]\to \bbH^1(\DR)[\lambda]\to \bbH^1(\DR')[\lambda]\to 0$. The description of $\bbH^1(\DR^{\lalg})[\lambda]$ follows from Lemma \ref{lemmaderhamlubintate}.
\end{proof}

\begin{remark}\label{remarkderhamtensor}
From the above theorem we see that (for simplicity we set $k=0$)
\begin{align*}
    R\Gamma_{\dR,c}(\calM_{\LT,\infty})\ox_{\calH(G)}^L \calA_{\bar G,0}^{K^p}[\lambda]
\end{align*}
is always isomorphic to the $\lambda$-isotypic part of the de Rham cohomology of the tower of quaternionic Shimura curves of finite levels, and so when $\lambda\in \sigma^{K^p}_{k,c}$ it concentrates only in degree $1$. The following formula in Proposition \ref{prop:HGLoxRHom}, together with \cite[Théorème 4.1, Théorème 5.8, \S 5.1]{CDN20} (the results in \textit{loc. cit.} are stated for a quotient of $\calM_{\LT,\infty}$, but one can use for example (\ref{euqationcompactinduction}) below to obtain the results for $\calM_{\LT,\infty}$ itself), is useful to do some explicit calculation of the representations. When $\calA_{\bar G,k}^{K^p}[\lambda]$ is a finite sum of copies of a supercuspidal representation of $\GL_2(\bbQ_p)$, only $H^1_{\dR,c}(\calM_{\LT,\infty})$ will contribute to the above term and the derived tensor product can be replaced by the usual tensor product. But when $\calA_{\bar G,k}^{K^p}[\lambda]$ is a sum of copies of a twist of the smooth Steinberg representation $\St_2^{\rm sm}$ of $\GL_2(\bbQ_p)$, both $H^1_{\dR,c}(\calM_{\LT,\infty})$ and $H^2_{\dR,c}(\calM_{\LT,\infty})$ will contribute to the above term nontrivially, and the derived tensor product cannot be replaced by the usual tensor product (see (\ref{equationderhamjacquetlanglands}) below). For example, if we denote $G^p$ to be the quotient of $G=\GL_2(\bbQ_p)$ by $\begin{pmatrix} p&0\\0&p\end{pmatrix}$, then $\St_2^{\rm sm}\ox^L_{\calH(G^p)}\St_2^{\rm sm}=1[0]$ and $1\ox^L_{\calH(G^p)}\St_2^{\rm sm}\isom 1[1]$. Finally, when $\calA_{\bar G,k}^{K^p}[\lambda]$ is a sum of smooth irreducible principal series representations of $\GL_2(\bbQ_p)$, the above term is zero. 
\end{remark}

The following formula describes some relation between taking $R\Hom_{\calH(G)}(-,-)$ and $-\otimes_{\calH(G)}^L-$. One can also find the result in \cite[\S III.3, Duality theorem]{schneider1997representation}, providing that the central character is fixed. See also \cite[Theorem V.5.1]{fargues2024geometrizationlocallanglandscorrespondence}.
\begin{proposition}\label{prop:HGLoxRHom}
Let $M,N$ be smooth representation of $G$, with $M$ being finitely generated over $\calH(G)$. Then there is a natural isomorphism 
\begin{align*}
    R\Hom_{\calH(G)}(\bbD_{\rm BZ}(M),N)\isom M\ox^L_{\calH(G)}N.
\end{align*}
where $\bbD_{\rm BZ}(M)$ is the smooth Bernstein--Zelevensky dual of $M$, defined as $\bbD_{\rm BZ}(M):=R\Hom_{\calH(G)}(M,\calH(G))$ (cf. \cite[IV.5]{bernstein1992notes}).
\begin{proof}
As $M$ is of finite type, there exists a projective resolution $P_{\bullet}\to V$ with each $P_i$ of the form $\cInd^G_K \sigma$ for some open compact subgroup $K\subset G$ and some finite dimensional smooth representation $\sigma$ of $K$. Here we use the fact that every subrepresentation of a finitely generated representation of $G$ is still finitely generated \cite[Proposition 32]{bernstein1992notes}. Objects such as $\cInd^G_K \sigma$ are flat over $\calH(G)$. Indeed, 
\begin{align*}
    \Hom_{\calH(G)}(\cInd^G_K\sigma,-)&=\Hom_C(\sigma,-)^K,\\
    \cInd^G_K\sigma\ox_{\calH(G)}-&=(\sigma\ox_C-)_K
\end{align*}
are all exact functors for smooth representations as $\sigma$ is finite dimensional and $K$ is compact. As $\bbD_{\rm BZ}(\cInd^G_K\sigma)=\cInd^G_K\sigma^{*}[0]$ where $\sigma^*=\Hom_C(\sigma,C)$, the isomorphism in \cite[Lemma 8.4]{mantovan2004certain}:
\begin{align*}
    \Hom_{\calH(G)}(\cInd^G_K\sigma,-)\isom \Hom_C(\sigma,-)^K\isom (\sigma^*\ox_C-)^K\aisom (\sigma^*\ox_C-)_K\isom \cInd^G_K\sigma^*\ox_{\calH(G)}-.
\end{align*}
implies that
\begin{align*}
    R\Hom_{\calH(G)}(\bbD_{\rm BZ}(M),N)&\isom R\Hom_{\calH(G)}(\bbD_{\rm BZ}(P_{\bullet}),N)\isom P_{\bullet}\ox_{\calH(G)}N\isom M\ox_{\calH(G)}^LN.
\end{align*}
We get the desired isomorphism.
\end{proof}
\end{proposition}

The following results are analogue of \cite[Theorem 5.5.7, Corollary 5.5.14]{PanII} in the quaternionic setting. 
\begin{theorem}\label{thm:kerI1classical}
For $k\ge 0$, there is a decomposition 
\begin{align*}
    \ker I^1\isom \bigoplus_{\lambda\in \sigma_{k}^{K^p}}\ker I^1\tilde{[\lambda]},
\end{align*}
where $(-)\tilde{[\lambda]}$ means the generalized eigenspace of the action of $\bbT^S$ with eigenvalue $\lambda$. When $\lambda\in \sigma_{k,c}^{K^p}$, $\ker I^1\tilde{[\lambda]}=\ker I^1{[\lambda]}$.
\end{theorem}

\begin{proof}
As $I^1=H^1(\bar d_{K_p}')\circ H^1(d_{K^p})$ and $H^1(d_{K^p})$ is surjective (by Proposition \ref{prop:dKvsurj}), we see $\ker I^1$ fits into a short exact sequence
\begin{align*}
    0\to \ker H^1(d_{K^p})\to \ker I^1\to \ker H^1(\bar{d}'_{K^p})\to 0.
\end{align*}
By Lemma \ref{lem:cohoexactseqdKv} and Proposition \ref{prop:productformula}, $\ker H^1(d_{K^p})$ is a $\bbT^S$-quotients of $\calA_{\bar G,k}^{K^p} \ox_{\calH(G)} H^1_c(\calM_{\LT,\infty},\omega_{\calM_{\LT,\infty}}^{(0,-k),\sm})$. Similarly, by Proposition \ref{prop:dRcohOKvalg} and Proposition \ref{prop:dKvsurj}, $\ker H^1(\bar{d}'_{K^p})$ receives a surjection from $H^1(\ker \bar d'_{K^p})=H^1(\calO_{K^p}^{\lalg,(0,-k)}\ox_{\calO_{K^p}^{\sm}}(\Omega^{1,\sm}_{K^p})^{\ox k+1})=\bbH^2(\DR^{\lalg})$. Hence the actions of $\bbT^S$ on both $\ker H^1(d_{K^p})$ and $\ker H^1(\bar{d}'_{K^p})$ are locally finite by Theorem \ref{thm:decompositionautomorphicforms} and Proposition \ref{prop:dRcohOKvalg}. We see $\ker I^1$ also admits a generalized eigenspace decomposition for the action of $\bbT^S$. Moreover, the action of $\bbT^S$ on $\ker I^1\tilde{[\lambda]}$ when $\lambda\in \sigma_{k,c}^{K^p}$ is semisimple since $\bbH^2(\DR^{\lalg})[\lambda]=0$ by Proposition \ref{prop:dRcohOKvalg}.
\end{proof}

\begin{corollary}\label{cor:kerI1DRDRprime}
Let $\lambda\in \sigma_{k,c}^{K^p}$. There are $D_p^\times$-equivariant exact sequences:
\[\begin{array}{c}
    0\to H^0(\check\fl,\calO_{K^p}^{\lalg,(0,-k)}\ox_{\calO_{K^p}^{\sm}}(\Omega_{K^p}^{1,\sm})^{\ox k+1})[\lambda]\to \bbH^1(\check\fl,\DR)[\lambda]\to \ker I^1[\lambda]\to 0,\\
    0\to H^1(\check\fl,\calO_{K^p}^{\lalg,(0,-k)})[\lambda]\to \ker I^1[\lambda]\to \bbH^1(\check\fl,\DR')[\lambda]\to 0.
\end{array}\]
\end{corollary}

\begin{proof}
From the proof of Theorem \ref{thm:kerI1classical} above and Lemma \ref{lem:cohoexactseqdKv}, we get the first desired short exact sequence 
\[0\to H^0(\calO_{K^p}^{\lalg,(0,-k)}\ox_{\calO_{K^p}^{\sm}}(\Omega_{K^p}^{1,\sm})^{\ox k+1})[\lambda]\to \bbH^1(\DR)[\lambda]\to \ker H^1(d_{K^p})[\lambda]=\ker I^1[\lambda]\to 0.\]

For the second short exact sequence, since $I=d_{K^p}'\circ \bar d_{K^p}$, similarly as in the proof of Theorem \ref{thm:kerI1classical}, we get a short exact sequence
\[0\to \ker H^1(\bar d_{K^p})\to  \ker I^1\to \ker H^1(d_{K^p}')\to 0.\]
By Proposition \ref{prop:dKvsurj} and (3) of Lemma \ref{lem:H0vanish}, there is a short exact sequence  
\[0\rightarrow H^0(\calO_{K^p}^{\lan,(0,-k)}\ox_{\calO_{\check\fl}}(\Omega^1_{\check\fl})^{\ox k+1} )\rightarrow H^1(\ker \bar d_{K^p})=H^1(\calO_{K^p}^{\lalg,(0,-k)})\rightarrow \ker H^1(\bar d_{K^p})\rightarrow 0.\] 
For $\lambda\in\sigma_{k,c}^{K^p}$, $H^0(\calO_{K^p}^{\lan,(0,-k)}\ox_{\calO_{\check\fl}}(\Omega^1_{\check\fl})^{\ox k+1})[\lambda]=0$ by (2) of Lemma \ref{lem:H0H1lambda}. Hence $\ker H^1(\bar d_{K^p})[\lambda]\simeq  H^1(\calO_{K^p}^{\lalg,(0,-k)})[\lambda]$. By Lemma \ref{lem:cohoexactseqdKvprime}, $\ker H^1(d_{K^p}')=H^1(\ker \bar d_{K^p})=\bbH^1(\DR')$. Put these together, we get the desired short exact sequence $0\to H^1(\calO_{K^p}^{\lalg,(0,-k)})[\lambda]\to \ker I^1[\lambda]\to \bbH^1(\DR')[\lambda]\to 0.$
\end{proof}

In the next section we will study representations of $D_p^\times$ appeared in the above cohomology groups.

\section{Applications to the $p$-adic local Langlands program for $D_p^\times$}\label{sec:padicLLD}
In this section, we will use the results in the previous sections to study the regular de Rham representations appearing in the locally analytic vectors of the completed cohomology of quaternion Shimura curves. 
\subsection{Locally analytic representations of $D_p^{\times}$}\label{subsec:repofDptimes}
We fix an irreducible smooth representation $\pi_p$ of $\GL_2(\bbQ_p)$. We assume that there exists a sufficiently small tame level $K^p$ and a system of Hecke eigenvalues $\lambda\in \sigma_{k,c}^{K^p}$ such that $\pi_{\lambda,p}=\pi_p$ (Definition \ref{def:setHeckeeigenvalues}). Fix $k\geq 1$ and recall we have defined the de Rham cohomology for the Lubin--Tate tower before Lemma \ref{lemmaderhamlubintate}.
\begin{definition}\label{def:tauc}
We define the following (locally analytic) representations of $D_p^\times$:
\begin{align*}
    &\tilde{\tau}:=H^1_c(\calM_{\LT,\infty},\omega_{\calM_{\LT,\infty}}^{(0,-k),\sm})\ox_{\calH(G)}\pi_p, \\
    &\tau_c :=H^1_c(\calM_{\LT,\infty},\omega_{\calM_{\LT,\infty}}^{(-k-1,1),\sm})\ox_{\calH(G)}\pi_p.
\end{align*}  
Here, as in Proposition \ref{prop:productformula}, the actions of $G=\GL_2(\bbQ_p)$ on the cohomology groups of $\calM_{\LT,\infty}$ is induced by the left $G$-action on $\calM_{\LT,\infty}$ in \S\ref{subsection:twotowers} twisted by the inverse transpose of $G$ as in Remark \ref{remarkuniformizationtwist}. We also set
\[\tau_p:=\mathrm{JL}(\pi_p)\]
to be the Jacquet--Langlands transfer of $\pi_p$ to $D_p^{\times}$ \cite{langlands1970automorphic} (if $\pi_p$ is not a discrete series representation, then we set $\tau_p=\mathrm{JL}(\pi_p)=0$).
\end{definition}
\begin{remark}
    \begin{enumerate}[(1)]
        \item From our normalization of the group actions, one can check that the central characters of $\widetilde{\tau},\tau_c,\tau_p$ coincide with that of $\pi_p$.
        \item If $\pi_p$ is a discrete series representation (essentially square-integrable), the representation $\tau_p$ is an irreducible smooth representation of $D_p^{\times}$. We will prove in \S\ref{subsectionnonvanishing} that $\tilde{\tau}$ and $\tau_c $ are never zero. These representations of $D_p^{\times}$ in Definition \ref{def:tauc} are local, namely they depend only on $\pi_p$ but not on $\lambda$.
    \end{enumerate}
\end{remark}
Since $\pi_{\lambda,p}=\pi_p$, by Theorem \ref{thm:decompositionautomorphicforms} and its proof, we have
\begin{align}\label{equationmultiplicity}
    \calA_{\bar G,k}^{K^p}[\lambda]=\pi_p\otimes_C (\pi_{\lambda}^{p,\infty})^{K^p}.
\end{align}
The following results are reformulations of Theorem \ref{thm:H1DR} and Corollary \ref{cor:kerI1DRDRprime} in terms of these representations of $D_p^{\times}$.
\begin{proposition}\label{propositionexactseqDptimesrep}
    Let $\tau_p$, $\tilde{\tau}$, and $\tau_c$ be the representations of $D_p^{\times}$ in Definition \ref{def:tauc}.
    \begin{enumerate}[(1)]
        \item There exists a $D_p^\times$-equivariant exact sequence 
        \begin{align}\label{equationexactsequencetautildetuac}
            0\to (W^{(0,-k)}\ox_C \tau_p )^{\oplus 2}\to  \tilde{\tau}\to \tau_c \to 0.
        \end{align}
        \item There are $D_p^\times$-equivariant isomorphisms:
        \[\begin{array}{c}
        0\to W^{(0,-k)}\ox_C \tau_p\otimes_C (\pi_{\lambda}^{p,\infty})^{K^p} \to  \tilde{\tau}\otimes_C (\pi_{\lambda}^{p,\infty})^{K^p} \to \ker I^1[\lambda]\to 0,\\
        0\to W^{(0,-k)}\ox_C \tau_p\otimes_C (\pi_{\lambda}^{p,\infty})^{K^p} \to \ker I^1[\lambda]\to \tau_c \otimes_C (\pi_{\lambda}^{p,\infty})^{K^p}\to 0.
    \end{array}\]
    \end{enumerate}
\end{proposition} 
\begin{proof}
    (1) Using the $p$-adic uniformization (see also Remark \ref{remarkderhamtensor}), we have
    \begin{align}\label{equationderhamjacquetlanglands}
    R\Gamma_{\dR,c}(\calM_{\LT,\infty},\calO_{\calM_{\LT,\infty}}^{\sm})\ox_{\calH(G)}^L\pi_p=\tau_p^{\oplus 2}[-1].
    \end{align}    
    By Theorem \ref{thm:H1DR}, Proposition \ref{prop:productformula}, Lemma \ref{lemmaderhamlubintate} and (\ref{equationmultiplicity}), there is a $\bbT^S$-equivariant isomorphism
    \begin{align*}
        0\rightarrow(W^{(0,-k)}\ox_C \tau_p )^{\oplus 2} \otimes_C (\pi_{\lambda}^{p,\infty})^{K^p}\rightarrow \widetilde{\tau}\otimes_C (\pi_{\lambda}^{p,\infty})^{K^p}\rightarrow \tau_c\otimes_C (\pi_{\lambda}^{p,\infty})^{K^p}\rightarrow 0.
    \end{align*}
    Take colimits for all $K^p$ small enough and apply $\Hom_{G(\bbA_f^p)}(\pi_{\lambda}^{p,\infty},-)$ to the above exact sequence, we get the desired exact sequence in (1). 

    (2) By the Hodge--Tate decomposition and the spectral decomposition of the classical de Rham cohomology of quaternion Shimura curves (cf. \cite[Theorem 6.2.3]{lan2023rham} and the proof for Proposition \ref{prop:dRcohOKvalg}), we can get isomorphisms
    \begin{align*}
        H^1(\calS_{K^p},\omega^{(0,-k),\sm}_{\calS_{K^p}})[\lambda]\isom  \tau_p\otimes_C (\pi_{\lambda}^{p,\infty})^{K^p}\isom 
        H^0(\calS_{K^p},\omega^{(-k-1,1),\sm}_{\calS_{K^p}})[\lambda].
    \end{align*}
    The two short exact sequences in (2) follow from Corollary \ref{cor:kerI1DRDRprime} (using Lemma \ref{lem:cohoofomegasm} and (\ref{equationOKplalg})).
\end{proof}
\begin{remark}\label{remarkHodgefiltration}
    The map $H^0(\check\fl,\calO_{K^p}^{\lalg,(0,-k)}\ox_{\calO_{K^p}^{\sm}}(\Omega_{K^p}^{1,\sm})^{\ox k+1})[\lambda]\to \bbH^1(\check\fl,\DR)[\lambda]$ in Corollary \ref{cor:kerI1DRDRprime} is the composition 
    \[H^0(\check\fl,\calO_{K^p}^{\lalg,(0,-k)}\ox_{\calO_{K^p}^{\sm}}(\Omega_{K^p}^{1,\sm})^{\ox k+1})[\lambda]\to \bbH^1(\check\fl,\DR^{\lalg})[\lambda]\to \bbH^1(\check\fl,\DR)[\lambda],\] where the map $H^0(\calO_{K^p}^{\lalg,(0,-k)}\ox_{\calO_{K^p}^{\sm}}(\Omega_{K^p}^{1,\sm})^{\ox k+1})[\lambda]\to \bbH^1(\DR^{\lalg})[\lambda]$ gives the Hodge filtration, see (the proof of) Proposition \ref{prop:dRcohOKvalg}. Let $\rho:\Gal(\overline{\bbQ}/\bbQ)\rightarrow \GL_2(\overline{E})$ be the Galois representation attached to $\lambda$ and let $\rho_p:=\rho|_{\Gal(\overline{\bbQ_p}/\bbQ_p)}$. Let $K^p$ vary and apply $\Hom_{G(\bbA_f^p)}(\pi_{\lambda}^{p,\infty},-)$ as in the proof of Proposition \ref{propositionexactseqDptimesrep} above, we see the last map, in the notation of Proposition \ref{propositionexactseqDptimesrep} (2),
    \[W^{(0,-k)}\ox_C \tau_p\otimes_C (\pi_{\lambda}^{p,\infty})^{K^p} \to  W^{(0,-k)}\ox_C \tau_p\otimes_CD_{\dR}(\rho_p)_C\otimes_C (\pi_{\lambda}^{p,\infty})^{K^p}\subset \tilde{\tau}\otimes_C (\pi_{\lambda}^{p,\infty})^{K^p}\]
    is induced by the Hodge filtration inside $D_{\dR}(\rho_p)$ after applying the de Rham-\'etale comparison for (\ref{eq:rhoisotypic}).
\end{remark}
\begin{remark}
    The short exact sequence (\ref{equationexactsequencetautildetuac}) is stated for $\pi_p$ being a local component of an automorphic representation. But one can also prove directly that the same sequence is exact for any irreducible smooth representation $\pi_p$ of $\GL_2(\bbQ_p)$ with $\dim \pi_p>1$ (and $\widetilde{\tau},\tau_c$ defined as in Definition \ref{def:tauc}) using the results on the compactly supported de Rham cohomology of $\calM_{\LT,\infty}$.
\end{remark}
We will discuss later the relation between these sequences in the above proposition and the $p$-adic local Langlands program for $D_p^{\times}$.

\subsection{Representation-theoretic properties of $\tau_c $}\label{sec:representationproperties}
Let $\pi_p$ be the irreducible smooth representation of $\GL_2(\bbQ_p)$ fixed in \S\ref{subsec:repofDptimes}. Let $\tau_c $ be the representation of $D_p^{\times}$ in Definition \ref{def:tauc}. In this subsection, we will study some representation-theoretic properties of $\tau_c $.

The organization of this part is as follows. First, we will show that $\tau_c$ is always non-zero. Using the geometric definition of $\tau_c $ involving Lubin--Tate spaces, we find a natural $E$-structure $\tau_{c,E}$ of $\tau_c $, where $E$ is a sufficiently large finite extension of $\bbQ_p$. Then we show that $\tau_{c,E}$ is admissible. We will also show that $\tau_{c,E}$ has Gelfand-Kirillov dimension $1$.  

\subsubsection{Non-vanishing of $\tau_c $}\label{subsectionnonvanishing}
We start by showing that $\tau_c$ in Definition \ref{def:tauc} is non-zero for any $\lambda\in \sigma_{0,c}^{K^p}$ (Definition \ref{def:setHeckeeigenvalues}). 

Let $\calM_{\LT,0}^{(0)}$ be the component of $\calM_{\LT,0}$ such that the quasi-isogeny in the moduli problem for $\calM_{\LT,0}$ (\S\ref{subsection:twotowers}) is of height zero. Abstractly $\calM_{\LT,0}^{(0)}$ is an open unit disk over $C$, and we let $u\in H^0(\calM_{\LT,0}^{(0)},\calO_{\calM_{\LT,0}^{(0)}})$ be a coordinate. Write $\pi_{\LT,\GM,0}^{(0)}:\calM_{\LT,0}^{(0)}\to \check{\fl}\isom \bbP^{1,\ad}_C$ as the restriction of the Gross--Hopkins period map of $\calM_{\LT,0}$ on $\calM_{\LT,0}^{(0)}$. From the discussion in \cite[\S 25]{GH94}, we know there exist two rigid analytic functions $\phi_0,\phi_1\in H^0(\calM_{\LT,0}^{(0)},\calO_{\calM_{\LT,0}^{(0)}})$ on $\calM_{\LT,0}^{(0)}$ with no common zeros, such that $\pi_{\calM_{\LT,0}^{(0)}}(x)=[\phi_0(x):\phi_1(x)]\in\bbP^{1,\mathrm{ad}}_C$ for $x\in \calM_{\LT,0}^{(0)}$. Moreover, all zeros of $\phi_0,\phi_1$ are simple.

Let $\calM_{\LT,\infty}^{(0)}$ be the preimage of $\calM_{\LT,0}^{(0)}$ under the natural projection $\calM_{\LT,\infty}\rightarrow \calM_{\LT,0}$. Let $\GL_2(\bbQ_p)^{\circ}=\{g\in\GL_2(\bbQ_p)\mid \det(g)\in\bbZ_p^{\times}\}$. Then $\GL_2(\bbQ_p)^{\circ}\times \calO_{D_p}^{\times}$ acts on $\calM_{\LT,\infty}^{(0)}$ (cf. the end of \cite[\S 23]{GH94}).
\begin{proposition}\label{prop:phi0}
Then multiplication by $\phi_0$ on $ \calO_{\calM_{\LT,\infty}^{(0)}}^{\sm}$ induces an exact sequence
\begin{align*}
    0\to \calC^{\sm}_c(\GL_2(\bbQ_p)^{\circ},C)\to H^1_c(\calM_{\LT,\infty}^{(0)},\calO_{\calM_{\LT,\infty}^{(0)}}^{\sm}) \ov{\times \phi_0}\to H^1_c(\calM_{\LT,\infty}^{(0)},\calO_{\calM_{\LT,\infty}^{(0)}}^{\sm})\to 0.
\end{align*}
\begin{proof}
We first show that multiplication by $\phi_0$ on the sheaf of Kähler differentials $\Omega_{\calM_{\LT,0}^{(0)}}^1$ induces an exact sequence 
\begin{align}\label{equationexactseqphi0}
    0\to H^0(\calM_{\LT,0}^{(0)},\Omega^1_{\calM_{\LT,0}^{(0)}})\ov{\times \phi_0}\to H^0(\calM_{\LT,0}^{(0)},\Omega^1_{\calM_{\LT,0}^{(0)}})\to \prod_{x\in (\pi_{\LT,\GM,0}^{(0)})^{-1}([0:1])}C_x\to 0
\end{align}
where $[0:1]\in \check\fl=\bbP^{1,\mathrm{ad}}_C$ and $C_x$ is the skyscraper sheaf at $x\in (\pi_{\LT,\GM,0}^{(0)})^{-1}([0:1])\subset \calM_{\LT,0}^{(0)}$ with value $C$. Indeed, if we pick a Stein covering $\{U_n\}$ of $\calM_{\LT,0}^{(0)}$ by affinioids (i.e., $U_n\Subset U_{n+1}$ for $n\ge 0$ and $\calM_{\LT,0}^{(0)}=\bigcup_n U_n$), then as $(\pi_{\LT,\GM,0}^{(0)})^{-1}([0:1])\simeq \GL_2(\bbQ_p)^{\circ}/\GL_2(\bbZ_p)$ is discrete, $(\pi_{\LT,\GM,0}^{(0)})^{-1}([0:1])\cap U_n$ is finite for each $n$ being discrete and quasi-compact. Since $\phi_0$ vanishes exactly on $(\pi_{\LT,\GM,0}^{(0)})^{-1}([0:1])$ with simple zeros and $\Omega^1_{\calM_{\LT,0}^{(0)}}\simeq \calO_{\calM_{\LT,0}^{(0)}}$ as a coherent $\calO_{\calM_{\LT,0}^{(0)}}$-module, we see that for each $n$, the sequence
\begin{align*}
    0\to H^0(U_n,\Omega^1_{\calM_{\LT,0}^{(0)}})\ov{\times \phi_0}\to H^0(U_n,\Omega^1_{\calM_{\LT,0}^{(0)}})\to \prod_{x\in U_n:\pi_{\LT,\GM,0}(x)=[0:1]}C_x\to 0.
\end{align*}
is exact. Taking limits, as $R^1\lim_n H^0(U_n,\Omega^1_{\calM_{\LT,0}^{(0)}})=0$ ($H^0(\calM_{\LT,0}^{(0)},\Omega^1_{\calM_{\LT,0}^{(0)}})$ is nuclear Fréchet), we deduce the exactness of the original sequence (\ref{equationexactseqphi0}). Taking the strong dual, by Serre duality \cite{beyer1997serre,chiarellotto2006duality,van1992serre} and \cite[Proposition 9.11]{schneider2013nonarchimedean} (these results do not depend on whether the coefficient field is spherical complete), we deduce that there exists an exact sequence
\begin{align*}
    0\to \bigoplus_{x\in (\pi_{\LT,\GM,0}^{(0)})^{-1}([0:1])}C_x\to H^1_c(\calM_{\LT,0}^{(0)},\calO_{\calM_{\LT,0}^{(0)}})\to H^1_c(\calM_{\LT,0}^{(0)},\calO_{\calM_{\LT,0}^{(0)}})\to 0.
\end{align*}
Let $\calM_{\LT,n}^{(0)}\subset \calM_{\LT,n}$ be the preimage of $\calM_{\LT,0}^{(0)}$ under the finite \'etale map $\pi_n:\calM_{\LT,n}\to \calM_{\LT,0}$. Since $\pi_n$ is \'etale, the pullback via $\pi_n^*$ of the sequence $0\rightarrow \Omega^1_{\calM_{\LT,0}^{(0)}}\rightarrow \Omega^1_{\calM_{\LT,0}^{(0)}}\rightarrow \prod_{x\in (\pi_{\LT,\GM,0}^{(0)})^{-1}([0:1])}C_x\rightarrow 0$ is still exact. Hence the same argument as above shows that multiplication by $\phi_0$ on $\calO_{\calM_{\LT,n}^{(0)}}$ induces an exact sequence
\[0\to \bigoplus_{x\in (\pi_{\LT,\GM,n}^{(0)})^{-1}([0:1])}C_x\to H^1_c(\calM_{\LT,n}^{(0)},\calO_{\calM_{\LT,n}^{(0)}})\to H^1_c(\calM_{\LT,n}^{(0)},\calO_{\calM_{\LT,n}^{(0)}})\to 0\]
where $\pi_{\LT,\GM,n}^{(0)}= \pi_{\LT,\GM,0}^{(0)}\circ\pi_n$. Taking colimits on the level $n$, using Proposition \ref{lem:RpOLTsm}, we get an exact sequence
\begin{align*}
    0\to \varinjlim_n\bigoplus_{x\in (\pi_{\LT,\GM,n}^{(0)})^{-1}([0:1])}C_x\to H^1_c(\calM_{\LT,\infty}^{(0)},\calO_{\calM_{\LT,\infty}}^{\sm}) \ov{\times \phi_0}\to H^1_c(\calM_{\LT,\infty}^{(0)},\calO_{\calM_{\LT,\infty}}^{\sm})\to 0.
\end{align*}
Finally, we can identify $\varinjlim_n\bigoplus_{x\in (\pi_{\LT,\GM,n}^{(0)})^{-1}([0:1])}C$ as the space of compactly supported smooth functions on the locally profinite set $\varprojlim_{n}(\pi_{\LT,\GM,n}^{(0)})^{-1}([0:1])=(\pi_{\LT,\GM})^{-1}([0:1])\cap \calM_{\LT,\infty}^{(0)}=\GL_2(\bbQ_p)^{\circ}$.
\end{proof}
\end{proposition}
We remark that the exact sequence in Proposition \ref{prop:phi0} is not $\GL_2(\bbQ_p)^{\circ}$-equivariant. (Indeed, if so, then $\phi_0$ will induce a $\GL_2(\bbQ_p)^{\circ}$-equivariant map $\calM_{\LT,\infty}^{(0)}\to \bbA^{1,\an}_C$, which will descend to $\check{\fl}\to \bbA^1$ as $\calM_{\LT,\infty}^{(0)}\to \check{\fl}$ is a $\GL_2(\bbQ_p)^{\circ}$-torsor. But such a map must be a constant.) However, a variant of this exact sequence is $\GL_2(\bbQ_p)$-equivariant and is useful, as we will see in Proposition \ref{prop:t00} below.

By \cite[\S 25, (25.14)]{GH94}, there exists an element $t_{00}\in \check{\frg}/\check{\frz}$, such that the action of $t_{00}$ on $\calO_{\calM_{\LT,0}^{(0)}}$ is given by $t_{00}=\frac{\phi_0(u)\phi_1(u)}{\epsilon(u)}\frac{\partial}{\partial u}$. Here, $\check\frz$ is the center of $\check{\frg}=\Lie D_p^\times\ox_{\bbQ_p}C$, which acts trivially on $\calO_{\check{\fl}}$, and $\epsilon=\phi_0\phi_1'-\phi_0'\phi_1$. Moreover, $t_{00}$ acts naturally on the $D_p^{\times}$-equivariant sheaf $\calO_{\calM_{\LT,n}}$ by the derivation of the $D_p^{\times}$-action.  
\begin{lemma}\label{lem:t00}
    Let $t_{00}\in \check{\frg}/\check{\frz}$ be as above.
    \begin{enumerate}[(1)]
        \item $t_{00}=c\begin{pmatrix}1&0\\0&-1 \end{pmatrix}\in \check{\frg}$ for a non-zero $c\in C^{\times}$. Here $\check{\frg}\simeq \mathfrak{gl}_{2,C}$ via (\ref{equationsplittingD}).
        \item The action map $t_{00}:\calO_{\calM_{\LT,\infty}}^{\sm}\rightarrow \calO_{\calM_{\LT,\infty}}^{\sm}$ is a composition
        \[\calO_{\calM_{\LT,\infty}}^{\sm}\stackrel{d_{\LT}}{\rightarrow}\Omega_{\calM_{\LT,\infty}}^{1,\sm}\stackrel{\Phi}{\rightarrow}\calO_{\calM_{\LT,\infty}}^{\sm} \]
        of morphisms of $\GL_2(\bbQ_p)$-equivariant abelian sheaves on $\calM_{\LT,\infty}$, where $d_{\LT}$ is the differential operator in \S\ref{theoremdifferentialoperatorsDr} and $\Phi$ is a morphism of $\calO_{\calM_{\LT,\infty}}^{\sm}$-modules whose restriction to $\Omega_{\calM_{\LT,\infty}^{(0)}}^{1,\sm}$ is given by $fdu\mapsto \frac{\phi_0\phi_1}{\epsilon}f$.
    \end{enumerate}
\end{lemma}
\begin{proof}
    By definition, $t_{00}$ is a derivation on $\calO_{\calM_{\LT,n}}$. By \cite[\href{https://stacks.math.columbia.edu/tag/00RO}{Tag 00RO}]{stacks-project}, the derivation is determined by an element $\Phi_{\LT,n}\in \Hom_{\calO_{\calM_{\LT,n}}}(\Omega_{\calM_{\LT,n}}^1,\calO_{\calM_{\LT,n}})$ so that the action of $t_{00}$ is given by
    \[ t_{00}: \calO_{\calM_{\LT,n}}\stackrel{d}{\rightarrow} \Omega_{\calM_{\LT,n}}^1\stackrel{\Phi_{\LT,n}}{\rightarrow}\calO_{\calM_{\LT,n}}.\]
    Since $\pi_{\LT,\GM,n}:\calO_{\calM_{\LT,n}}\rightarrow \check{\fl}$ is \'etale (thus $\Omega_{\calM_{\LT,n}}^1=(\pi_{\LT,\GM,n})^*\Omega^1_{\check\fl}$) and domimnant, and $t_{00}$ acts on $\calO_{\check\fl}$, we see that the element $\Phi_{\LT,n}$ is the image of an element 
    \[\Phi\in \Hom_{\calO_{\check\fl}}(\Omega^1_{\check\fl},\calO_{\check\fl})\subset\Hom_{\calO_{\check\fl}}(\Omega^1_{\check\fl},(\pi_{\LT,\GM,n})_*\calO_{\calM_{\LT,n}})= \Hom_{\calO_{\calM_{\LT,n}}}(\Omega_{\calM_{\LT,n}}^1,\calO_{\Omega_{\calM_{\LT,n}}})\]
    which gives the action of $t_{00}$ on $\calO_{\check\fl}$ by \textit{loc. cit.} In other words, first order differential operators on $\calO_{\calM_{\LT,n}}$ stabilizing $(\pi_{\LT,\GM,n})^{-1}\calO_{\check\fl}$ are in natural bijection with differential operators on $\calO_{\check\fl}$.

    To prove (1), which was already noticed after the definition of $t_{00}$ in \cite[\S 25]{GH94}, we calculate the action of $h:=\begin{pmatrix}1&0\\0&-1 \end{pmatrix}$ on $\calO_{\check\fl}$. Write $\check\fl=\GL_2/\bar B=N\bar B/\bar B\cup sN\bar B/\bar B$. Let $z$ be the coordinate of $N=\{\begin{pmatrix}1&z\\0&1\end{pmatrix}\}$ and $z'=\frac{1}{z}$ be the coordinate $z'$ of $sN=\{\begin{pmatrix}0&1\\1&0\end{pmatrix}\begin{pmatrix}1&z'\\0&1\end{pmatrix}\}$. Since $\begin{pmatrix}t&0\\0&t^{-1}\end{pmatrix}^{-1}\begin{pmatrix}1&z\\0&1\end{pmatrix}=\begin{pmatrix}1&t^{-2}z\\0&1\end{pmatrix}\begin{pmatrix}t&0\\0&t^{-1}\end{pmatrix}^{-1}$, we see $h$ acts on $\calO(N)$ via $f\mapsto \frac{d}{d\epsilon}f(\exp(\epsilon)^{-2}z)=-2z\frac{d}{dz}f$. Similarly, $h$ acts on $\calO(sN)$ via $f(z')\mapsto 2z'\frac{d}{dz'}f$. Hence the corresponding element in $\Hom_{\calO_{\check\fl}}(\Omega^1_{\check\fl},\calO_{\check\fl})$ is given by $dz\mapsto -2z$ on $N\bar B/\bar B$ and $dz'\mapsto 2z'$ on $sN\bar B/\bar B$. After possibly changing variables, we may assume that $z=\frac{\phi_0(u)}{\phi_1(u)}$ where $u$ is the coordinate of $\calM_{\LT,0}^{(0)}$. Then $dz=d\frac{\phi_0(u)}{\phi_1(u)}=\frac{\phi_0'(u)\phi_1(u)-\phi_1'(u)\phi_0(u)}{\phi_1(u)^2}du$. Hence $h$ corresponds to the morphism sending $du=-\frac{\phi_1^2}{\epsilon}dz$ to $-\frac{\phi_1^2}{\epsilon}\cdot (-2\frac{\phi_0}{\phi_1})=2\frac{\phi_0\phi_1}{\epsilon}$, which equals to $t_{00}$ up to a scalar.

    The decomposition in (2) of the $t_{00}$-action follows from the discussions above. The morphism $\Phi$ is $\GL_2(\bbQ_p)$-equivariant since it is the pullback of a morphism in $\Hom_{\calO_{\check\fl}}(\Omega^1_{\check\fl},\calO_{\check\fl})$ along the $\GL_2(\bbQ_p)$-equivariant map $\pi_{\LT,\GM}$. The composition $t_{00}:\calO_{\calM_{\LT,\infty}}^{\rm sm}\rightarrow \Omega_{\calM_{\LT,\infty}}^{1,\sm}\to \calO_{\calM_{\LT,\infty}}^{\rm sm}$ is a priori $\GL_2(\bbQ_p)$-equivariant since the action of $\GL_2(\bbQ_p)$ commutes with the action of the Lie algebra of $D^{\times}_p$.
\end{proof}
The element $t_{00}$ also acts on the $D_{p}^{\times}$-equivariant sheaf $\Omega^{1,\sm}_{\calM_{\LT,\infty}}$ by derivation. 
\begin{proposition}\label{prop:t00}
The action of $t_{00}$ on $H^1_c(\calM_{\LT,\infty},\Omega_{\calM_{\LT,\infty}}^{1,\sm})$ induces a $\GL_2(\bbQ_p)$-equivariant exact sequence 
\begin{align*}
    0\to \ker t_{00}\to H^1_c(\calM_{\LT,\infty},\Omega_{\calM_{\LT,\infty}}^{1,\sm})\to H^1_c(\calM_{\LT,\infty},\Omega_{\calM_{\LT,\infty}}^{1,\sm})\to \coker t_{00}\to 0
\end{align*}
with $\coker t_{00}\isom H^2_{\dR,c}(\calM_{\LT,\infty})$, and there exists a $\GL_2(\bbQ_p)$-equivariant short exact sequence
\begin{align*}
    0\to \calC^\sm_c(\GL_2(\bbQ_p),C)^{\oplus 2}\to \ker t_{00}\to H^1_{\dR,c}(\calM_{\LT,\infty})\to 0.
\end{align*}
\begin{proof}
By (2) of Lemma \ref{lem:t00}, the map $H^0(t_{00}):H^0(\calM_{\LT,\infty},\calO_{\calM_{\LT,\infty}}^{\sm})\to H^0(\calM_{\LT,\infty},\calO_{\calM_{\LT,\infty}}^{\sm})$ can be written as the composition of $H^0(d_{\LT}):H^0(\calM_{\LT,\infty},\calO_{\calM_{\LT,\infty}}^{\sm})\rightarrow H^0(\calM_{\LT,\infty},\Omega^{1,\sm}_{\calM_{\LT,\infty}})$ and $H^0(\Phi): H^0(\calM_{\LT,\infty},\Omega^{1,\sm}_{\calM_{\LT,\infty}})\rightarrow H^0(\calM_{\LT,\infty},\calO_{\calM_{\LT,\infty}}^{\sm})$. 

We first work on $\calM_{\LT,n}$, and by abuse of notation we let $d,\Phi$ be the corresponding maps on $\calM_{\LT,n}$. By the kernel-cokernel lemma, we see $t_{00}$ induces an exact sequence
\begin{align}\label{equationkernelcokernel}
    0\to \ker H^0(d)\to \ker H^0(t_{00})\to \ker H^0(\Phi)\to \coker H^0(d)\to \coker H^0(t_{00})\to \coker H^0(\Phi)\to 0.
\end{align}
Locally on $\calM_{\LT,n}$, as on $\calM_{\LT,n}^{(0)}$, $\Omega_{\calM_{\LT,n}}^1\simeq \calO_{\calM_{\LT,n}}$ as $\calO_{\calM_{\LT,n}}$-modules and $\Phi$ is then given by multiplying a non-zero function as $\frac{\phi_0\phi_1}{\epsilon}$. Since $\calM_{\LT,n}$ is locally integral, we see $H^0(\Phi)$ is injective on $H^0(\calM_{\LT,n},\Omega^{1}_{\calM_{\LT,n}})$. Hence $\ker H^0(\Phi)=0$. Then $\ker H^0(t_{00})=\ker H^0(d)$. By the vanishing of higher cohomology groups of coherent sheaves on the Stein space $\calM_{\LT,n}$, we know $\ker H^0(t_{00})=\ker H^0(d)=H^0_{\dR}(\calM_{\LT,n})$ and $\coker H^0(d)=H^1_{\dR}(\calM_{\LT,n})$. Moreover, $\coker H^0(\Phi|_{\calM_{\LT,n}^{(0)}})=\prod_{x\in \calM_{\LT,n}^{(0)},\pi_{\LT,\GM,n}(x)\in\{[0:1],[1:0]\}}C_x$ by using similar methods as in the proof of Proposition \ref{prop:phi0} and that $\phi_0,\phi_1$ have no common zeros while $\epsilon$ is invertible. 

Now we take the strong dual (for the Fr\'echet topology on the cohomology groups) and take colimits on $n$, so that by Serre duality the original maps $H^0(t_{00}),H^0(d),H^0(\Phi)$ induce
\begin{align*}
    H^0(t_{00})^{\vee}:H^1_c(\calM_{\LT,\infty},\Omega_{\calM_{\LT,\infty}}^{1,\sm})\ov{H^0(\Phi)^{\vee}}\rightarrow H^1_c(\calM_{\LT,\infty},\calO_{\calM_{\LT,\infty}}^{\sm})\ov{H^0(d)^{\vee}}\to H^1_c(\calM_{\LT,\infty},\Omega_{\calM_{\LT,\infty}}^{1,\sm}).
\end{align*}
By the discussions on groups in (\ref{equationkernelcokernel}) above and the duality, we see $\coker H^0(\Phi)^{\vee}=0$, and $\ker H^0(\Phi)^{\vee}=\cInd_{\GL_2(\bbQ_p)^{\circ}}^{\GL_2(\bbQ_p)}\ker H^0(\Phi|_{\calM_{\LT,\infty}^{(0)}})^{\vee}=\calC^\sm_c((\pi_{\LT,\GM})^{-1}(\{[0:1],[1:0]\}),C)=\calC^\sm_c(\GL_2(\bbQ_p),C)^{\oplus 2}$ by Proposition \ref{prop:phi0} and the $\GL_2(\bbQ_p)$-equivariance of $\Phi$. Also $\ker H^0(d)^{\vee}=H^1_{\dR,c}(\calM_{\LT,\infty}):=\colim_n H^1_{\dR,c}(\calM_{\LT,n})$ and $\coker H^0(d)^{\vee}=H^2_{\dR,c}(\calM_{\LT,\infty})$ by the duality of de Rham cohomology, cf. \cite[Theorem 4.11]{grosse2000rigid}. Then we see $\coker H^0(t_{00})^{\vee}=\coker H^0(d)^{\vee}=H^2_{\dR,c}(\calM_{\LT,\infty})$, and there is a short exact sequence 
\begin{align*}
    0\to \calC^\sm_c(\GL_2(\bbQ_p),C)^{\oplus 2}\to \ker H^0(t_{00})^{\vee}\to H^1_{\dR,c}(\calM_{\LT,\infty})\to 0.
\end{align*}
Finally, the natural action of $t_{00}$ on $H^1_c(\calM_{\LT,n},\Omega_{\calM_{\LT,n}}^{1})$ is dual to its action on $H^0(\calM_{\LT,n},\calO_{\calM_{\LT,n}})$ via the Serre duality which is $D_p^{\times}$-equivariant.
\end{proof}
\end{proposition}

\begin{theorem}\label{thm:nonvanishingoftildetau}
Let $\lambda\in \sigma_{k,c}^{K^p}$. Then ${\tau_c }$ is of infinite-dimensional over $C$. In particular, ${\tau_c }\neq 0$. 
\end{theorem}
\begin{proof}
First, we assume $k=0$. Suppose $\tau_c = 0$, which amounts to saying that $H^1_c(\calM_{\LT,\infty},\Omega_{\calM_{\LT,\infty}}^{1,\sm})\ox_{\calH(G)}\pi_p=0$. Using (\ref{equationmultiplicity}) and Proposition \ref{prop:Torvanishing}, we find that 
\[H^1_c(\calM_{\LT,\infty},\Omega_{\calM_{\LT,\infty}}^{1,\sm})\ox_{\calH(G)}^L \pi_p=0.\] 
Besides $R\Gamma_{\dR,c}(\calM_{\LT,\infty})\ox^L_{\calH(G)}\pi_p$ is a dirct summand of 
\begin{align}\label{eq:fff}
    R\Gamma_{\dR,c}(\calM_{\LT,\infty})\ox^L_{\calH(G)}\calA_{\bar G,k}^{K^p}[\lambda]=R\Gamma_{\dR}(\calS_{K^p})[\lambda]
\end{align}
which concentrates in degree $1$ and is finite-dimensional (see also Remark \ref{remarkderhamtensor}). The corresponding spectral sequence \cite[\href{https://stacks.math.columbia.edu/tag/0662}{Tag 0662}]{stacks-project}
\[E_2^{i,j}=\mathrm{Tor}_{-i}^{\calH(G)}(H^j_{\rm dR,c}(\calM_{\LT,\infty}),\calA_{\bar G,k}^{K^p}[\lambda])\Rightarrow H^{i+j}_{\dR}(\calS_{K^p})[\lambda]\]
degenerates (cf. \cite[IV.4]{bernstein1992notes}) and we see that $H_{\dR,c}^i(\calM_{\LT,\infty})\otimes^L_{\calH(G)}\pi_p$ has finite-dimensional cohomology groups for $i=1,2$. 

By Proposition \ref{prop:t00}, we see there is a $\GL_2(\bbQ_p)$-equivariant exact sequence 
\begin{align}\label{eq:exactt00111}
    0\to \ker t_{00}\to H^1_c(\calM_{\LT,\infty},\Omega_{\calM_{\LT,\infty}}^{1,\sm})\ov{t_{00}}\to H^1_c(\calM_{\LT,\infty},\Omega_{\calM_{\LT,\infty}}^{1,\sm})\to \coker t_{00}\to 0
\end{align}
with $\coker t_{00}\isom H^2_{\dR,c}(\calM_{\LT,\infty})$ and another $\GL_2(\bbQ_p)$-equivariant exact sequence
\begin{align*}
    0\to \calC^\sm_c(\GL_2(\bbQ_p),C)^{\oplus 2}\to \ker t_{00}\to H^1_{\dR,c}(\calM_{\LT,\infty})\to 0.
\end{align*}
As $\calC^\sm_c(\GL_2(\bbQ_p),C)^{\oplus 2}\ox_{\calH(G)}^L \pi_p\isom \pi_p^{\oplus 2}[0]$ is infinite-dimensional (see Theorem \ref{thm:decompositionautomorphicforms}), $\ker t_{00}\otimes^L_{\calH(G)}\pi_p$ has infinite-dimensional cohomology. Applying $-\otimes^L_{\calH(G)}\pi_p$ to (\ref{eq:exactt00111}), we see this contradicts to the finiteness results of $ H_{\dR,c}^2(\calM_{\LT,\infty})\otimes^L_{\calH(G)}\pi_p$. From the proof above, we also see that $\tau_c $ is infinite-dimensional.

For general $k$, we can use the translation functors as in \cite[\S 2.2]{su2025locallyanalytictextext1conjecturetextgl2l} and \cite[Theorem 3.2.1]{JLS22}, see also Remark \ref{remarktranslation}, to reduce to the case $k=0$. 
\end{proof}

\subsubsection{Admissibility of $\tau_c $}\label{sec:admissibilityoftauc}
Similar to the treatment as in \cite[\S 6.2]{QS24} and \cite[Remark 7.3.5]{PanII}, as we have realized the representation $\tau_c $ (which can be purely defined using the de Rham complex of Lubin--Tate curves) inside the completed cohomology of quaternionic Shimura curves, we can prove the admissibility of $\tau_c $ for some natural rational structure. 
\begin{theorem}\label{thm:admissibilityoftaucE}
Let $\tau_c$ be defined as in Definition \ref{def:tauc}. There exists a locally analytic representation $\tau_{c,E}$ of $D_p^\times$ over a finite extension $E$ of $\bbQ_p$, such that $\tau_{c}\isom \tau_{c,E}\hat\ox_E C$, and $\tau_{c,E}$ is admissible.
\end{theorem}
\begin{proof}
    Using translation functors, we may assume $k=0$ for simplicity. Recall that (by Lemma \ref{lem:RpOLTsm})
    \begin{align*}
        \tilde{\tau}&=\dlim_nH^1_c(\calM_{\LT,n},\calO_{\calM_{\LT,n}})\ox_{\calH(G)}\pi_p.\\
        \tau_c&=\dlim_nH^1_c(\calM_{\LT,n},\Omega^1_{\calM_{\LT,n}})\ox_{\calH(G)}\pi_p.
    \end{align*}
    The action of $\GL_2(\bbQ_p)\times D_p^{\times}$ on $\calM_{\LT,\infty}$ factors through $H=(\GL_2\times D_p^{\times})/\{(a,a^{-1}),a\in \bbQ_p^{\times}\}$ in our normalization. Consider the subgroup $H^{\circ}=\{(g,h)\mid \det(g)\mathrm{Nm}(h)\in \bbZ_p^{\times}\}\subset H$ where $\mathrm{Nm}$ denotes the reduced norm of $D_p$. Then $H^{\circ}$ stabilize $\calM_{\LT,\infty}^{(0)}$, and there exist $H$-equivariant isomorphisms
    \begin{align}
        H^1_c(\calM_{\LT,\infty},\calO^{\sm}_{\calM_{\LT,\infty}})&=\cInd_{H^{\circ}}^{H}H^1_c(\calM_{\LT,\infty}^{(0)},\calO^{\sm}_{\calM_{\LT,\infty}})\label{euqationcompactinduction}\\&=\chi_{\rm sm} \otimes_C \cInd_{H^{\circ}}^{H}H^1_c(\calM_{\LT,\infty}^{(0)},\calO^{\sm}_{\calM_{\LT,\infty}})\nonumber
    \end{align}
    for any smooth character $\chi_{\sm}:H/H^{\circ}=\bbQ_p^{\times}/\bbZ_p^{\times}\rightarrow C^{\times}$. Hence we may assume that $\begin{pmatrix} p&0\\0&p \end{pmatrix}\in G=\GL_2(\bbQ_p)$ acts trivially on $\pi_p$, after possibly a smooth twist. 
    
    The induced Weil descent datum on $\calM_{\LT,n}/\begin{pmatrix} p&0\\0&p \end{pmatrix}^{\bbZ}$ is effective \cite[Theorem 2.16, Theorem 3.49]{RZ96}. Besides, there exists a sufficiently large finite extension $E$ of $\bbQ_p$, such that $\pi_p=\pi_{p,E}\hat\ox_E C$ for some smooth representation $\pi_{p,E}$ of $G$ over $E$ (since $\pi_p$ is a local component of a regular cuspidal algebraic automorphic representation). Therefore, using the above $E$-structures of the Lubin--Tate spaces and $\pi_p$, and the definition of $\widetilde{\tau}$, one can directly find a locally analytic representation $\tilde{\tau}_E$ of $D_p^\times$ over $E$, such that $\tilde{\tau}\isom \tilde{\tau}_E\hat\ox_E C$.
    
    Recall that by Proposition \ref{propositionexactseqDptimesrep} and Proposition \ref{thm:PicheckrhoLankerI1}, we have an exact sequence 
    \begin{align*}
        0\to \tau_p\otimes_C (\pi_{\lambda}^{p,\infty})^{K^p} \to \tilde{\tau}\otimes_C (\pi_{\lambda}^{p,\infty})^{K^p}\to \tilde{H}^1(K^p,C)[\lambda]^{\lan,\Theta_{\Sen}=0}\to 0.
    \end{align*}
    The algebraic automorphic representation $\pi_{\lambda}$ can be defined over the sufficiently large $E$ and we denote the $E$-structure as $\pi_{\lambda,E}$. Let $\rho_{\lambda}:\Gal(\overline{\bbQ}/\bbQ)\rightarrow\GL_2(E)$ be the corresponding Galois representation. The above exact sequence shows that the $D_p^{\times}$-representation $\tilde{H}^1(K^p,C)[\lambda]^{\lan,\Theta_{\Sen}=0}$, up to the multiplicity space $(\pi_{\lambda}^{p,\infty})^{K^p}$, is a quotient of $\tilde{\tau}$ by one copy of $\tau_p$. The relative position of the $\tau_p$ inside $\tau_p^{\oplus 2}\subset \tilde{\tau}$ is $E$-rational as it comes from the position of the Hodge filtration of $\rho_{\lambda}$ (see Remark \ref{remarkHodgefiltration} and the proof of Corollary \ref{cor:kerI1DRDRprime}). We define $\tau_E$ to be the cokernel of the composition $\tau_{p,E}\rightarrow \tau_{p,E}^{\oplus 2}\subset \widetilde{\tau}_E$ induced by the Hodge filtration. This gives an $E$-structure of $\tilde{H}^1(K^p,C)[\lambda]^{\lan,\Theta_{\Sen}=0}$:
    \begin{align}\label{eq:tauE1}
        \tilde{H}^1(K^p,C)[\lambda]^{\lan,\Theta_{\Sen}=0}\isom C\widehat{\ox}_E \tau_E\ox_E (\pi_{\lambda,E}^{p,\infty})^{K^p}.
    \end{align}
    On the other hand, by (\ref{eq:rhoisotypic}) we know that there is an isomorphism 
    \begin{align*}
        \tilde{H}^1(K^p,E)[\lambda]^{\lan}\isom \rho_{\lambda}\ox \check{\Pi}(\rho_{\lambda})^{\lan}
    \end{align*}
    where recall that $\check{\Pi}(\rho_{\lambda})$ is defined as $\Hom_{\Gal(\bar \bbQ/\bbQ)}(\rho_{\lambda},\tilde{H}^1(K^p,E))$. Therefore, 
    \begin{align}\label{eq:tauE2}
        (\tilde{H}^1(K^p,E)[\lambda]\hat\ox_E C)^{\lan,\Theta_{\Sen}=0}\isom (C\ox_E\rho_{\lambda})^{\Theta_{\Sen}=0}\hat\ox_E\check\Pi(\rho_{\lambda})^{\lan}.
    \end{align}
    Here, we fixed an embedding of $E$ into $C$. By construction, the isomorphisms in \eqref{eq:tauE1} and \eqref{eq:tauE2} are compatible with the Galois action. Taking $\Gal(\bar L/E)$-invariants, we deduce that there is a $D_p^\times$-equivariant, $E$-linear and topological isomorphism 
    \begin{align*}
        (C\ox_E\rho_{\lambda})^{\Theta_{\Sen}=0,{\Gal(\bar L/E)}}\hat\ox_E\check{\Pi}(\rho_{\lambda})^{\lan}\isom C^{\Gal(\bar L/E)}\hat\ox_E \tau_E\ox_E (\pi_{\lambda,E}^{p,\infty})^{K^p}.
    \end{align*}
    From this, we deduce that $\tau_E$ is admissible, as $\check{\Pi}(\rho_{\lambda})$ is admissible by the admissibility of the completed cohomology (cf. \cite[Theorem 2.2.11]{Emerton2006interpolation}). Finally, similarly as above for $\widetilde{\tau}$, one can find a natural $E$-structure $\tau_{c,E}$ of $\tau_c$. As $\tau_{c,E}$ is a quotient of $\tau_E$, we deduce that $\tau_{c,E}$ is also admissible.
\end{proof}
\subsubsection{Gelfand-Kirillov dimension of $\tau_c $}\label{sec:GKD}
Let $\tau_{c,E}$ be the (natural) $E$-structure of $\tau_c$ in Theorem \ref{thm:admissibilityoftaucE}, where $E$ is a finite extension of $\bbQ_p$. For definition and basic properties of the Gelfand-Kirillov dimension of a locally analytic representation, our main references are \cite[\S 8]{ST03}, \cite[\S 3]{dospinescu2023gelfand}. The following result is known by the work of Hu--Wang \cite[Theorem 1.1]{HuWang} under some genericity condition (on the mod-$p$ representations of $\Gal(\overline{\bbQ}_p/\bbQ_p)$).
\begin{theorem}\label{thm:GKD}
The Gelfand-Kirillov dimension of $\tau_{c,E}$ is $1$.
\begin{proof}
There are two ingredients in the proof, which give an upper bound and a lower bound. We assume again $k=0$, and the general cases can be deduced using the translation operators. Let $\tau_E$ be the $E$-structure of $\tilde{H}^1(K^p,C)[\lambda]^{\lan,\Theta_{\Sen}=0}$ constructed in the proof of Theorem \ref{thm:admissibilityoftaucE}. 

First, as $\tau_E$ has trivial infinitesimal character, and is isomorphic to the locally analytic vector of an admissible unitary Banach representation (as it comes from the completed cohomology), by \cite[Corollary 6.8]{dospinescu2023gelfand}, we know that $\tau_E$ has Gelfand-Kirillov dimension at most $1$. As the difference between $\tau_{c,E}$ and $\tau_E$, up to a multiplicity, is a finite dimensional representation of $D_p^\times$ (whose base change to $C$ is $\tau_p$), we deduce that $\tau_{c,E}$ has Gelfand-Kirillov dimension at most $1$. 

Besides, we know that $\tau_c $ is of infinite-dimensional over $C$ by Theorem \ref{thm:nonvanishingoftildetau}, hence $\tau_E$ is of infinite dimensional over $E$. As $\tau_E$ is isomorphic to the locally analytic vectors of a Banach representation, from \cite[Lemma 3.9]{dospinescu2023gelfand}, we deduce that the Gelfand-Kirillov dimension of $\tau_E$ is at least $1$. Again, this shows that the Gelfand-Kirillov dimension of $\tau_{c,E}$ is at least $1$. 
\end{proof}
\end{theorem}
\begin{remark}
    Let $L$ be a finite extension of $\bbQ_p$. For an irreducible smooth representation of $\GL_2(L)$, one can define locally $L$-analytic representations $\widetilde{\tau},\tau_c$ of $D_L^{\times}$, where $D_L$ denotes the unique non-split quaternion algebra over $L$, in the same way as in Definition \ref{def:tauc} using the Lubin--Tate tower for $L$. One can prove similar results on the Gelfand-Kirillov dimensions for these representations. Note that \cite[Lemma 3.9]{dospinescu2023gelfand} is only stated for $L=\bbQ_p$, but one can use the following argument for general $L$. Assume that $\tau_{c,E}$ is an admissible locally analytic representation of $D_L^{\times}$ of the Gelfand-Kirillov dimension $0$ with the infinitesimal character the same as the smooth representations. Then $\tau_{c,E}$ must be a smooth admissible representation, using \cite[Corollary 9.1]{ardakov2013irreducible} and \cite[Theorem 1.4.2]{kohlhaase2007invariant}. But our $\tau_{c,E}$ is not smooth, otherwise its wall-crossing is zero (using similar methods as in \cite[Theorem 4.7.6]{QS24}, its wall-crossing still contains $\tau_{c,E}$, which is non-zero). See the proof of Lemma \ref{lem:no-smooth-quotient} below for more relavent details on the wall-crossing. 
\end{remark}

\subsection{Structure of the Hecke eigenspaces of the completed cohomology}\label{sec:completedcohoclassicality}
In this section, we summarize our study on the Hecke eigenspace $\check\Pi(\rho)^{\lan}$ (\ref{eq:defofTrho}) of the locally analytic completed cohomology of the Shimura curves attached to a $2$-dimensional de Rham Galois representation $\rho$. Combining computations before we explain what the locally analytic $D_p^\times$-representation $\check\Pi(\rho)^{\lan}$ looks like. Then we can deduce that $\check\Pi(\rho)^{\lan}$ essentially only depends on the local Galois representation, and even only depends on the Weil--Deligne representation associated to the local Galois representation in the crystabeline case. 

Recall that in \S \ref{sec:quaternionicshimuracurve} we have defined $S_{K^pK_p}$ as the quaternion Shimura curve over the reflex field $\bbQ$ of level $K=K^pK_p$. Let $E\subset C$ be a finite extension of $\bbQ_p$. After fixing the tame level $K^p$ and taking limit over $K_p\subset D_p^\times$, we defined in \S\ref{subsectioncompletedcohomology} the completed cohomology group $\tilde{H}^1(K^p,E)$ for the tower $(S_{K^pK_p})_{K_p}$ with coefficients in $E$.

Similar to \cite[Theorem 7.1.2]{PanII}, we have the following classicality result.
\begin{theorem}\label{thm:classicality}
Let $\rho$ be a $2$-dimensional $E$-linear continuous absolutely irreducible representation of $\Gal(\bar \bbQ/\bbQ)$. Suppose:
\begin{enumerate}[(i)]
    \item $\rho$ appears in $\tilde{H}^1(K^p,E)^{\lan}$, namely $\Hom_{\Gal(\bar \bbQ/\bbQ)}(\rho,\tilde{H}^1(K^p,E)^{\lan})\neq 0$.
    \item $\rho|_{\Gal(\overline{\bbQ_p}/\bbQ_p)}$ is de Rham of Hodge Tate weight $0,k+1$ with $k\in \bbZ_{\ge 0}$. 
\end{enumerate}
Then $\rho$ is classical, i.e., the Hecke eigenvalue associated to $\rho$ lies in $\sigma_{k,c}^{K^p}$ (Definition \ref{def:setHeckeeigenvalues}).
\begin{proof}
This is a combination of Theorem \ref{thm:PicheckrhoLankerI1} and Theorem \ref{thm:kerI1classical}. 
\end{proof}
\end{theorem}

Next, we show that any $\lambda\in \sigma_{k,c}^{K^p}$ must appear in $\tilde{H}^1(K^p,C)$. Fix $k\geq 0$ and $\lambda\in \sigma_{k,c}^{K^p}$. We use the notation in \S\ref{subsec:repofDptimes} and let $\pi_p$ be the component at $p$ of the automorphic representation $\pi_{\lambda}$ of $\bar G$. When $\pi_p$ is a discrete series representation, we already know that $\tilde{H}^1(K^p,C)^{\lalg}[\lambda]\supset \tau_p\otimes_{C}W^{(0,-k)} \neq 0$ by the global Jacquet--Langlands correspondence (the locally algebraic vectors of $\tilde{H}^1(K^p,C)$ recovers the classical cohomology of $S_{K^pK_p}$, cf. \cite[(4.3.4)]{Emerton2006interpolation}). So the crucial part is to show that $\tilde{H}^1(K^p,C)[\lambda]\neq 0$ even when $\pi_p$ is an irreducible principal series representation.  
\begin{theorem}\label{thm:nonvanishingofHeckeeigenspace}
Suppose $\lambda\in \sigma_{k,c}^{K^p}$. Then $\tilde{H}^1(K^p,C)^{\lan}[\lambda]\neq 0$.
\begin{proof}
Suppose $\tilde{H}^1(K^p,C)^{\lan}[\lambda]=0$. Then, since $\ker I^1$ is a subspace of $\tilde{H}^1(K^p,C)^{\lan}$, we deduce that $\ker I^1[\lambda]=0$. But by Proposition \ref{propositionexactseqDptimesrep} (deduced from Theorem \ref{thm:H1DR} and Corollary \ref{cor:kerI1DRDRprime}), this will imply $\tau_c =0$, which contradicts Theorem \ref{thm:nonvanishingoftildetau}.
\end{proof}
\end{theorem}

\begin{remark}
This result was known by $\mathrm{Pa\check{s}k\mathaccent "7016\relax{u}nas}$ \cite[Theorem 1.4]{pavskunas2022some}, under some genericity conditions (on the mod $p$ Galois representations).
\end{remark}

By Theorem \ref{thm:PicheckrhoLankerI1}, Theorem \ref{thm:nonvanishingofHeckeeigenspace}, together with Theorem \ref{thm:H1DR}, Corollary \ref{cor:kerI1DRDRprime} (Proposition \ref{propositionexactseqDptimesrep}) and (\ref{equationmultiplicity}), we deduce the following description of the locally analytic representation associated to $\lambda$.
\begin{corollary}\label{corollarydescription}
    Let $\rho$ be the Galois representation associated to $\lambda$ such that $\bbT^S$ acts on 
    \[\check\Pi(\rho)=\Hom_{\Gal(\bar \bbQ/\bbQ)}(\rho,\tilde{H}^1(K^p,E))\] 
    via $\lambda$. We have an isomorphism
    \begin{align}\label{eq:strofT(rho)L-an}
    \check\Pi(\rho)^{\lan}\widehat{\otimes}_EC\isom (\tilde{\tau}/i_{\rho}(\tau_p^{\lalg}))\otimes_C(\pi_{\lambda}^{p,\infty})^{K^p}.
    \end{align}
    where $\tilde{\tau}/i_{\rho}(\tau_p^{\lalg})$ is moreover an extension of $\tau_c $ by $\tau_p^{\lalg} $. Here, 
    \begin{itemize}
        \item $\tau_p^{\lalg}=\tau_p\otimes_CW^{(0,-k)}$ with $\tau_p $ the Jacquet--Langlands transfer of $\pi_p$ and $W^{(0,-k)}$ the algebraic representation over $C$ of $D_p^{\times}$ of highest weight $(0,-k)$.
        \item $\tilde{\tau}:=\dlim_n H^1_c(\calM_{\LT,n},\omega_{\calM_{\LT,n}}^{(0,-k)})\ox_{\calH(G)}\pi_p$ with $G=\GL_2(\bbQ_p)$ and $\calH(G)$ the smooth Hecke algebra of $G$.
        \item $\tau_c :=\dlim_n H^1_c(\calM_{\LT,n},\omega_{\calM_{\LT,n}}^{(-k-1,1)})\ox_{\calH(G)}\pi_p$.
        \item The map $i_{\rho}:\tau_p^{\lalg}\to \tilde{\tau}$ comes from the position of the Hodge filtration of $\rho|_{\Gal(\overline{\bbQ_p}/\bbQ_p) }$ using the $p$-adic de Rham comparison theorem (see Remark \ref{remarkHodgefiltration}).
    \end{itemize}
\end{corollary}
We can deduce the following locality result. 
\begin{theorem}\label{thm:PS}
Let $\lambda\in \sigma_{k,c}^{K^p}$. Then up to a multiplicity, the locally analytic $D_p^\times$-representation $\check\Pi(\rho)^{\lan}$ only depends on the local Galois representation $\rho|_{\Gal(\overline{\bbQ_p}/\bbQ_p) }$. Moreover, when $\pi_p$ is an irreducible smooth principal series representation, $\check\Pi(\rho)^{\lan}$ only depends on the Weil--Deligne representation associated to $\rho|_{\Gal(\overline{\bbQ_p}/\bbQ_p)}$ (and the Hodge--Tate weights).
\begin{proof}
By Theorem \ref{thm:nonvanishingofHeckeeigenspace}, we know that every $\lambda\in \sigma_{k,c}^{K^p}$ must appear in $\tilde{H}^1(K^p,C)$. Hence the condition on $\rho$ required in Theorem \ref{thm:PicheckrhoLankerI1} is satisfied. Then using (\ref{eq:strofT(rho)L-an}), we see that $\check\Pi(\rho)^{\lan}$ only depends on the local Galois representation $\rho|_{\Gal(\overline{\bbQ_p}/\bbQ_p) }$. Moreover, when $\pi_p$ is an irreducible principal series representation, we know that $\tau_p=0$. In particular the information of Hodge filtration of $\rho|_{\Gal(\overline{\bbQ_p}/\bbQ_p) }$ disappears in $\check\Pi(\rho)^{\lan}$. Hence $\check\Pi(\rho)^{\lan}$ only depends on the Weil--Deligne representation associated to $\rho|_{\Gal(\overline{\bbQ_p}/\bbQ_p) }$.
\end{proof}
\end{theorem}
\begin{remark}\label{remarkordinary}
In particular, when $\pi_p$ is an irreducible principal series representation, one can not capture the information of the Hodge filtration in $D_{\dR}(\rho|_{\Gal(\overline{\bbQ_p}/\bbQ_p) })$ from $\check{\Pi}(\rho)^{\lan}$. For example, assume that $\rho_p:\Gal(\overline{\bbQ_p}/\bbQ_p)\rightarrow \GL_2(E)$ is a non-split extension of a crystalline character $\chi$ of $\Gal(\overline{\bbQ_p}/\bbQ_p)$ by the trivial character $1$ where $\chi$ has Hodge-Tate weight $k+1\geq 1$ and $\chi$ is not the inverse of the cyclotomic character. Then both $\rho_p$ and its semisimplification $\rho_p^{\rm ss}=\chi\oplus 1$ are crystalline representations with the same associated Weil--Deligne representation but different positions of Hodge filtrations, see \cite[\S 4.5]{Colmez2008Trianguline}. In this case, the smooth $\GL_2(\bbQ_p)$-representation $\pi_p$ is an irreducible principal series representation, and $\tau_p=0$. If $\rho|_{\Gal(\overline{\bbQ_p}/\bbQ_p) }$ has the same semisimplification as $\rho_p$, then the $D_p^{\times}$-representation $\check{\Pi}(\rho)^{\lan}\widehat{\otimes}_{E}C$ cannot distinct $\rho_p$ and $\rho_p^{\rm ss}$. This is a new phenomenon in the $p$-adic local Langlands program in the $D_p^\times$-case, compared to the $\GL_2(\bbQ_p)$-case. In the mod $p$ setting, similar results were known by \cite[Theorem 1.3(ii)(a)]{HuWang}. On the other hand, when $\pi_p$ is a twist of the smooth Steinberg representation, we expect that $\check{\Pi}(\rho)^{\lan}$ still determines the Hodge filtration of $\rho|_{\Gal(\overline{\bbQ_p}/\bbQ_p)}$ as in \cite[Remark 8.31]{HuWang}
\end{remark}

\subsection{Applications to the Scholze functor}\label{sec:application}
In this section, we want to discuss some applications to the $p$-adic Jacquet--Langlands correspondence. Our main goal in this section  (Theorem \ref{theoremscholzefunctor}) is to compute the value of the Scholze functor \cite{Scholze2018LubinTate,dospinescu2024jacquetlanglandsfunctorpadiclocally} on the locally analytic representations of $\GL_2(\bbQ_p)$ attached to $2$-dimensional de Rham Galois representations, in terms of the locally analytic representations of $D_p^{\times}$ considered in the previous sections. Let $\check{G}=D_p^{\times}$ and $G=\GL_2(\bbQ_p)$.

We briefly recall the definition of the Scholze functor, following \cite{Scholze2018LubinTate,dospinescu2024jacquetlanglandsfunctorpadiclocally}. Let $\pi_{\LT,\GM}:\calM_{\LT,\infty}\to \check{\fl}$ be the Gross--Hopkins period map on $\calM_{\LT,\infty}$, which is a pro-\'etale $\GL_2(\bbQ_p)$-torsor. Given an admissible unitary Banach representation $\Pi$ of $\GL_2(\bbQ_p)$, we can descend the constant sheaf on $\calM_{\LT,\infty}$ with fiber $\Pi$ to $\check{\fl}$, so that we get a pro-\'etale sheaf $\calF_{\Pi}$ on $\check{\fl}$. The pro-\'etale cohomology group $R\Gamma_{\pro\et}(\check{\fl},\calF_{\Pi})$ is again a unitary Banach representation of $D_p^\times$, which is also equipped with a commuting action of $\Gal(\overline{\bbQ_p}/\bbQ_p)$. As is explained in \cite[Corollary 5.3.5]{dospinescu2024jacquetlanglandsfunctorpadiclocally}, if $\Pi^0$ is a $G$-stable lattice in $\Pi$, then 
\begin{align}\label{eq:proetcohomologyaslimitofetale}
    R\Gamma_{\pro\et}(\check{\fl},\calF_{\Pi})\isom R\ilim_n R\Gamma_{\et}(\check{\fl},\calF_{\Pi^0/p^n})[1/p].
\end{align}
Moreover, if we denote $\Pi^{\lan}\subset \Pi$ as the subspace of locally analytic vectors of $\Pi$, we can also descend $\Pi^{\lan}$ along $\pi_{\LT,\GM}$ so that to get a pro-\'etale sheaf $\calF_{\Pi^{\lan}}$ on $\check{\fl}$. By \cite[Theorem 5.3.6]{dospinescu2024jacquetlanglandsfunctorpadiclocally}, we know that 
\begin{align}\label{eq:scholzefunctorlocanaly}
    R\Gamma_{\pro\et}(\check{\fl},\calF_{\Pi})^{R\check{G}\-\lan}\isom R\Gamma_{\pro\et}(\check{\fl},\calF_{\Pi^{G\-\lan}}).
\end{align}
For $i\geq 0$, set 
\[S^i(\Pi^{\lan}):=H^i_{\pro\et}(\check{\fl},\calF_{\Pi^{\lan}}).\]

Let $Z_{K^p}$ be the Shimura set associated to $\bar G$ in \S\ref{sec:uniformization}. Define 
\begin{align*}
    \tilde{\calA}^{K^p}_{\bar G,E}:=\calC^{\cont}(Z_{K^p},E)
\end{align*}
to be the completed cohomology of the Shimura set. This is an admissible unitary Banach representation of $\bar G(\bbQ_p)=\GL_2(\bbQ_p)=G$, which is also equipped with an action of the Hecke algebra $\bbT^S$. Let $\tilde{\calA}^{K^p,G\-\lan}_{\bar G,E}$ be the subspace of locally $G$-analytic vectors.  Similar to \cite[Theorem 6.2]{Scholze2018LubinTate}, we can use the $p$-adic uniformization of the quaternionic Shimura curve $\calS_{K^p}$ to deduce the compatibility of Scholze functor with the completed cohomology groups $\tilde{\calA}^{K^p}_{\bar G,E}$ and $\tilde{H}^i(K^p,E)$ in \S\ref{subsectioncompletedcohomology}. 
\begin{theorem}\label{thm:ScholzeFunctor}
For $i\ge 0$, there is a $\Gal(\overline{\bbQ_p}/\bbQ_p)\times \bbT^S\times \check{G}$-equivariant isomorphism:
\begin{align*}
    H^i_{\pro\et}(\check{\fl},\calF_{\tilde{\calA}^{K^p,G\-\lan}_{\bar G,E}})\isom \tilde{H}^i(K^p,E)^{\check{G}\-\lan}.
\end{align*}
\end{theorem}
\begin{proof}
    Write $\pi_{\HT}=\pi_{K^p,\HT}:\calS_{K^p}\to \check{\fl}$ for the Hodge--Tate period map on the quaternionic Shimura curve $\calS_{K^p}$ of infinite level (recalled in \S\ref{subsec:quaternionicshimuracurve}). Define $\calF_{K^p,\bbQ_p}$ as the pro-\'etale sheaf on $\check{\fl}$ by descending the unitary Banach $G$-representation $\tilde\calA^{K^p}_{\bar G,\bbQ_p}$ (with $\bbQ_p$-coefficients). We show that there is a $\check{G}$-equivariant and Weil group equivariant isomorphism of sheaves on the pro-\'etale site of $\check{\fl}$
	$$R\pi_{\HT,\proet*}\bbQ_p\cong \calF_{K^p,\bbQ_p}.$$ 
    Then one can use \cite[Corollary 6.1.7]{camargo2024locallyanalyticcompletedcohomology} to deduce the claim.     

    We first show that $R^n\pi_{\HT,\proet*}\bbQ_p=0$ when $n>0$. Since $\calS_{K^p}\cong \sqcup_i \Gamma_{x_i}\backslash\calM_{\Dr,\infty}$ as in (\ref{uniformization}) and the Hodge--Tate period map on $\calS_{K^p}$ is compatible with the one on $\calM_{\Dr,\infty}$, it suffices to show the vanishing of higher direct image of $\pi_{\HT,\Gamma}:\calS_{\Gamma}:=\Gamma_{x_i}\backslash\calM_{\Dr,\infty}\to \bbP^1_{C}$ for each $\Gamma=\Gamma_{x_i}$. For  $n>0$, $R^n\pi_{\HT,\Gamma,\proet*}\bbQ_p$ is the sheafification of 
	$$X\to H^n_{\proet}(X\times_{\bbP^1_{C}} \calS_{K^p},\bbQ_p)$$
	where $X$ are affinoid perfectoid spaces pro\'etale over $\check{\fl}=\bbP^1_{C}$. Since $\pi_{\Dr,\HT}$ exhibits $\calM_{\Dr,\infty}$ as a pro\'etale $\GL_2(\bbQ_p)$-torsor on $\bbP^1_{C}$, for each pro\'etale covering of $X$, we can find a pro\'etale refinement by strictly totally disconnected spaces $S_i$ such that there is a $\GL_2(\bbQ_p)$-equivariant isomorphism $S_i\times_{\bbP^1_{\bbC_{p}}} \calM_{\Dr,\infty}\cong S_i\times \GL_2(\bbQ_p)$. Let $T=\Gamma\backslash \GL_2(\bbQ_p)$ be the profinite set. Because $\Gamma$ acts trivially on $\bbP^1_{\bbC_{p}}$, we get an isomorphism $S_i\times_{\bbP^1_{\bbC_{p}}} \calS_{\Gamma}\cong S_i\times T$. Now the claim follows from $$H^n_{\proet}(S_i\times T,\bbQ_p)=0$$
	as $S_i\times T$ is still a strictly totally disconnected space (for example by \cite[Lemma 7.19]{scholze2022etale}). Note that here we need $T$ to be profinite (rather than just locally profinite) to have $H^n_{\proet}(S_i\times T,\bbQ_p)\cong H^n_{\et}(S_i\times T,\bbZ_p)[1/p]=0$.
	
	To identify $\pi_{\HT,\proet*}\bbQ_p$ with $\calF_{K^p,\bbQ_p}$. Using the same argument as in the end of the proof of \cite[Proposition 6.5]{Scholze2018LubinTate}, replacing $\bbQ_p/\bbZ_p$ by $\bbQ_p$, we have a map $\pi_{\HT,\proet*}\bbQ_p\to  \calF_{K^p,\bbQ_p}$. To check this is an isomorphism, it suffices to check it on strictly totally disconnected $S_i$ as in the last paragraph. In this case, it is easy to see both sides can be identified with $\Cont(|S_i\times T|,\bbQ_p)$.
\end{proof}

The $G$-smooth vectors $\calA^{K^p,G\-\sm}_{\bar G,E}$ of $\tilde{\calA}^{K^p}_{\bar G,E}$ is the space of $E$-valued classical automorphic forms on $\bar G$ of tame level $K^p$ and trivial weight at $p$ (i.e. take $k=0$ in Definition \ref{def:algebraicautomorphicform} and with $E$-coefficients). Take a cuspidal Hecke eigensystem $\lambda:\bbT^S\rightarrow E$ appearing in $\calA^{K^p,G\-\sm}_{\bar G,E}$. Let $\rho_{\lambda}:\Gal(\overline{\bbQ}/\bbQ)\rightarrow\GL_2(E)$ be the Galois representation associated with $\lambda$. Then as (\ref{equationmultiplicity}) we have
\begin{align*}
    \calA^{K^p,G\-\sm}_{\bar G,E}[\lambda]\isom \pi_{p,E}\otimes (\pi_{\lambda,E}^{p,\infty})^{K^p} 
\end{align*}
where $\pi_{\lambda,E}$ is the corresponding automorphic representation of $\bar G(\bbA)$ and $\pi_{p,E}$ is its $p$-component, a smooth irreducible admissible representation of $G$ over $E$. We furthur assume that mod $p$ reduction $\bar\rho_{\lambda}|_{\Gal(\overline{\bbQ_p}/\bbQ_p)}$ is absolutely irreducible for simplicity (this is the same assumption as \cite[\S 5.1, Hypoth\'ese]{DLB17}, which can certainly be weakened, cf. \cite[Corollary 5.7]{pavskunas2022some}). Then using the $p$-adic local-global compatibility results \cite{emerton2011local,DLB17}, we know that
\begin{align}\label{equationeigenspacecompletedbarG}
    \tilde{\calA}^{K^p,G\-\lan}_{\bar G,E}[\lambda]\isom \pi_E\otimes (\pi_{\lambda,E}^{p,\infty})^{K^p} ,
\end{align}
where $\pi_E$ is a locally analytic representation of $G$ over $E$, corresponding to $\rho_{\lambda}|_{\Gal(\overline{\bbQ_p}/\bbQ_p)}$ via the $p$-adic local Langlands correspondence for $\GL_2(\bbQ_p)$. See for example \cite[\S 7.3]{PanII} for a concrete description of $\pi_E$. We note that $\pi_{p,E}$ is the closed subspace of smooth vectors of $\pi_E$.

Starting from $\pi_{p,E}$, we have defined various locally analytic $\check{G}=D_p^\times$-representations over $E$. Let $\tau_{p,E}$ be the Jacquet--Langlands transfer of $\pi_{p,E}$ to $D_p^\times$ with $\tau_{p,E}=0$ if $\pi_{p,E}$ is an irreducible principal series representation as in \S\ref{subsec:repofDptimes}. Set $\pi_p:=\pi_{p,E}\hat\ox_E C$. Let $\tilde{\tau}, \tau_c$ be the representations of $D_p^\times$ over $C$ defined as in Definition \ref{def:tauc} using $\pi_p$, with the $E$-structures $\widetilde{\tau}_E,\tau_{c,E}$ constructed in \S \ref{sec:admissibilityoftauc}. Also, recall that we have constructed a locally analytic representation $\tau_E$ depending only on $\rho_{\lambda}|_{\Gal(\overline{\bbQ_p}/\bbQ_p)}$ in the proof of Theorem \ref{thm:admissibilityoftaucE}, so that \begin{align}\label{equationeigenspacecompletedG}
        \tilde{H}^1(K^p,E)[\lambda]=\rho_\lambda\otimes_E\Hom_{\Gal(\bar \bbQ/\bbQ)}(\rho_\lambda,\tilde{H}^1(K^p,E))=\rho_{\lambda}\otimes_E\tau_E\otimes (\pi_{\lambda,E}^{p,\infty})^{K^p}.
\end{align}
\begin{theorem}\label{theoremscholzefunctor}
Let $\lambda:\bbT^S\rightarrow E$ be a cuspidal Hecke eigensystem appearing in $\widetilde{\calA}^{K^p,G\-\sm}_{\bar G,E}$, with the associated Galois representation $\rho_{\lambda}$. Assume that the reduction and the restriction $\bar\rho_{\lambda}|_{\Gal(\overline{\bbQ_p}/\bbQ_p)}$ of $\rho_\lambda$ is absolutely irreducible. Let $\pi_E$ (resp. $\tau_E$) be the locally analytic $G$ (resp. $\check{G}$)-representation associated to $\lambda$. Then we have a $\Gal(\overline{\bbQ_p}/\bbQ_p)$-equivariant isomorphism of locally analytic $D_p^\times$-representations
\begin{align*}
    S^1(\pi_{E})\isom \rho_{\lambda}|_{\Gal(\overline{\bbQ_p}/\bbQ_p)}\ox_E\tau_{E}.
\end{align*}
\end{theorem} 
\begin{proof}
    This is a combination of Theorem \ref{thm:ScholzeFunctor} and Proposition \ref{propositionscholzefunctoreigenspace} below, with the descriptions of the Hecke eigenspaces (\ref{equationeigenspacecompletedbarG}) and (\ref{equationeigenspacecompletedG}).
\end{proof}

Using a similar treatment as in \cite[Proposition 7.7]{Scholze2018LubinTate} and Lemma \ref{lem:no-smooth-quotient} below, we have the following result. 
\begin{proposition}\label{propositionscholzefunctoreigenspace}
Let $\lambda$ be as in Theorem \ref{theoremscholzefunctor}. Then the natural map
\begin{align}
   S^1(\tilde\calA^{K^p,G\-\lan}_{\bar G,E}[\lambda])\aisom  S^1(\tilde\calA^{K^p,G\-\lan}_{\bar G,E})[\lambda].
\end{align}
is an isomorphism.
\begin{proof}
By our assumption, the mod $p$ reduction of $\rho_\lambda$ is absolutely irreducible. This gives a maximal ideal $\frm$ of the abstract Hecke algebra $\bbT^S$ over $\bbZ$. Let $\bbT(K_pK^p)_{\frm}$ be the $\frm$-adic completion of the image of $\bbT^S$ in $\End_{\bbZ_p}( H^1(S_{K^pK_p}(\bbC),\bbZ_p))$, and put $\bbT(K^p)_{\frm}:=\ilim_{K_p}\bbT(K_pK^p)_{\frm}$. It acts faithfully on $\tilde{H}^1(K^p,E)_{\frm}$, where $(-)_{\frm}$ denotes the localization at $\frm$ of $\bbT(K^p)$-modules. Using \cite[Corollary 7.3]{Scholze2018LubinTate} and taking $p$-adic completion, we know that the $\bbT^S$-action on $\tilde{\calA}^{K^p}_{\bar G,E,\frm}$ extends to a continuous action of $\bbT(K^p)_{\frm}$.

For simplicity, we write ${\calA}:=\tilde\calA^{K^p,G\-\lan}_{\bar G,E,\frm}$. Using a finite set of generators of the ideal of $\bbT(K^p)_{\frm}$ corresponding to $\lambda$, there is an exact sequence 
\begin{align}\label{eq:exactseqofA}
    0\to {\calA}[\lambda]\to {\calA}\to \oplus_{i\in I}{\calA}
\end{align}
with $I$ finite. The above exact sequence induces the following two exact sequences 
\begin{align*}
    0\to {\calA}[\lambda]\to {\calA}\to M\to 0,\\
    0\to M\to \oplus_{i\in I}{\calA}\to Q\to 0.
\end{align*}
where $M$ is the cokernel of the first map and $Q$ is the cokernel of the second map. As the formation $\pi\mapsto \calF_{\pi}$ is exact, we get the following exact sequences 
\begin{align*}
    S^0(M)\to S^1({\calA}[\lambda])\to S^1({\calA})\to S^1(M),\\
    S^0(Q)\to S^1(M)\to S^1(\oplus_{i\in I}{\calA}).
\end{align*}
where $S^i(\pi)=H^i_{\pro\et}(\check{\fl},\calF_{\pi})$. First, $S^0(M)=0$. Indeed, as $M\subset \oplus_{i\in I}\calA$, the Hecke action on $S^0(M)\subset \oplus_{i\in I}S^0(\calA)$ is Eisenstein by Theorem \ref{thm:ScholzeFunctor}. But $S^0(M)$ is also a subspace of $S^1({\calA}[\lambda])$, which is non-Eisenstein. Hence $S^0(M)=0$. Then we know there is an exact sequence 
\begin{align*}
    0\to S^1(\calA[\lambda])\to S^1(\calA)[\lambda]\to \ker \beta
\end{align*}
where $\beta:S^1(M)\to S^1(\oplus_{i\in I}\calA)$. In particular, we know that there is a surjection $S^0(Q)\surj \ker \beta$. Now we claim that the action of $\check{G}$ on $S^0(Q)$ factors through the reduced norm map. Indeed, using the definition of $Q$ and that Hecke operators act unitarily, we see that $Q$ has an admissible unitary Banach completion $\hat Q$, whose subspace of locally analytic vectors is $Q$. Let $\hat Q^0$ be a lattice in $\hat Q$. Then by (\ref{eq:proetcohomologyaslimitofetale}) we know
\begin{align*}
    S^0(\hat Q)=\ilim_n S^0(\hat Q^0/p^n)[\frac{1}{p}].
\end{align*}
Using \cite[Proposition 4.7]{Scholze2018LubinTate} applied to $\hat Q^0/p^n$, we deduce that the $\check{G}$-action on $S^0(\hat Q)$ factors through the reduced norm map, and so does its subspace $S^0(Q)$. Hence we know that $\check{G}$ also acts on $\ker \beta$ via the reduced norm map by the surjection $S^0(Q)\surj \ker \beta$. On the other hand, we know that $S^1(\calA)[\lambda]$ has trivial infinitesimal character. Indeed, from Theorem \ref{thm:ScholzeFunctor} we know that $S^1(\calA)=\tilde{H}^1(K^p,E)_{\frm}^{\check{G}\-\lan}$. Hence using \cite[Theorem 1.4]{dospinescu2025infinitesimal} we know that $\tilde{H}^1(K^p,E)_{\frm}^{\check{G}\-\lan}[\lambda]$ has trivial infinitesimal character. So that the image of $S^1(\calA)[\lambda]\to \ker \beta$ also has trivial infinitesimal character. This forces that this image must be a smooth $\check{G}$-representation. By Lemma \ref{lem:no-smooth-quotient} below, we know that $S^1(\calA)[\lambda]$ can not have such quotients. Hence the image of $S^1(\calA)[\lambda]$ in $\ker \beta$ is zero. In other words, we get the desired isomorphism $S^1(\calA[\lambda])\aisom  S^1(\calA)[\lambda].$
\end{proof}
\end{proposition}
\begin{lemma}\label{lem:no-smooth-quotient}
Let $\lambda:\bbT^S\rightarrow E$ be a cuspidal Hecke eigensystem appearing in $\widetilde{\calA}^{K^p,G\-\sm}_{\bar G,E}$ as above, with the associated locally analytic representation $\tau_E$. Then $\tau_E$ has no non-zero smooth quotients.
\begin{proof}
    The proof is similar to \cite[Proposition 2.5.2, Theorem 2.5.7]{su2025locallyanalytictextext1conjecturetextgl2l} or \cite[Theorem 4.7.6]{QS24}. Recall that $\DR,\DR'$ are certain complexes on $\check{\fl}$ defined in \S\ref{sec:cohomologyI}, such that $\bbH^1(\check{\fl},\DR)[\lambda]\isom \tilde{\tau}^{\oplus n}$ and $\bbH^1(\check{\fl},\DR')[\lambda]\isom \tau_c^{\oplus n}$, where $n$ is the multiplicity coming from the tame part of the automorphic representation. By Proposition \ref{prop:dKvsurj}, we know $\bbH^2(\check{\fl},\DR')$ is the second cohomology of an abelian sheaf on $\check{\fl}$, so that it is zero since $\check{\fl}$ is of dimension $1$. This shows that the natural map induced by wall-crossing \cite{JLS22,Din24} $\Theta_s\bbH^1(\DR)\to \bbH^1(\DR)$ is surjective (since the cokernel is a subspace of $\bbH^2(\DR')$ by the exact triangle $\DR'\rightarrow \Theta_s\mathrm{DR}\rightarrow\DR\rightarrow $ as in the proof for \cite[Theorem 4.7.6]{QS24}). Since the action of $\bbT^S$ on these cohomology groups are semisimple by Theorem \ref{thm:decompositionautomorphicforms} and Theorem \ref{thm:H1DR}, we get that the natural map $\Theta_s(\tilde{\tau})\to \tilde{\tau}$ is surjective. By Corollary \ref{cor:kerI1DRDRprime}, we know $\tau_E\hat\ox_E C$ is a quotient of $\tilde{\tau}$ by a smooth representation. The wall-crossing of a smooth representation is zero. From this one can deduce that the natural map induced by adjunction property of the wall-crossing functor  gives a surjection map $\Theta_s(\tau_E)\surj \tau_E$. Let $\tau'$ be a smooth quotient of $\tau_E$.  Then by functoriality, we get a commutative diagram 
\[\begin{tikzcd}[cramped]
	{\Theta_s(\tau_E)} & {\Theta_s(\tau')=0} \\
	{\tau_E} & {\tau'}
	\arrow[from=1-1, to=1-2]
	\arrow[two heads, from=1-1, to=2-1]
	\arrow[from=1-2, to=2-2]
	\arrow[two heads, from=2-1, to=2-2]
\end{tikzcd}\]
which forces that $\tau'=0$, as desired.
\end{proof}
\end{lemma}

\appendix

\section{Proper pushforward of solid abelian sheaves on adic spaces}\label{section:properpushforward}
In this section we develop the formalism for proper pushforward of solid abelian sheaves, following \cite{abe2020proper}.
\subsection{Setup and Definitions}
For an adic space, we assume it is a separated analytic adic space over $\bbQ_p$. For a partially proper morphism $f:X\to Y$ between adic spaces, we mean it is partially proper after passing to the associated diamonds \cite[Definition. 18.4]{scholze2022etale}. The morphism $f$ is proper if it is quasi-compact, separated and universally closed. This definition coincides with \cite[Definition. 18.1]{scholze2022etale} since passing to diamonds does not change the underlying topological spaces \cite[Lemma. 15.6]{scholze2022etale}.

\begin{definition}
	A germ is a pairing $(S,X)$ where $X$ is an adic space and $S\subseteq X$ is a closed subset. A morphism of germs $(S,X)\to (T,Y)$ is a morphism $f:X\to Y$ of adic spaces which induces a map $S\to T$.
\end{definition}

\begin{remark}
	The category of analytic adic spaces admits fiber product (the morphisms are adic, c.f \cite[Proposition 5.1.5]{scholze2020berkeley}). In particular, if $(S,X)\longrightarrow (W,Z)\longleftarrow (T,Y)$ are morphisms of germs then $K:=X\times_Z Y$ exists as an adic space and $V:=p^{-1}(S)\cap q^{-1}(T)$ is closed, where $p$ and $q$ are projections from $K$ to $X$ and $Y$ such that $(V,K)$ is the fiber product.
\end{remark}

Recall the following definition in Huber's theory of pesudo-adic spaces.

\begin{definition}
Let $f:(S,X)\to (T,Y)$ be a morphism of germs, we call $f$ is
	\begin{enumerate}
		\item quasi-compact (quasi-separated or coherent or taut) if the map of topological spaces $S\to T$ has such a property.
		\item separated if $\Delta(S)\subseteq V$ is closed where $(V,Z)=(S,X)\times_{(T,Y)}(S,X)$.
		\item Universally closed if for any morphism of germs $(V,Z)\to (T,Y)$ with fiber product $(W,P)$, the map $W\to V$ is a closed map.
		\item proper if it is quasi-compact, separated and universally closed.
	\end{enumerate}
\end{definition}

\begin{remark}\label{remarkunivclosed}
	Being quasi-compact or being universally closed is stable under composition and base change: By a standard argument, if $h:(S,X)\stackrel{f}\to (T,Y)\stackrel{g}\to (V,Z)$ is the composition of morphisms of germs such that $h$ is quasi-compact (or universally closed) and $g$ is separated, then $f$ is quasi-compact (or universally closed). For base change of quasi-compactness, it's because morphisms between analytic adic spaces are adic morphisms.
\end{remark}

\subsection{Sheaves in condensed abelian groups}
We use the condensed mathematics developed in \cite{scholze2019condensed}. Fix $\kappa$ an uncountable strong limit cardinality. Let $\ExtDisc_{\kappa}$ be the category of $\kappa$-small extremally disconnected sets endowed with coverings given by finite families of jointly surjective maps. A $\kappa$-small condensed set is a sheaf of sets on the site $\ExtDisc_{\kappa}$. In the following context, we will fix this $\kappa$ and omit $\kappa$ in the notation. Let $\cond(\bZ)$ be the category of condensed abelian groups and $\solid(\bZ_\square)$ be the full subcategory of solid abelian groups. They are Grothendieck abelian categories.

Let $X$ be an adic space. We denote by $\mathit{PSh}(X,\cond(\bZ))$ the category of presheaves in condensed abelian groups on $X$ and $\mathit{Sh}(X,\cond(\bZ))$ be the full subcategory of sheaves. The category of condensed abelian groups satisfies the condition  \cite[(17.4.1)]{kashiwara2006categories}, so the inclusion $\mathit{Sh}(X,\cond(\bZ))\subseteq \mathit{PSh}(X,\cond(\bZ))$ admits an exact left adjoint, called sheafification. Similarly, we denote by $\mathit{PSh}(X,\bZ_\square)$ (resp. $\mathit{Sh}(X,\bZ_\square)$) the category of presheaves (sheaves) of solid abelian groups on $X$. They also admit an exact sheafification functor.

\begin{remark}\label{smallremarks}
\begin{enumerate}
	\item  Since each covering of an extremally disconnected set splits, a map $f:M\to N$ between two condensed sets is an epimorphism (resp. a monomorphism)  if and only the map $f(S): M(S)\to N(S)$ is surjective (resp. injective) for each $S\in \ExtDisc$.
	\item Assume $\sF\in \mathit{Sh}(X,\cond(\bZ))$. For each $S\in \ExtDisc$, the functor $\sF(-)(S)$ is a sheaf of abelian groups on $X$.	
\end{enumerate}
\end{remark}

\begin{lemma}\label{lemmasurjectivecriterion}
	The functor (See Remark \ref{smallremarks}(2))
	$$\Phi: \mathit{(P)Sh}(X,\cond(\bZ))\to \Func(\ExtDisc^{op},\mathit{(P)Sh}(X,\Ab));\ \sF\mapsto (S\to \sF(-)(S))$$
	induces an equivalence between $\mathit{(P)Sh}(X,\cond(\bZ))$ and $(P)\sC\subseteq \Func(\ExtDisc^{op},\mathit{(P)Sh}(X,\Ab))$ containing functors $F: \ExtDisc^{op}\to \mathit{(P)Sh}(X,\Ab)$ such that $F(\emptyset)$ is the zero sheaf and for $S_1,S_2\in \ExtDisc$, $F(S_1\coprod S_2)=F(S_1)\times F(S_2)$.
\end{lemma}

\begin{proof}
	We have equivalences of categories 
	$$\mathit{PSh}(X,\Func(\ExtDisc^{op},\Ab))\simeq  \Func(\open^{op}(X)\times \ExtDisc^{op},\Ab)\simeq \Func(\ExtDisc^{op},\mathit{PSh}(X,\Ab))$$
	thus $\Phi$ is fully faithful. We prove the claim for $\mathit{Sh}(X,\cond(\bZ))$ and the claim for $\mathit{PSh}(X,\cond(\bZ))$ is inherited in the proof.  A functor $\sF\in \mathit{PSh}(X,\Func(\ExtDisc^{op},\Ab))$ lies in $\mathit{Sh}(X,\cond(\bZ))$ if and only if for any $U\subseteq X$ open, $S_1,S_2\in \ExtDisc$, $\sF(U)(\emptyset)=0$, $\sF(U)(S_1\coprod S_2)=\sF(U)(S_1)\times \sF(U)(S_2)$  and for any covering $\{U_i\to U\}_{i\in I}$, the sequence 
	\begin{equation}\label{equationexact}
	0\to \sF(U)\to \prod_{i\in I}\sF(U_i)\to \prod_{i,j\in I} \sF(U_i\cap U_j)
	\end{equation}
	is an exact sequence in $\cond(\bZ)$. The sequence \eqref{equationexact} is exact if and only if it is exact after evaluating on each $S\in \ExtDisc$. Then it is easy to check that the functor 
	$$\Psi: \sC\to \mathit{Sh}(X,\cond(\bZ));\ F\mapsto (U\to F(-)(U))$$
	is an inverse of $\Phi$.
\end{proof}

\begin{lemma}\label{sheafificationofstalks}
	Suppose $\sF\in \mathit{PSh}(X,\cond(\bZ))$, the sheafification $\sF^a$ of $\sF$ corresponds to the functor $F: S\to (U\to \sF(U)(S))^a$ under the equivalence $\Phi$.
\end{lemma}

\begin{proof}
	It is clear that $F$ lies in $\sC$. It suffices to see $F$ satisfies the universal property of sheafification: Evaluating on each $S\in \ExtDisc$, it is clear that $F(S)$ satisfies the universal property. For functoriality in $S$, it follows from that sheafification of abelian sheaves on $X$ is a functor and the unit of an adjunction pair is a natural transformation, which is functorial.
\end{proof}

\begin{remark}\label{remarktosolid}
Assume $\sF\in \mathit{PSh}(X,\bZ_\square)$. Since the category of solid abelian groups is stable under limit and colimit in the category of condensed abelian groups, by the construction of sheafification (cf. \cite[Definition 17.4.5]{kashiwara2006categories}), the sheafification of $\sF$ in $\mathit{PSh}(X,\cond(\bZ))$ and $\mathit{PSh}(X,\bZ_\square)$ coincides. This implies that the inclusion functor $\mathit{Sh}(X,\bZ_\square)\to \mathit{Sh}(X,\cond(\bZ))$ is exact.
\end{remark}
	
	\begin{corollary}\label{corosheafpreservestalk}
		Let $\sF\in \mathit{PSh}(X,\cond(\bZ))$ and $\sF^a$ be its sheafification. For each $x\in X$, $\sF_x\cong \sF_x^a$.
	\end{corollary}
	
	\begin{proof}
		For each $S\in \ExtDisc$, 
		\begin{equation}\label{stalk}
			\sF_x(S)=(\colim_{U\ni x} \sF(U))(S)=\colim_{U\ni x} (\sF(U)(S))=\Phi(\sF)(S)_x.
		\end{equation}
		But by Lemma \ref{sheafificationofstalks}, we have 
		$$\Phi(\sF^a)(S)_x=(\Phi(\sF)(S))^a_x=\Phi(\sF)(S)_x.$$
		The result follows.
	\end{proof}

\begin{corollary}\label{corollaryexact}
	A sequence $0\to \sF\to \sG\to \sH\to 0$ is exact in $\mathit{Sh}(X,\cond(\bZ))$ if and only if for any $S\in \ExtDisc$ the sequence $0\to \sF(-)(S)\to \sG(-)(S)\to \sH(-)(S)\to 0$ is an exact sequence of abelian sheaves on $X$ if and only if $0\to \sF_x\to \sG_x\to \sH_x\to 0$ is an exact sequence of condensed abelian groups for all $x\in X$. 
\end{corollary}

\begin{proof}
For any morphisms $\sF\stackrel{f}\to \sG\stackrel{g}\to \sH$ in $\mathit{Sh}(X,\cond(\bZ))$ such that $g\circ f=0$, there is a natural map $\eta: \mathrm{Im}(f)\to \ker(g)$. Recall that $\mathrm{Im}(f)$ is the sheafification of the presheaf image of $f$. By the description of sheafification in Lemma \ref{sheafificationofstalks},  $\eta$ is an isomorphism if and only if $\eta(-)(S): \mathrm{Im}(f(-)(S))\to \ker(g(-)(S))$ is an isomorphism for each $S$. The first equivalence now follows formally. For the second equivalence, we can evaluate on $S\in \ExtDisc$ and use stalkwise criterion for abelian sheaves and \eqref{stalk}.
\end{proof}

\begin{corollary}\label{corollarystalk}
	A map $f:\sF\to \sG$ in $\mathit{Sh}(X,\cond(\bZ))$ is an isomorphism if and only if for each point $x\in X$, the induced map on stalks $f_x: \sF_x\to \sG_x$ is an isomorphism.
\end{corollary}

\begin{proof}
	One direction is clear. Assume $f_x$ are isomorphisms. By Lemma \ref{lemmasurjectivecriterion}, we need to show $f(-)(S): \sF(-)(S)\to \sG(-)(S)$ is an isomorphism of abelian sheaves. By stalkwise criterion of abelian sheaves, we need to show the map 
	$$\colim_{U\ni x}\sF(U)(S)\to \colim_{U\ni x}\sG(U)(S)$$ is an isomorphism for all $x\in X$. Considering  \eqref{stalk}, this is our condition.
	\end{proof}
	
	If $f:X\to Y$ is a morphism of adic spaces, by \cite[\S 17]{kashiwara2006categories}, there is an adjoint pair $(f^{-1},f_*)$ between condensed abelian sheaves on $X$ and $Y$. We have $((f_*\sF)(V))(S)=f_*(\sF(-)(S))(V)$. It is formal to check $f^{-1}$ preserves stalks, thus it is an exact functor by corollary \ref{corollarystalk}. When $f$ is an open immersion, $f^{-1}$ has a left adjoint $f_!$ given by the extension by zero functor. That is, $f_!\sF$ is the sheafification of the presheaf
	$$U\to \begin{cases}
	\sF(f^{-1}(U)), & \text{if } U\subseteq X\\
	0 & \text{otherwise} 
	\end{cases}$$
	Note that if $j:U\to X$ is an open immersion and $\sF\in  \mathit{Sh}(U,\bZ_\square)$, the explicit formula above tells that $j_!\sF \in  \mathit{Sh}(X,\bZ_\square)$.
	
	 Assume $j:U\to X$ is an open immersion, $i:Z\to X$ is the closed complement and $\sF\in \mathit{Sh}(X,\cond(\bZ))$. By Corollary \ref{corollarystalk}, we have an exact sequence in $\mathit{Sh}(X,\cond(\bZ))$
	$$0\to j_!j^{-1}\sF\to \sF\to i_*i^{-1}\sF\to 0.$$
	
	We call $\sF\in \mathit{Sh}(X,\cond(\bZ))$ flasque if for any open subsets $U\subseteq V\subseteq X$, the map 
	$\sF(V)\to \sF(U)$ is an epimorphism of condensed abelian groups.
	
	\begin{proposition}\label{propositionacycofflasque}
		Assume $\sF\in \mathit{Sh}(X,\cond(\bZ))$ is flasque, then $\sF$ is $\Gamma(X,-)$-acyclic.
	\end{proposition}	
	
	\begin{proof}
		By abstract nonsense, it suffices to show injective sheaves are flasque and for an exact sequence $0\to \sF\to \sG\to \sH\to 0$ in $\mathit{Sh}(X,\cond(\bZ))$ with $\sF,\sG$ flasque, the sequence 
		$$0\to \sF(X)\to \sG(X)\to \sH(X)\to 0$$
		is exact in $\cond(\bZ)$. For $S\in \ExtDisc$, let $\bZ[S]_X\in \mathit{Sh}(X,\cond(\bZ))$ be the sheaf associated to the presheaf $U\to \bZ[S]\in \cond(\bZ)$. Then for any $\sF\in \mathit{Sh}(X,\cond(\bZ))$, $\Hom(\bZ[S]_X,\sF)=\Hom(\bZ[S],\sF(X))=\sF(X)(S)$. So, for an injective object $\sI$ in $\mathit{Sh}(X,\cond(\bZ))$, the first claim follows from applying $\Hom(-,\sI)$ to $0\to j_!j^{-1}\bZ[S]_X\to \bZ[S]_X$ for all $S\in \ExtDisc$. For the second claim, we need to show the sequence $0\to \sF(X)(S)\to \sG(X)(S)\to \sH(X)(S)\to 0$ is exact for all $S\in \ExtDisc$. By Corollary \ref{corollaryexact}, $0\to \sF(-)(S)\to \sG(-)(S)\to \sH(-)(S)\to 0$ is an exact sequence of abelian sheaves on $X$ and $\sF(-)(S)$ is flasque. We can apply the classical argument to conclude. 
	\end{proof}

\begin{proposition}\label{propositionmaincolimit}
Let $X$ be an adic space, $\{U_i\}_{i\in I}$ be a filtered diagram of open subsets of $X$ such that each $U_i$ is spectral and $Z=\cap_{i\in I} U_i$. Assume $\sF\in \mathit{Sh}(X,\cond(\bZ))$, then for each $q\geq 0$ the map 
$$\colim_{i\in I} H^q(U_i,\sF|_{U_i})\to H^q(Z,\sF|_Z)$$
is an isomorphism.
\end{proposition}	

\begin{proof}
The proof is similar to \cite[Proposition 3.1.19]{fujiwara2017foundationsrigidgeometryi} and the argument is standard, so we just sketch it:
For the first, one establishes statements \cite[Proposition 3.1.8; 3.1.10]{fujiwara2017foundationsrigidgeometryi} in the condensed setting. It is enough to check both isomorphisms after evaluating on $S\in \ExtDisc$, then one can use the same argument there.

For the second, we need the Godement resolution in condensed setting: For $\sF\in \mathit{Sh}(X,\cond(\bZ))$, let $\cG(\sF)$ be the sheaf $\cG(\sF)(U)=\prod_{x\in U} \sF_x$ and one defines a complex 
$$0\to \sF\stackrel{d^0}\to \cG^0(\sF)\stackrel{d^1}\to \cG^1(\sF)\stackrel{d^1}\to \cG^2(\sF)\to \cdots$$
inductively by setting $\cG^0(\sF):=\cG(\sF)$ and $\cG^n(\sF):= \cG(\coker(d^{n-1}))$. It is easy to check the functor $\cG(-)$ is exact, $\cG(\sF)$ is flasque and $d^0$ is injective. So the above complex is a functorial flasque resolution of $\sF$.

Let $p_i:Z\to U_i$ be the inclusion. If $\sF_i$ is a flasque sheaf on $U_i$ then $p_i^{-1}\sF_i$ is also flasque and $\colim_i p^{-1}_i \sF_i$ is quasi-flasque (See \cite[Exercise 0.3.1]{fujiwara2017foundationsrigidgeometryi}). Now we can use the same argument as in  \cite[Proposition 3.1.19]{fujiwara2017foundationsrigidgeometryi}.
\end{proof}

\begin{remark}
Note that by Remark \ref{remarktosolid}, Corollary \ref{corosheafpreservestalk}, \ref{corollaryexact}, \ref{corollarystalk}, Proposition \ref{propositionacycofflasque}, \ref{propositionmaincolimit} all hold when we work in $\mathit{(P)Sh}(X,\bZ_\square)$.
\end{remark}

\subsection{Sections with compact support}

We only define sections with compact support for partially proper morphisms: Let $f:X\to Y$ be a partially proper morphism of adic spaces.  Assume $\sF\in \mathit{Sh}(X,\bZ_\square)$, we define a subsheaf $f_!\sF$ of $f_*\sF$. For each $V\subseteq Y$ being open, $S\in \ExtDisc$ and $s\in f_*\sF(V)(S)$ being a section. Let
$$\supp(s):=\{x\in f^{-1}(V): 0\neq \mathrm{Im}(s)\in \sF_x(S)\}.$$
It is a closed subset of $f^{-1}(V)$ and $(\supp(s),f^{-1}V)$ is a germ.

\begin{definition}
	Let $f_!\sF\subseteq f_*\sF$ be the subsheaf (use lemma \ref{lemmasurjectivecriterion}) such that for each $S\in \ExtDisc$ 
	$$f_!\sF(V)(S):=\{s\in f_*\sF(V)(S): (\supp(s),f^{-1}(V))\to (V,V) \text{ is proper}\}  $$
	We also write $H^0(Y,f_!\sF)$ as $H^0_c(X/Y,\sF)$
\end{definition}

\begin{lemma}\label{partialsection}
	Assume $f:X\to Y$ is a partially proper morphism of adic spaces, $V\subseteq Y$ is a quasi-compact open subset and $s\in f_*\sF(V)$, then $f_s:(\supp(s),f^{-1}(V))\to (V,V)$ is proper if and only if $\supp(s)$ is quasi-compact. 
\end{lemma}

\begin{proof}
	One direction is clear. Assume $\supp(s)$ is quasi-compact, then $f_s$ is quasi-compact (Remark \ref{remarkunivclosed}). Since $f^{-1}(V)\to V$ is separated and $\supp(s)\subseteq f^{-1}(V)$ is closed, $f_s$ is separated. 
	
	To check universally closed: Since partially proper morphism between adic spaces and quasi-compact morphism between germs are stable under base change, it is enough to show $f_s':\supp(s)\to |V|$ is a closed map. We can prove it after passing $f$ to its associated diamond $f^\diamond:X^\diamond\to Y^\diamond$ and identifying $\supp(s)$ with a quasi-compact closed subset of $|X^\diamond|$ which we still denote by $\supp(s)$. By \cite[Proposition 18.10]{scholze2022etale} and its proof, $\supp(s)$ is contained in a quasicompact closed generalizing subsets $\overline{U}\subseteq |X^\diamond|$ which corresponds to a spatial sub $v$-sheaf $X'\subseteq X^\diamond$ such that $X'\to Y^\diamond$ is proper. In particular $\overline{U}\to |Y^\diamond|=|Y|$ is a closed map, so is $\supp(s)\to |Y|$ as $\supp(s)$ is closed in $\overline{U}$.
	\end{proof}
	
\begin{lemma}\label{coveringpartial}
	Let $f:X\to Y$ be a partially proper morphism between adic spaces where $Y$ is quasi-compact and quasi-separated. Then there exists a covering of $X$ by cofiltered family of quasi-compact open subsets $\{U_i\}_{i\in I}$ each of which is closed under specialization and has quasi-compact closure.  
\end{lemma}

\begin{proof}
	By the proof of \cite[Lemma 5.3.3]{huber2013etale}, the result holds when $|X|$ is taut and valuative (valuative means it is locally spectral and the set of generalisation of any point is totally ordered). Since passing to associated diamond preserves underlying topological spaces, the desired properties follow from \cite[Proposition 18.10, 11.19]{scholze2022etale}.
\end{proof}
	
\begin{lemma}\label{imagesolid}
 Assume $\sF\in \mathit{Sh}(X,\bZ_\square)$ and $f:X\to Y$ is a partially proper morphism of adic spaces. Then $f_!\sF\in \mathit{Sh}(X,\bZ_\square)$.
\end{lemma}

\begin{proof}
Since the category of solid abelian groups is closed under limit in the category of condensed abelian groups, it suffices to check $f_!\sF(V)$ is solid for each $V\subseteq Y$ quasi-compact quasi-separated open subset and we will assume $Y=V$ is quasi-compact quasi-separated. By Lemma \ref{coveringpartial}, there is a covering $\{U_i\}_{i\in I}$ of $X$ satisfying the condition in the cited lemma. Let $j_i: U_i\to X$ be the open immersion and $f_i: U_i\to Y$ be the restriction of $f$ to $U_i$. There is a natural injective map $f_{i*}j_{i!}\sF|_{U_i}\to f_*\sF$. Since $\overline{U}_i$ is quasi-compact, for each $S$ extremally totally disconnected and $s\in (f_{i*}j_{i!}\sF|_{U_i})(Y)(S)$, we have $\supp(s)$ is quasi-compact. By Lemma \ref{partialsection}, $s\in f_!\sF(Y)(S)$ thus we get a natural injective map $(f_{i*}j_{i!}\sF|_{U_i})\to f_!\sF$. For each $U_i\subseteq U_k$ there is a natural injective map $f_{i*}j_{i!}\sF|_{U_i}\to f_{k*}j_{k!}\sF|_{U_k}$. Thus there is an injective map $\colim_i f_{i*}j_{i!}\sF|_{U_i}\to f_!\sF$. It suffices to show this is an isomorphism because then (use $Y$ is quasi-compact in the second equality)
$$f_!\sF(Y)= (\colim_i f_{i*}j_{i!}\sF|_{U_i})(Y)=\colim_i (j_{i!}\sF|_{U_i})(U_i).$$
The right hand side is solid since $(j_{i!}\sF|_{U_i})(U_i)$ is solid (see the paragraph after Corollary \ref{corollarystalk}) and solid abelian groups are stable under colimit. To show the isomorphism, by Corollary \ref{corollaryexact} we can evaluate on $S$. Since $\cup_i U_i =X$, for each $s\in \sF(X)(S)$ such that $\supp(s)$ is quasi-compact, there is $U_i$ such that $\supp(s)\subseteq U_i$ which implies $s$ lies in the image of $j_{i!}\sF|_{U_i}(X)(S)$.
\end{proof}

\begin{lemma}\label{lemmalowershriekcolimit}
Let $f:X\to Y$ be a partially proper morphism of adic spaces.
\begin{enumerate}[(1)]
	\item $f_!$ is a left exact functor.
	\item If $f$ is an open immersion, then $f_!$ is extension by zero.
	\item $f_!$ commutes with filtered colimits.
	\item Assume $f:X\to Y; g:Y\to Z$ are partially proper morphisms, then $(g\circ f)_!\sF=g_!f_!\sF$ for any condensed abelian sheaf $\sF$.
\end{enumerate}
\end{lemma}

\begin{proof}
	 (1): It can be deduced easily from that $f_*$ is left exact, $f_!\sF\subseteq f_*\sF$ is a subsheaf and that if $\alpha: \sF\to \sG$ is an injective morphism of condensed abelian sheaves on $X$ and $s\in \sF(X)(S)$ then $\supp(s)\subseteq \supp(\alpha(s))$ (actually $\supp(s) = \supp(\alpha(s))$).

	 (2): There is an obvious functor from extension by zero sheaf to $f_!\sF$, then we can apply the same proof as in \cite[Lemma 4.2]{abe2020proper} using corollary \ref{corollarystalk}.
	 	
	 (3): By the universal property of colimit, there is a natural functor $ \colim_i (f_!\sF_i)\to f_!(\colim_i \sF_i)$. Note that by the proof of Lemma \ref{imagesolid}, we have
	 \begin{equation}\label{2}
	 f_!(\colim_i \sF_i)=\colim_U f_{U!}((\colim_i \sF_i)|_U)=\colim_U f_{U!}(\colim_i \sF_i|_U)
	 \end{equation}
	  where $U\subseteq X$ are quasi-compact open subspaces. For $V\subseteq Y$ quasi-compact open (for the equality below, use that any cover of $V$ has a finite refinement)
	$$(\colim_i f_!\sF_i)(V)\longrightarrow (\colim_i f_*\sF_i)(V)=\colim_i \sF_i(f^{-1}(V))$$
	contains objects which comes from some $\sF_n(f^{-1}(V))$ whose support is proper over $V$ (after evaluating on $S$). Evaluate \eqref{2} on $V$, it is easy to see $f_!(\colim_i \sF_i)(V)=(\colim_i f_!\sF_i)(V)$.
	
	(4): Let $h=g\circ f$. Since $h_*\sF=g_* f_*\sF$ and $h_!\sF$, $g_!f_!\sF$ are subsheaves, we can assume $Z$ is quasi-compact and quasi-separated. Properness of a morphism can be checked locally on the target, the result follows easily from Lemma \ref{partialsection} and Remark \ref{remarkunivclosed}.	
	\end{proof}
	
	For an adic space $X$, let $D^+(X,\bZ_\square)$ be the left bounded derived category of solid abelian sheaves on $X$. Note that it only depends on the underlying topological space $|X|$.

\begin{definition}\label{partialproperpushforward}
	Let $f:X\to Y$ be a partially proper morphism of adic spaces. Define 
	$$Rf_!: D^+(X,\bZ_\square)\to D^+(Y,\bZ_\square)$$
	to be the right derived functor of $f_!$ (Lemma \ref{imagesolid}). 
\end{definition}

\begin{lemma}\label{flasqueacyclic}
		Assume that $f:X\to Y$ is a partially proper morphism of adic spaces. Then flasque sheaves are $f_!$-acyclic.
	\end{lemma}

	\begin{proof}
		 The question is local so we assume $Y$ is quasi-compact and quasi-separated. There exists $\{U_i\}_{i\in I}$ as in Lemma \ref{coveringpartial} and we put $j_i:U_i\to X$. Then by the same argument as \cite[Lemma 4.8(2), 4.9]{abe2020proper} using Proposition \ref{propositionmaincolimit}, we get an isomorphism $R^qf_!\sF\simeq \colim_i R^qf_{i*}(j_{i!}\sF|_{U_i}))$ where $f_i:\overline{U_i}\to Y$ are morphisms of topological spaces which is the restriction of $f$ to $\overline{U_i}$. In fact, \cite[Lemma 4.9]{abe2020proper} shows that if $\sI$ is a flasque sheaf on $U_i$ then $R^qf_{i*}(j_{i!}\sI)=0$ for $q>0$. The result follows.
	\end{proof}

\begin{proposition}\label{prop:compositionofexceptionalpushforward}
	Let $X\stackrel{f}\longrightarrow Y\stackrel{g}\longrightarrow Z$ be partially proper morphisms of adic spaces. Then there is a canonical isomorphism $R(g\circ f)_!\cong Rg_!\circ Rf_!$ of functors $D^+(X,\bZ_\square)\to D^+(Z,\bZ_\square)$.
\end{proposition}

\begin{proof}
	Since $(g\circ f)_!=g_!\circ f_!$, it suffices to show that for an injective sheaf $\sI$ on $X$, $R^qg_!(f_!\sI)=0$ when $q>0$. The question is local on $Z$ so we assume $Z$ is quasi-compact and quasi-separated. We can find open cover $\{V_i\}_{i\in I}$ of $Y$ as in Lemma \ref{coveringpartial} and let $j_i: V_i\to \overline{V_i}$ and $g_i: \overline{V_i}\to Z$. By the proof of lemma \ref{flasqueacyclic}, we have
	$$R^qg_!(f_!\sI)=\colim_{i\in I} R^qg_{i*}(j_{i!}(f_!\sI)|_{V_i})$$
	and it suffices to show $(f_!\sI)|_{V_i}$ is flasque. So we can assume $V_i=Y$ is quasi-compact and quasi-separated. 
	
	Consider $W\subseteq Y$ be an open subset, it suffices to show $f_!\sI(Y)(S)\to f_!\sI(W)(S)$ is surjective for any $S\in \ExtDisc$. Pick a section $s \in f_!\sI(W)(S)$ with proper support over $W$. Let $T$ be the closure of $\supp(s)$ in $X$, and let $s' \in H^0(f^{-1}(W) \cup (X \setminus T),\sI)(S)$ be the unique section with $s'|_{f^{-1}(W)} = s$ and $s'|_{X\setminus T} = 0$. Since $\sI$ is flasque, we can pick $s'' \in \Gamma(X,\sI)(S)$ with $s''|_{f^{-1}(W)\cup X\setminus T} = s'$. Then $\supp(s'')\subseteq T$. Since $|X|$ is taut (by \cite[Proposition 18.10]{scholze2022etale}), $T$ is quasi-compact so $\supp(s'')$ is proper over $Y$. 
	\end{proof}
		
\subsection{Equivariant sheaves}\label{sec: appendixequivariantsheaves}
Let $\Gamma$ be a discrete group. Suppose that $\Gamma$ acts properly discontinuous on an adic space $X$ and let $f:X\rightarrow Y=X/\Gamma$ be the quotient map. Consider the Čech nerve
\[ (X^{\bullet/Y})_{\bullet}\rightarrow Y\]
where for $n\in\bbN$, $X^{n/Y}:=\underbrace{X\times_Y\cdots \times_YX}_{n}$ and let $f_n:X^{n/Y}\rightarrow  Y$.

Let $\sF\in D^+(Y,\bbZ_{\square})$ and set $\sG=f^{-1}\sF$ which is a $\Gamma$-equivariant sheaf (or a complex of $\Gamma$-equivariant sheaves) on $X$. For $\gamma\in \Gamma$ that induces $\gamma:X\rightarrow X$, the isomorphism $\sG\simeq \gamma_*\sG$ induces an action of $\gamma$ on $f_!\sG$ via $f_!\sG\simeq f_!\gamma_*\sG=f_!\sG$ by Lemma \ref{lemmalowershriekcolimit}. Locally on $Y$ where the $\Gamma$-covering $f$ is trivialized, it follows from the definition that $f_!\sG=Rf_!\sG$ can be identified with the $\Gamma$-module $\bbZ[\Gamma]\otimes_{\bbZ}\sF$. 
\begin{lemma}\label{lemmashriekdescent}
	Suppose that $\sF\in D^+(Y,\bbZ_{\square})$. We have natural maps $f_{n,!}f_n^{-1}\sF\rightarrow \sF$ and for all $n$. The natural map 
	\[\colim_{[n]\in\Delta^{\rm op}}f_{n,!}f_n^{-1}\sF\rightarrow \sF\]
	is an isomorphism, where $\Delta$ is the simplex category. Moreover $f_{n,!}f_n^{-1}\sF\simeq \bbZ[\Gamma^n]\otimes_{\bbZ[\Gamma]}f_!\sG$ and $\sF=\bbZ\otimes^L_{\bbZ[\Gamma]}f_!\sG$.
\end{lemma}
\begin{proof}
	Let $\sG=f^{-1}\sF$ which is a $\Gamma$-equivariant sheaf on $X$. Take $U\subset Y$ small enough such that $f^{-1}(U)\simeq U\times \Gamma$. Then $f_n|_{f_n^{-1}(U)}$ can be identified with the projection $U\times \Gamma^n\rightarrow U$. By Lemma \ref{lemmalowershriekcolimit}, $f_{n,!}$ commutes with direct sums. Since $f_n^{-1}\sF|_{f^{-1}(U)}=\oplus_{\gamma\in \Gamma^{n}}f_n^{-1}\sF|_{(U,\gamma)}$, we have a natural map $(f_{n,!}f_n^{-1}\sF)|_{U}=\oplus_{\gamma\in \Gamma^{n}}(f_{n}|_{(U,\gamma)})_{!}f_{n}^{-1}\sF|_{(U,\gamma)}=\oplus_{\gamma\in \Gamma^{n}}\sF|_{U}\rightarrow \sF|_U$. Let $U$ ranges over basis of $Y$ we see there exists a natural map $f_{n,!}f_n^{-1}\sF\rightarrow \sF$ for any $n$.
	
	Let $d_{n,i}:X^{(n+1)/Y}\rightarrow X^{n/Y}$ be the $i$-th face map for $i=0,\cdots,n$. Similarly as above, there exists a natural map $ d_{n,i,!}d_{n,i}^{-1}f_{n}^{-1}\sF\rightarrow f_{n}^{-1}\sF$. Applying $f_{n,!}$, we obtain a map $d_{n,i}:f_{n+1,!}f_{n+1}^{-1}\sF\rightarrow f_{n,!}f_{n}^{-1}\sF$ and let
	\[ \delta_{n}=\sum_{i=0}^n(-1)^id_{n,i}: f_{n+1,!}f_{n+1}^{-1}\sF\rightarrow f_{n,!}f_{n}^{-1}\sF\]
    be the alternating sum of the face maps. We need to check that the map $\colim_{[n]}f_{n,!}f_{n}^{-1}\sF\rightarrow \sF$ is an isomorphism. The problem is local and we may assume that there is a section $Y\hookrightarrow X$ such that $\Gamma\times Y=X:(\gamma_{n-1},y)\mapsto \gamma_{n-1}y$. 
	Then
	\begin{equation}\label{equationGammatorsor}
		\Gamma^{n}\times Y\simeq X^{n/Y}:(\gamma_0,\cdots,\gamma_{n-1},y)\mapsto (\gamma_0y,\cdots,\gamma_{n-1}y),
   \end{equation}  
	and $f_{n,!}f_{n}^{-1}\sF=\bbZ[\Gamma^n]\otimes_{\bbZ}\sF: s\mapsto \gamma\otimes s|_{\gamma^{-1}U}$ for $s\in (f_{n,!}f_{n}^{-1}\sF)(U),\gamma\in\Gamma^n$.
	The map $\delta_n$ can be identified with 
	\[\bbZ[\Gamma^{n}]\otimes_{\bbZ}\sF\rightarrow \bbZ[\Gamma^{n-1}]\otimes_{\bbZ}\sF: [(\gamma_0,\cdots, \gamma_{n-1})]\otimes s\mapsto\sum_{i=0}^{n-1}(-1)^i [(\gamma_0,\cdots,\widehat{\gamma}_{i},\cdots,\gamma_{n-1})]\otimes s.\]
	Hence $\colim_{[n]}f_{n,!}f_{n}^{-1}\sF=[\cdots\rightarrow\bbZ[\Gamma^n]\rightarrow \cdots\rightarrow \bbZ[\Gamma]]\otimes^L_{\bbZ}\sF=\sF$ since the complex $\cdots\rightarrow\bbZ[\Gamma^n]\rightarrow \cdots\rightarrow \bbZ[\Gamma]\rightarrow \bbZ$
    is a standard resolution of $\bbZ$. Note that as in (\ref{equationGammatorsor}), we have
	\[\Gamma^{n-1}\times X\simeq X^{n/Y}:(\gamma_0,\cdots,\gamma_{n-2},x)\mapsto (\gamma_0x,\cdots,\gamma_{n-2}x).\]
	We see $f_{n,!}f_n^{-1}\sF=\bbZ[\Gamma^n]\otimes_{\bbZ[\Gamma]}f_!\sG$.
\end{proof} 
\begin{proposition}\label{prop:Gammaequivariantpushforward}
	Let $g:Y\rightarrow *$ be the structure map. Let $\sF\in D^+(Y,\bbZ_{\square})$. Suppose that $\Gamma$ has finite cohomological dimension (e.g. \cite[\S 6.1]{borel1976cohomologie}). Then $Rg_!\sF=\bbZ\otimes^L_{\bbZ[\Gamma]}Rg_!(f_!f^{-1}\sF).$
\end{proposition}
\begin{proof}
	By the local description of $f_!f^{-1}\sF$, we see $f_!f^{-1}\sF\in D^+(Y,\bbZ_{\square})$. By Lemma \ref{lemmashriekdescent}, $\sF=\bbZ\otimes^L_{\bbZ[\Gamma]}f_!\sG$. We need to show that $Rg_!(\bbZ\otimes^L_{\bbZ[\Gamma]}f_!f^{-1}\sF)=\bbZ\otimes^L_{\bbZ[\Gamma]}Rg_!(f_!f^{-1}\sF)$. If $\Gamma$ has finite cohomological dimension $d$, there exists a finite projective resolution 
	\[0\rightarrow P^{-d}\rightarrow\cdots\rightarrow P^{-1}\rightarrow P_0\rightarrow \bbZ\rightarrow 0\]
	of the trivial $\bbZ[\Gamma]$-module $\bbZ$. Hence $Rg_!(\bbZ\otimes^L_{\bbZ[\Gamma]}f_!f^{-1}\sF)=Rg_!(P^{\bullet}\otimes^L_{\bbZ[\Gamma]}f_!f^{-1}\sF)=P^{\bullet}\otimes^L_{\bbZ[\Gamma]}Rg_!(f_!f^{-1}\sF)$ since the derived functor $Rg_!$ preserves finite colimits. (One can replace $f_!f^{-1}\sF$ by an injective resolution $\sI^{\bullet}$ of sheaves of condensed $\bbZ[\Gamma]$-modules. Then $Rg_!(\bbZ\otimes^L_{\bbZ[\Gamma]}f_!f^{-1}\sF)$ is calculated as the $g_!\mathrm{Tot}(P^{\bullet}\otimes_{\bbZ[\Gamma]}\sI^{\bullet})=\mathrm{Tot}(P^{\bullet}\otimes_{\bbZ[\Gamma]}g_!\sI^{\bullet})$.)
\end{proof}
\bibliographystyle{amsalpha}
\bibliography{Panquaternion.bib}
\end{document}